\documentclass[10pt]{amsart}

\usepackage{amsfonts}
\usepackage{mathrsfs}
\usepackage{amsmath}
\usepackage{amsthm}
\usepackage{amssymb}
\usepackage{latexsym}
\usepackage[all]{xy}
\usepackage{verbatim}
\usepackage{txfonts}
\usepackage{graphicx}
\usepackage{amscd,amsbsy}
\usepackage{nccmath}
\usepackage{faktor}
\usepackage{comment}

\usepackage{etoolbox}

\usepackage[]{hyperref}

\newtheorem{theorem}{Theorem}[section]
\newtheorem{proposition}[theorem]{Proposition}
\newtheorem{corollary}[theorem]{Corollary}
\newtheorem{lemma}[theorem]{Lemma}
\newtheorem{definition}[theorem]{Definition}
\newtheorem{remark}[theorem]{Remark}

\newcommand{\nc}{\newcommand}

\nc{\tl}{\tilde}
\nc{\wt}{\widetilde}
\nc{\wh}{\widehat}
\nc{\mc}{\mathcal}
\nc{\mf}{\mathfrak}
\nc{\ms}{\mathsf}
\nc{\mr}{\mathrm}

\newcommand{\al}{\alpha}
\newcommand{\be}{\beta}
\newcommand{\del}{\delta}
\newcommand{\eps}{\epsilon}
\newcommand{\veps}{\varepsilon}

\newcommand{\ka}{\kappa}
\newcommand{\la}{\lambda}
\newcommand{\om}{\omega}
\newcommand{\si}{\sigma}

\newcommand{\mcA}{\mathcal{A}}
\newcommand{\mcB}{\mathcal{B}}

\newcommand{\mcI}{\mathcal{I}}

\newcommand{\mcR}{\mathcal{R}}

\newcommand{\mfg}{\mf{g}}
\newcommand{\mfgl}{\mf{g}\mf{l}}
\newcommand{\mfh}{\mf{h}}

\newcommand{\mfI}{\mf{I}}

\newcommand{\mfS}{\mf{S}}
\newcommand{\mfsl}{\mf{s}\mf{l}}
\newcommand{\mfso}{\mf{s}\mf{o}}
\newcommand{\mfsp}{\mf{s}\mf{p}}

\newcommand{\mfU}{\mf{U}}

\newcommand{\g}{\mf{g}}

\newcommand{\msJ}{\mathsf{J}}

\newcommand{\msR}{\mathsf{R}}

\newcommand{\msX}{\mathsf{X}}

\newcommand{\End}{\mr{End}}

\newcommand{\Hom}{\mr{Hom}}

\newcommand{\Ker}{\mr{Ker}}

\newcommand{\iso}{\stackrel{\sim}{\longrightarrow}}

\newcommand{\lra}{\longrightarrow}

\newcommand{\into}{\hookrightarrow}
\newcommand{\onto}{\twoheadrightarrow}

\newcommand{\C}{\mathbb{C}}
\newcommand{\Z}{\mathbb{Z}}

\newcommand{\ot}{\otimes}
\newcommand{\ol}{\overline}

\nc{\qu}{\quad}
\nc{\qq}{\qquad}

\newcommand {\Omit}[1]{}

\newcommand{\N}{\mathbb{N}}

\nc{\cur}{{\rm cr}}

\DeclareMathOperator{\ad}{ad}

\DeclareMathOperator{\id}{id}
\DeclareMathOperator{\gr}{gr}

\nc{\msum}{\medop\sum}

\numberwithin{equation}{section}

\renewcommand{\,}{\kern 0.1em} 

\nc{\spl}[1]{\begin{equation}\begin{aligned}#1\end{aligned}\end{equation}}
\nc{\eqa}[1]{\begin{align}#1\end{align}}
\nc{\eqn}[1]{\begin{align*}#1\end{align*}}
\nc{\eg}[1]{\begin{gather}#1\end{gather}}
\nc{\egn}[1]{\begin{gather*}#1\end{gather*}}

\nc{\ali}[1]{\begin{alignat}{50}#1\end{alignat}}
\nc{\aln}[1]{\begin{alignat*}{50}#1\end{alignat*}}

\usepackage{color}
\usepackage[usenames,dvipsnames]{xcolor}
\nc{\red}{\color{red}}
\nc{\blu}{\color{blue}}
\nc{\br}{\color{Brown}}
\nc{\gre}{\color{green!50!black}}

\begin{document}

\begin{center}
{\Large{\textbf{Equivalences between three presentations of orthogonal and symplectic 
Yangians}}} 

\bigskip

Nicolas Guay$^{1a}$, Vidas Regelskis$^{2b}$, Curtis Wendlandt$^{1c}$

\end{center}

\bigskip

\begin{center} \small 
$^1$ University of Alberta, Department of Mathematics, CAB 632, Edmonton, AB T6G 2G1, Canada.\\ 
$^2$ University of York, Department of Mathematics, York, YO10 5DD,  UK. \\
\smallskip
E-mail: $^a$\,nguay@ualberta.ca $^b$\,vidas.regelskis@york.ac.uk $^c$\,cwendlan@ualberta.ca
\end{center}

\patchcmd{\abstract}{\normalsize}{}{}{}

\begin{abstract} \small 
We prove the equivalence of two presentations of the Yangian $Y(\mfg)$ of a simple Lie algebra $\mfg$ and we also show the equivalence with a third presentation when $\mfg$ is either an orthogonal or a symplectic Lie algebra. As an application, we obtain an explicit correspondence between two versions of the classification theorem of finite-dimensional irreducible modules for orthogonal and symplectic Yangians. 
\end{abstract}

\makeatletter
\@setabstract
\makeatother

\medskip

{
\setlength{\parskip}{0ex}
\tableofcontents
}

\thispagestyle{empty}


\section{Introduction}

Many mathematical objects can be described in apparently different, but actually equivalent, ways. Different ways lead to different points of view, some of which being better than others depending on what one is interested. This is true for algebraic structures given by generators and relations: different sets of generators and relations can lead to isomorphic structures, and which set is the most appropriate depends on the use made of that structure. In this paper, we consider three different presentations for the Yangian of a complex simple Lie algebra $\mfg$. Yangians form one  of the two important families of quantized enveloping algebras of affine type, the second being the quantum affine algebras. The original presentation of Yangians given in the work of V. Drinfeld (see \cite{Dr1}) is convenient for quantizing the standard Lie bialgebra structure on the polynomial current algebra of a semisimple Lie algebra $\g$ and it is commonly used in the work of theoretical physicists (see, for instance, \cite{Ber,Lo,Ma}). We will call it the $J$-presentation: see Definition \ref{D:Y(g)-DI}. It is not convenient, however, for the study of representations of the Yangian of $\mfg$, which is what motivated V. Drinfeld to obtain another presentation which is better suited for this purpose \cite{Dr3}; we will call this one the current presentation, see Definition \ref{D:Ycr(g)-}. There is a third set of generators and relations for Yangians which has its origins in the quantum inverse scattering method of theoretical physics and leads to the so-called $RTT$-presentation: it is sometimes also called the $R$-matrix or FRT-presentation, see \cite{FRT} and Definition \ref{D:Y(g)-RTT}. 

Those three different presentations of Yangians are all equivalent. The equivalence between the $J$ and $RTT$-presentation was stated in Theorem 6 of \cite{Dr1} and the equivalence between the $J$ and current presentation is the content of Theorem 1 in \cite{Dr3}. When $\mfg=\mfsl_n$, formulas for an explicit isomorphism between the current and $RTT$-presentations are given in \cite{Dr3}. It is possible to interpolate between the current and $RTT$-presentations for the Yangian of $\mfgl_n$: this was accomplished in \cite{BrKl} where the authors obtained so-called parabolic presentations of that Yangian depending on a partition of $n$, one extreme case being the $RTT$-presentation, the other extreme case being the current presentation. As a consequence, an isomorphism between the $RTT$ and current presentations of the Yangian of $\mfsl_n$ is obtained; it is in agreement with the formulas provided in \cite{Dr3}: see Remarks 5.12 and 8.8 of \cite{BrKl}. The results of \cite{BrKl} are also explained in Section 3.1 of \cite{Mo3}.

The formulas for the equivalence between the $RTT$ and current presentation in \cite{BrKl} came from the Gauss decomposition of the matrix of generators in the $RTT$-presentation. Very recently, this approach has been successfully extended to Yangians of orthogonal and symplectic Lie algebras: see \cite{JLM}. (The $\mfso_3$-case was treated previously in \cite{JL}.)

In this paper, we give a proof of Theorem 6 in \cite{Dr1} for orthogonal and symplectic Lie algebras (see Theorem~\ref{T:YR-iso}) and a proof of Theorem 1 in \cite{Dr3} for any $\mfg$ (see Theorem~\ref{T:Ycr(g)-}). These are two of the three main contributions of this paper. The $RTT$-presentation for symplectic and orthogonal Yangians was first treated in \cite{AACFR,AMR}, and later received more attention in the mathematical literature in such papers as \cite{Mo4, MM1, MM2,GR,GRW2}. The initial motivation for this paper came from a desire to better understand why this presentation of the Yangian is equivalent to the others. This led us to consider the analogous question regarding the $J$ and current presentations. Another motivation came from representation theory. It has been known since \cite{Dr3} that finite-dimensional irreducible representations of Yangians are classified by certain polynomials, usually called Drinfeld polynomials in the literature: to prove this classification result, Drinfeld used the current presentation. When $\mfg=\mfsl_n$, such a classification theorem can also be proved using the $RTT$-presentation (see, for instance, Corollary 3.4.8 in \cite{Mo3}) and the resulting classification is also given in terms of certain polynomials. This raises the question of how these two families of polynomials are related: the answer is provided in the proof of Corollary 3.4.9 in \textit{loc. cit.} and uses the Gauss decomposition. When $\mfg=\mfso_N$ or $\mfg = \mfsp_N$, the classification theorem was reproved in \cite{AMR} using the $RTT$-presentation: see Corollary 5.19 in \textit{loc.~cit.} It was explained by the authors of \cite{AMR} that, without an explicit isomorphism between the $RTT$ and current presentations available, it was not clear how to translate Drinfeld's classification result into one compatible with the $RTT$-presentation of the Yangian. However, they were able to use techniques available in the $R$-matrix setting to provide a separate proof which again yields a parameterization in terms of certain tuples of monic polynomials.  This left open the question of how the polynomials in Drinfeld's classification are related to those in \cite{AMR}: in Theorem \ref{T:cr-R}, we are able to use the isomorphism \eqref{J->R} to provide an answer to this question. This is the third main result of this paper.  It should also be possible to use the isomorphism obtained very recently in \cite{JLM} via the Gauss decomposition to answer that question, as was done when $\mfg=\mfsl_n$ in the proof of Corollary 3.4.9 in \cite{Mo3}.

The results of Section \ref{sec:Jcur} are valid for any semisimple Lie algebra $\mfg$. After recalling the $J$ and current presentations of the Yangian in Sections \ref{sec:Jpres} and \ref{sec:curpres}, we prove in Section \ref{sec:equivJcur} that these are equivalent: see Theorem \ref{T:Ycr(g)-}. (The $\mfsl_2$-case is proved separately in the Appendix.) Along the way, we show that the associated graded ring of the Yangian as given by the $J$-presentation is isomorphic to the enveloping algebra of the polynomial current algebra of $\mfg$: see Proposition \ref{grYzeta}. This result has been known for a long time for the current presentation \cite{Le1} and the $RTT$-presentation \cite{Mo3,AMR}, but Proposition \ref{grYzeta} cannot be deduced from these without first proving the equivalence between these presentations and our approach requires this proposition to prove Theorem \ref{T:Ycr(g)-}. Of course, Proposition \ref{grYzeta} is also already known, but not only has its proof never been published, it seems to have not been explicitly stated anywhere in the literature as far as we know. A major part of the argument in the proof of Theorem \ref{T:Ycr(g)-} relies on Theorem \ref{smallerset}, which is the main result of \cite{Le2} and tells us precisely how to reduce the number of generators and relations in the current presentation. 

In the next section, we prove the equivalence of the $J$ and $RTT$-presentation when $\mfg$ is an orthogonal or symplectic Lie algebra - see Theorem \ref{T:YR-iso}. We start with the Yangian given by the $J$-presentation and show how to extend certain representations of $\mfg$ to the Yangian (e.g.~the natural representation on $\C^N$), thereby providing a proof of Theorem~7 in \cite{Dr2} in types B, C and D. (Actually, we fill in some details in the proof of Proposition 12.1.17 in \cite{ChPr2} which the authors had decided not to include.) In order to construct the isomorphism between the $J$ and $RTT$ presentations as given in Theorem \ref{T:YR-iso}, we needed to know that the $R$-matrix \eqref{R(u)} could be obtained by evaluating the universal $R$-matrix on the natural representation: this is the content of Proposition \ref{C:RmcR}, which is a special case of Theorem \ref{RmcR}, which in turn is Theorem 4 in \cite{Dr1}. Whereas Proposition \ref{C:RmcR} is valid when $\mfg$ is an orthogonal or symplectic Lie algebra, Theorem \ref{RmcR} is valid for any simple Lie algebra $\mfg$. In Subsection \ref{indapp}, a more direct and elementary proof, bypassing the  use of the universal $R$-matrix, is provided of the isomorphism between the $J$ and $RTT$ presentations: see Proposition \ref{T:J->R}. This makes use of \cite{Le2} via Proposition \ref{L:J->R}. 

In the last section, we prove Theorem \ref{T:cr-R} which establishes a correspondence between the two families of Drinfeld polynomials originating from the two equivalent classifications of finite-dimensional irreducible representations of the Yangian. Furthermore, we explain how a highest weight vector remains a highest weight vector after pulling back the action of the Yangian via the isomorphism \eqref{J->R}.  

Yangians and quantum affine algebras are known to be intimately connected: see for instance \cite{Dr2, GuMa, GTL1,GTL2, GTL3}.  It should thus be possible to obtain analogous results for quantum affine algebras: for instance, a complete proof of the equivalence of the original definition (see Section 4 in \cite{Dr1}) with Drinfeld's loop realization (given in \cite{Dr3}) was provided recently in \cite{Da1,Da2}. 

{\it Acknowledgements.} We thank V. Toledano Laredo for some helpful pointers and H. Nakajima for indicating that the results in Section A.8 of \cite{GTL1} can be used to prove Proposition \ref{grYzeta} below. The first and third named authors gratefully acknowledge the financial support of the Natural Sciences and Engineering Research Council of Canada provided via the Discovery Grant Program and the Alexander Graham Bell Canada Graduate Scholarships - Doctoral Program, respectively. Part of this work was done during the second named author's visits to the University of Alberta; he thanks the University of Alberta for the hospitality. The second named author was supported in part by the Engineering and Physical Sciences Research Council of the United Kingdom, grant number EP/K031805/1; he gratefully acknowledges the financial support.


\section{\texorpdfstring{$J$}{J} and current presentations of the Yangian}\label{sec:Jcur}


\subsection{\texorpdfstring{$J$}{J}-presentation} \label{sec:Jpres}

As our starting point we recall the definition of the Yangian as it was first introduced by Drinfeld in \cite{Dr1} and is now conveniently called Drinfeld's first or $J$-presentation of the Yangian. Let $\mfg$ be a complex simple Lie algebra and $(\cdot,\cdot)$ be an invariant, non-degenerate symmetric bilinear form on $\mfg$. We set $\mfg[s] = \mfg\ot \C[s]$, the polynomial current Lie algebra of $\mfg$. We will assume $\hbar$ to be a formal variable. Tensor products $\ot$ and spaces of homomorphisms (including $\Hom$ and $\End$) will be over $\C$, unless specified otherwise. Lastly, $\mfS_m$ will denote the symmetric group on the set $\{1,\ldots,m\}$.

\begin{definition} \cite{Dr1}  \label{D:Y(g)-DI} 
Let $\{ X_{\la} \}_{\la\in\Lambda}$ be an orthonormal basis of $\mfg$. The Yangian $Y_\hbar(\mfg)$ is the unital associative $\C[\hbar]$-algebra generated by elements $X$ and $J(X)$ $\forall\, X \in \mfg$ with the following defining relations:
\eqa{
& X_1 X_2 - X_2 X_1 = [X_1,X_2] \in \mfg , \qu [X_1,J(X_2)] = J([X_1,X_2]), \label{J0} \\
& J(aX_1+bX_2) = aJ(X_1) + bJ(X_2), \label{J1} \\
& [J(X_1),[J(X_2),X_3]] - [X_1,[J(X_2),J(X_3)]] = \hbar^2 \msum_{\la,\mu,\nu} \al_{\la\mu\nu} \{ X_\la, X_\mu, X_\nu \}, \label{J2}
}
where $\al_{\la\mu\nu} = ([X_1,X_{\la}],[[X_2,X_{\mu}],[X_3,X_{\nu}]])$ and $\{x_1,x_2,x_3\} = \tfrac1{24} \sum_{\si\in\mfS_3} x_{\si(1)}x_{\si(2)}x_{\si(3)}$. 
When $\mfg=\mfsl_2$, relation \eqref{J2} should be replaced by the following one: 
\eq{ 
[[J(X_1),J(X_2)],[X_3,J(X_4)]] + [[J(X_3),J(X_4)],[X_1,J(X_2)]] = \hbar^2 \msum_{\la,\mu,\nu} \beta_{\la\mu\nu}  \{ X_\la, X_\mu, J(X_\nu) \}, \label{J3}
}
where $\beta_{\la\mu\nu} = ([X_1,X_{\la}],[[X_2,X_{\mu}],[[X_3,X_4],X_{\nu}]]) + ([X_3,X_{\la}],[[X_4,X_{\mu}],[[X_1,X_2],X_{\nu}]])$. The Hopf algebra structure of $Y_\hbar(\mfg)$ is given by: 
\eqa{
& \Delta(X) = X \ot_{\C[\hbar]} 1 + 1 \ot_{\C[\hbar]} X, \qu
\Delta(J(X)) = J(X) \ot_{\C[\hbar]} 1 + 1 \ot_{\C[\hbar]} J(X) +  \tfrac12 \hbar [X \ot_{\C[\hbar]} 1, \Omega], \notag \\
& S(X)=-X, \qu S(J(X))=-J(X) + \tfrac14 \hbar \,\mathsf{c}\, X, \qu \veps(X)=\veps(J(X))=0, \notag
}
where $\mathsf{c}$ is the eigenvalue of the quadratic Casimir element $\Omega$ in the adjoint representation of $\mfg$. 
\end{definition}

The Yangian $Y_\hbar(\mfg)$ enjoys the following properties: 
\begin{itemize}
\item  It is a free graded $\C[\hbar]$-algebra, with grading induced by the degree assignment $\deg X=0$ and $\deg J(X)=1$ for all $X\in \mfg$, together with $\deg \hbar=1$; 

\item It has a one-parameter group of Hopf algebra automorphisms $\tau_z$ given by
\begin{equation}
\tau_z(X) = X \;\text{ and }\; \tau_z(J(X)) =  J(X) + z\hbar X \;\text{ for all }\; X\in\mfg \;\text{ and }\; z\in\C. \label{isotau}
\end{equation}

\item  Let $\mf{Ug}[s]$ denote the universal enveloping algebra of $\mfg[s]$. There is a natural isomorphism of Hopf algebras $\mf{Ug}[s]\to Y_\hbar(\mfg)/\hbar Y_\hbar(\mfg)$ which sends $\mfg$ identically to its image in $Y_\hbar(\mfg)/\hbar Y_\hbar(\mfg)$, and $X\ot s$ to $J_0(X)$ for each $X\in \mfg$, where $J_0(X)$ is the image of $J(X)$ in the quotient $Y_\hbar(\mfg)/\hbar Y_\hbar(\mfg)$. (See \cite{Dr1} and Proposition 3.1 in \cite{BeRe}.)
\end{itemize}

Let $\zeta\in\C$ and denote by $Y_{\zeta}(\mfg)$ the Hopf algebra $Y_\hbar(\mfg) / (\hbar-\zeta)Y_\hbar(\mfg)$. We will denote the generators of $Y_\zeta(\mfg)$ again by $J(X)$ and $X$, except in the rare instances where additional emphasis is required on the parameter $\zeta$. In such a case, we shall replace $J(X)$ by $J_\zeta(X)$.
The grading on $Y_\hbar(\mfg)$ induces a filtrating on $Y_{\zeta}(\mfg)$ defined as follows: set $\deg J(X)=1$ and $\deg X=0$ for all $X\in \mfg$, and let $F_k Y_{\zeta}(\mfg)$ denote the subspace of $Y_{\zeta}(\mfg)$ spanned by monomials of degree less than or equal to $k$. Note that for any two nonzero complex numbers $\zeta_1,\zeta_2\in \C^{\times}$, the assignment 
\begin{equation}
 X\mapsto X, \quad J_{\zeta_1}(X)\mapsto \tfrac{\zeta_1}{\zeta_2} J_{\zeta_2}(X) \; \text{ for all }X\in \mfg \label{z1z2iso}
\end{equation}
extends to an isomorphism of filtered Hopf algebras.

\begin{proposition}\label{grYzeta}
Let $\gr Y_{\zeta}(\mfg)$ denote the associated graded algebra of $Y_{\zeta}(\mfg)$ with respect to the filtration $\{F_k Y_{\zeta}(\mfg)\}$. Then $\gr Y_{\zeta}(\mfg)$
 is isomorphic to $\mf{Ug}[s]$. 
\end{proposition}
Although this result is certainly well known (see \cite{Le1} for instance), a proof has never appeared in the literature for the Yangian in Drinfeld's $J$-presentation. Since this result is significant and will be used throughout our paper, we include a complete proof. It cannot be deduced directly from the main result in \cite{Le1} because Proposition \ref{grYzeta} is needed to obtain an isomorphism between Drinfeld's $J$-presentation and the one used in \cite{Le1} (see Theorem \ref{T:Ycr(g)-} below).

\begin{proof}
The argument which we present was observed by H. Nakajima in an unpublished note related to \cite{GNW}. It closely follows the techniques of \cite[A.7-A.9]{GTL1}, which have been attributed to Drinfeld. 

This is immediate for $\zeta=0$, so let us assume that $\zeta\neq 0$. Since $Y_{\zeta_1}(\mfg)\cong Y_{\zeta_2}(\mfg)$ as filtered algebras for any $\zeta_1,\zeta_2\in \C^\times$, it is enough to prove the proposition for the Yangian $Y(\mfg)=Y_{\zeta=1}(\mfg)$. 
In order to distinguish between generators of $Y_\hbar(\mfg)$ and generators of $Y(\mfg)$, which both play a role in our argument,  we will henceforth denote those of the latter algebra by $\msX$ and $\msJ(\msX)$, for all $\msX\in \mfg$.

Let $\mr{Rees}\,Y(\mfg)$ denote the Rees algebra corresponding to the filtrating $\{F_k Y(\mfg)\}_{k=0}^{\infty}$:
\begin{equation*}
 \mr{Rees}\,Y(\mfg)=\bigoplus_{k=0}^\infty \hbar^k F_k Y(\mfg) \subset Y(\mfg)[\hbar].
\end{equation*}
We have $\mr{Rees}\,Y(\mfg)|_{\hbar=0}\cong \gr Y(\g)$ and $\mr{Rees}\,Y(\mfg)|_{\hbar=1}\cong Y(\mfg)$. 

There is a natural homomorphism $\Phi_\hbar$ of graded $\C[\hbar]$-algebras given by 
\begin{equation}
 \Phi_\hbar: Y_\hbar(\mfg)\to \mr{Rees}\,Y(\mfg),\quad \Phi_\hbar(X)=\msX, \; \Phi_\hbar(J(X))=\hbar\msJ(\msX) \; \text{ for all }\; X\in \mfg. \label{PhiRees}
\end{equation}
Since $\mr{Rees}\,Y(\mfg)$ is generated by $\{\msX,\hbar\msJ(\msX)\}_{\msX\in \mfg}$, $\Phi_\hbar$ is surjective. Our goal is to show that $\Phi_\hbar$ 
is in fact an isomorphism, which will imply the proposition. 

By the same argument as given in the proof of Proposition A.8 of \cite{GTL1}, to show that $\Ker\,\Phi_\hbar$ is trivial, 
it is enough to show that the quotient $K=\Ker\,\Phi_\hbar/\hbar\Ker\,\Phi_\hbar$ is zero. In addition, $K$ embeds into $Y_\hbar(\mfg)/\hbar Y_\hbar(\mfg)$ ($=\mf{Ug}[s]$). For completeness, we take a moment to recall why these two results are true.  If $K$ is trivial, then 
$\Ker\,\Phi_\hbar=\hbar^k\Ker\,\Phi_\hbar$ for all $k\in \Z_{\geq 0}$, hence $\Ker\,\Phi_\hbar=\cap_{k\geq 0}\hbar^k \Ker\,\Phi_\hbar\subset \cap_{k\geq 0}\hbar^k Y_\hbar(\mfg)=\{0\}$, since $Y_\hbar(\mfg)$ is separated. To show $K$ embeds into $Y_\hbar(\mfg)/\hbar Y_\hbar(\mfg)$, note that if $\hbar x \in \hbar Y_\hbar(\mfg)\cap \Ker\,\Phi_\hbar$ then $\hbar\Phi_h(x)=0$. In $\mr{Rees}\,Y(\mfg)\subset Y(\mfg)[\hbar]$ this implies that $\Phi_h(x)=0$; therefore $h\Ker\,\Phi_\hbar=\hbar Y_\hbar(\mfg)\cap \Ker\,\Phi_\hbar$, so $K$ is contained in $Y_\hbar(\mfg)/\hbar Y_\hbar(\mfg)$. 

$\mr{Rees}\,Y(\mfg)$ is equipped with a natural $\C[\hbar]$-bialgebra structure, in particular it has a coproduct also denoted $\Delta$ such that $(\Phi_\hbar\ot_{\C[\hbar]}\Phi_\hbar)\circ \Delta = \Delta \circ \Phi_\hbar$. Hence, $\Ker\,\Phi_\hbar$ is a coideal subalgebra of $Y_\hbar(\mfg)$: 
\begin{equation*}
 \Delta(\Ker\,\Phi_\hbar)\subset Y_\hbar(\mfg)\ot_{\C[\hbar]} \Ker\,\Phi_\hbar + \Ker\,\Phi_\hbar\ot_{\C[\hbar]}  Y_\hbar(\mfg). 
\end{equation*}
This implies that $K$ is a co-Poisson Hopf ideal of $\mf{Ug}[s]$. By \cite[Corollary A.9]{GTL1}, the only co-Poisson Hopf ideals of  $\mf{Ug}[s]$ are 
$\{0\}$ and $\mf{Ug}[s]$. Suppose that  $K=\mf{Ug}[s]$ and consider the composition $\mr{P}_\hbar\circ\Phi_\hbar: Y_\hbar(\mfg)\onto \mr{Rees}\,Y(\mfg)|_{\hbar=0}$, where $\mr{P}_\hbar$ is the natural quotient map $\mr{Rees}\,Y(\mfg)\onto\mr{Rees}\,Y(\mfg)|_{\hbar=0}$. Since the two-sided ideal $\hbar Y_\hbar(\mfg)+\Ker\,\Phi_\hbar$ is contained in the kernel of $\mr{P}_\hbar\circ\Phi_\hbar$, we obtain a surjective homomorphism 
$Y_\hbar(\mfg)/(\hbar Y_\hbar(\mfg)+\Ker\,\Phi_\hbar)\onto \mr{Rees}\,Y(\mfg)|_{\hbar=0}$. However, $K=\Ker\,\Phi_\hbar/\hbar \Ker\,\Phi_\hbar \cong \mf{Ug}[s]$ implies that 
\[ 
\faktor{ Y_\hbar(\mfg)}{(\hbar  Y_\hbar(\mfg)+\Ker\,\Phi_\hbar)}\cong \left(\faktor{ Y_\hbar(\mfg)}{\hbar Y_\hbar(\mfg)}\right) \Big/ \left(\faktor{(\hbar Y_\hbar(\mfg)+\Ker\,\Phi_\hbar)}{\hbar Y_\hbar(\mfg)}\right)\cong \left(\faktor{ Y_\hbar(\mfg)}{\hbar Y_\hbar(\mfg)}\right) \Big/ K=\{0\}, 
\]
and hence that $\gr Y(\mfg)\cong \mr{Rees}\,Y(\mfg)|_{\hbar=0}=\{0\}$. This is a contradiction since, for instance, $F_0Y(\mfg)$ embeds into 
$\gr Y(\mfg)$ and is nonzero. Therefore we must have $K=\{0\}$, and hence $\Phi_\hbar$ is an isomorphism of graded $\C[\hbar]$-algebras.

Since $\Phi_\hbar$ sends the ideal $\hbar Y_{\hbar}(\mfg)$ isomorphically onto $\hbar\mr{Rees}\,Y(\mfg)$, we obtain the sequence of isomorphisms 
\begin{equation*}
 \mf{Ug}[s]\cong Y_{\hbar}(\mfg)/\hbar Y_{\hbar}(\mfg)\cong \mr{Rees}\,Y(\mfg)|_{\hbar=0}\cong \gr Y(\mfg). \qedhere
\end{equation*}
\end{proof}

It can be deduced from \eqref{z1z2iso} and \eqref{PhiRees} that, under the isomorphism between $\gr Y_{\zeta}(\mfg)$ and $\mf{Ug}[s]$, $\ol{J(X)}$ corresponds to $\zeta X\ot s$ where $\ol{J(X)}$ is the image of $J(X)$ in the quotient $F_1 \, Y_{\zeta}(\mfg) /F_0\, Y_{\zeta}(\mfg)$. 

We end this subsection with a simple lemma which will play a role in the proof of Theorem \ref{T:Ycr(g)-} as well as in Sections \ref{sec:sosp-Y} and \ref{sec:rep}. 
\begin{lemma}\label{L:ext-aut}
 Let $\phi$ be an automorphism of the enveloping algebra $\mfU\mfg$. Then the assignment 
 \begin{equation}
 \quad X\mapsto \phi(X), \quad J(X)\mapsto J(\phi(X)) \quad \text{ for all }\; X\in \mfg \label{ext-aut}
 \end{equation}
extends uniquely to an automorphism $\bar \phi$ of $Y_\zeta(\mfg)$. 
\end{lemma}
\begin{proof}
  It is straightforward to see that the assignment \eqref{ext-aut} preserves the defining relations \eqref{J0} and 
\eqref{J1}. The same is true for the relations \eqref{J2} and \eqref{J3}. To see this, first observe that the right-hand sides of both of these relations are independent of the choice of orthonormal basis $\{X_\lambda\}_{\lambda \in \Lambda}$. The desired conclusion is then obtained from the second observation that, since the Killing form $B(\cdot,\cdot)$ of $\mfg$ is invariant under automorphisms of $\mfg$, the same is true for the form $(\cdot,\cdot)$. 
\end{proof}
%


\subsection{Current presentation}\label{sec:curpres}


We now focus our attention on what is often called Drinfeld's new or second presentation of the Yangian of $\mfg$. We will call it the current presentation of the Yangian and denote it by $Y_\hbar^\cur(\mfg)$. Let $\mcI$ be an indexing set for the simple roots $\al_i$ of $\mfg$,  $C=(c_{ij})_{i,j\in\mcI}$ denote the Cartan matrix of $\mfg$ and $\Delta^{+}$ denote the set of positive roots. We normalize the standard Chevalley generators $x_i^{\pm},h_i$ of the Lie algebra $\mfg$ such that $(x_i^+,x_i^-)=1$ and $h_i = [x_i^+,x_i^-]$.
\begin{definition}[\cite{Dr3}, Theorem 1]\label{D:Ycr(g)-}
Let $\zeta\in\C$. Set $\{x_1,x_2\}= x_1 x_2 + x_2 x_1$. The Yangian $Y_\zeta^\cur(\mfg)$ is the unital associative $\C$-algebra  generated by elements $x_{ir}^{\pm}$, $h_{ir}$ with $i\in \mcI$ and $r\ge 0$, subject to the following relations, for $i,j\in \mcI$ and $r,s\ge 0$:
\begin{gather}\label{Ycr(g)-1} 
[h_{ir},h_{js}]=0,\qu
[h_{i0},x_{js}^{\pm}]=\pm (\al_i,\al_j) \,x_{js}^{\pm},\qu  
[x_{ir}^{+},x_{js}^{-}]=\delta_{ij}\,h_{i,r+s} , \\
\label{Ycr(g)-2} 
[h_{i,r+1},x_{js}^{\pm}]-[h_{ir},x_{j,s+1}^{\pm}] = \pm\tfrac{1}{2}\zeta\, 
(\al_i,\al_j) \,\big\{ h_{ir}, x_{js}^{\pm} \big\} , \\
\label{Ycr(g)-3}  
[x_{i,r+1}^{\pm},x_{js}^{\pm}]-
[x_{ir}^{\pm},x_{j,s+1}^{\pm}]=\pm\tfrac{1}{2}\zeta\, 
(\al_i,\al_j) \,\big\{ x_{ir}^{\pm} , x_{js}^{\pm} \big\} , \\
\label{Ycr(g)-4}
 \msum_{\si\in \mfS_{m}} \Big[ 
x_{i,r_{\si(1)}}^{\pm},\big[x_{i,r_{\si(2)}}^{\pm},\ldots,[x_{i,r_{\si(m)}}^{\pm
},x_{js}^{\pm}]\ldots\big]\Big]=0 
\qu\text{for}\qu m=1-c_{ij},\;\;i\ne j,\;\; r_1,\ldots,r_m\ge 0 . 
\end{gather}
\end{definition}
\begin{theorem}[\cite{Le2}, Theorem 1.2 and \cite{GNW}, Theorem 2.13] \label{smallerset}
The algebra $Y_\zeta^\cur(\mfg)$ is generated by its elements $x_{ir}^{\pm},h_{ir}$ with $i\in\mcI$ and $r=0,1$ and is isomorphic to the associative $\C$-algebra generated by these elements subject to the following relations for $i,j\in \mcI$ and $r,s=0,1$:
\begin{gather}
[h_{ir},h_{js}]=0, \qu [h_{i0},x^\pm_{j0}]=\pm(\al_i,\al_j)\,x^{\pm}_{j0}, \qu  [\tl{h}_{i1},x^\pm_{j0}]=\pm(\al_i,\al_j)\,x^{\pm}_{j1} \;\text{ where }\; \tl{h}_{i1} = h_{i1} - \tfrac12\zeta \, h_{i0}^2 , \label{Le1} 
\\
[x^+_{ir},x^-_{j0}]=\del_{ij} h_{ir}, \qu
[x^\pm_{i1},x^\pm_{j0}] = [x^\pm_{i0},x^\pm_{j1}] \pm \tfrac12\zeta \,(\al_i,\al_j)\, \big\{x^\pm_{i0} , x^\pm_{j0} \big\}, \label{Le2} 
\\
1-c_{ij}\;\text{ commutators }\; \Big[ x^\pm_{i0}, \big[ x^\pm_{i0},\ldots,[x^\pm_{i0},x^\pm_{j0}]\ldots\big]\Big] =0 \;\text{ for }\; i\ne j. \label{Le3}
\end{gather}
When $\mfg=\mfsl_2$, we should also include $ \big[[\tl{h}_{i1},x^+_{i1}],x^-_{i1}\big] + \big[x^+_{i1},[\tl{h}_{i1},x^-_{i1}]\big] = 0$.
The isomorphism is given by the map: $x^\pm_{ir} \mapsto x^\pm_{ir}$ and $h_{ir} \mapsto h_{ir}$ for $r\ge0$, where $x^\pm_{i,r+1} = \pm (\al_i,\al_i)^{-1}\, [\tl{h}_{i1},x^\pm_{ir}]$ and $h_{ir} = [x^+_{ir},x^-_{i0}]$\; are defined recursively for $r\ge 1$.
\end{theorem}

In \cite{Le1}, an analog of the Poincar\'{e}-Birkhoff-Witt Theorem was proved for $Y_{\zeta}^\cur(\mfg)$. One consequence is that Proposition \ref{grYzeta} also holds for $Y_{\zeta}^\cur(\mfg)$ with respect to the filtration $F_{\bullet}$ on $Y_{\zeta}^\cur(\mfg)$  given by assigning filtration degree $r$ to $x_{ir}^{\pm}$ and $h_{ir}$. The isomorphism $\mf{Ug}[s] \iso \gr Y_{\zeta}^\cur(\mfg)$ sends $x_i^{\pm} \ot s^r$ and $h_i \ot s^r$ to $\zeta^{-r}\ol{x}_{ir}^{\pm}$ and $\zeta^{-r}\ol{h}_{ir}$, respectively, viewed as elements of  $F_r \, Y_{\zeta}^\cur(\mfg) / F_{r-1} \, Y_{\zeta}^\cur(\mfg)$. In particular, we have an embedding $\mfU\mfg \into Y_{\zeta}^\cur(\mfg)$ which sends $h_i$ to $h_{i0}$ and $x_i^{\pm}$ to $x_{i0}^{\pm}$. Consequently, we will sometimes write $h_i$ and $x_i^{\pm}$ instead of $h_{i0}$ and $x_{i0}^{\pm}$.


\subsection{Equivalence between the \texorpdfstring{$J$}{J} and current presentations of the Yangian}\label{sec:equivJcur}

In \cite{Dr2}, it was stated that the algebra $Y_\hbar(\mfg)$ given by Definition \ref{D:Y(g)-DI} is isomorphic to the Yangian $Y_\hbar^\cur(\mfg)$ as presented in Definition \ref{T:Ycr(g)-}. Moreover, an explicit isomorphism was given. We provide in this section a proof of the equivalence of those two definitions which makes use of Theorem \ref{smallerset}. This fills a gap in the literature since, to the best of our knowledge, no such proof has appeared. We start by stating the main result of {\it loc.~cit.} 

For each $\al\in\Delta^+$, let $x_{\al}^{\pm} \in \mfg_{\pm\al}$ be such that $(x_{\al}^+,x_{\al}^-) = 1$ and that $x_{\al_i}^{\pm} = x_{i0}^{\pm}$. We may write $x_i^{\pm}$ instead of $x_{i0}^{\pm}$. 

\begin{theorem}[\cite{Dr3}, Theorem 1]\label{T:Ycr(g)-}
The algebras $Y_\zeta^\cur(\mfg)$ and $Y_{\zeta}(\mfg)$ are isomorphic and an isomorphism $\Phi_{\mr{cr},J}$ between the two realizations is provided by:
\ali{
\Phi_{\mr{cr},J}(h_{i,0}) = {} & h_{i}, \qu && \Phi_{\mr{cr},J}(h_{i,1}) =  J(h_{i}) - \zeta\,v_i, \qu && \text{where} \qu && v_i = \tfrac{1}{4} \msum_{\alpha\in\Delta^+}  (\alpha,\al_i)\, \big\{x^+_\al , x^-_\al  \big\} - \tfrac{1}{2}\, h_i^2 , \notag
\\
\Phi_{\mr{cr},J}(x_{i,0}^\pm) = {} & x^\pm_{i}, && \Phi_{\mr{cr},J}( x_{i,1}^\pm) = J(x_i^\pm) - \zeta\, w_i^{\pm},\qu && \text{where} && w_i^{\pm} = \pm \tfrac{1}{4} \msum_{\al\in\Delta^+} \big\{ [x_i^\pm,x_\al^\pm], x_\al^\mp \big\} - \tfrac{1}{4}\,\big\{x_i^\pm, h_i\big\}. \notag
}
\end{theorem}

For (untwisted and twisted) quantum affine algebras, such a theorem was proved in \cite{Da1,Da2}.

To prove the theorem, we will employ the following lemma which has been proven in \cite{GNW}. Set $\tilde v_i=v_i+\tfrac{1}{2}h_{i}^2$ for all $i\in \mcI$. 
 
\begin{lemma}[\cite{GNW}, Lemma 3.9]\label{lem:GNW}
The following relations hold in the universal enveloping algebra $\mf{U}\mfg$:
\begin{align}
  & [h_{i}, v_j] = 0, \quad
  [h_{i}, w_j^\pm] = \pm(\alpha_i,\alpha_j)\, w_j^\pm, \label{hiwj}
\\
  & [\tilde v_i, x_{j}^\pm] = \pm (\alpha_i,\alpha_j)\, w^\pm_j, \label{vixj}
\\
  & [w^+_i,x^-_{j}] = \delta_{ij} v_i = [x^+_{i}, w^-_j], \label{wixj}
\\
  & [w^\pm_{i},x^\pm_{j}] - [x^\pm_{i}, w^\pm_j]
  = \mp \tfrac12 (\alpha_i,\alpha_j)\,(x^\pm_{i} x^\pm_{j} + x^\pm_{j} x^\pm_{i}). \label{wixjpm} 
  \end{align}
 \end{lemma}

\begin{proof}[Proof of Theorem \ref{T:Ycr(g)-}] We will assume for this proof that $\mfg\neq\mfsl_2$; the $\mfsl_2$-case will be considered in Appendix \ref{app}.

\medskip 

\noindent \textit{Step 1}: $\Phi_{\mr{cr},J}$ extends to a homomorphism of algebras. 

\medskip 

 It was observed in (3.12) of \textit{loc.~cit.}~that the above relations already imply that the images of the generators $x_{ir}^\pm, h_{ir}^\pm$  under $\Phi_{\mr{cr},J}$ 
 satisfy those defining relations of $Y_\zeta^\cur(\mfg)$ given in \eqref{Le1}-\eqref{Le3} except $[h_{i1},h_{j1}]=0$. For completeness, we repeat this observation here with some added detail. The relations involving only elements of the Lie algebra $\mfg$ hold automatically, so we need not consider these relations. 
 
\medskip

\noindent \textit{Step 1.1}: $\Phi_{\mr{cr},J}$ preserves the defining relations \eqref{Le1}-\eqref{Le3} except $[h_{i1},h_{j1}]=0$.  

\medskip

Let us begin by showing $[\Phi_{\mr{cr},J}(\tl{h}_{i1}),\Phi_{\mr{cr},J}(x_{j0}^\pm)]=\pm (\alpha_i,\alpha_j)\,\Phi_{\mr{cr},J}(x_{j1}^\pm)$ for all $i,j\in \mcI$.
Since $\tl{h}_{i1}=h_{i1}-\tfrac{1}{2}\zeta h_{i0}^2$, we have 
\begin{equation*}
[\Phi_{\mr{cr},J}(\tl{h}_{i1}),\Phi_{\mr{cr},J}(x_{j0}^\pm)]=[J(h_{i})-\zeta\, \tl{v}_i,x_{j}^\pm]=J([h_{i},x_{j}^\pm])\mp \zeta(\alpha_i,\alpha_j)\, w^\pm_j=\pm(\alpha_i,\alpha_j)\left(J(x_{j}^\pm)-\zeta w^\pm_j \right)=\pm(\alpha_i,\alpha_j)\,\Phi_{\mr{cr},J}(x_{j0}^\pm).
 \end{equation*}
Here we have made use of \eqref{vixj} and the relation $[J(X),Y]=J([X,Y])$ (for all $X,Y\in \mfg$) in the second equality. 

Next, we show that $[\Phi_{\mr{cr},J}(x_{i1}^+),\Phi_{\mr{cr},J}(x_{j0}^-)]=\delta_{ij}\Phi_{\mr{cr},J}(h_{i1})$. Using relation \eqref{wixj} and $[J(x_{i}^+),x_{j}^-]=J([x_{i}^+,x_{j}^-])$, we obtain the sequence of equalities
\begin{equation*}
 [\Phi_{\mr{cr},J}(x_{i1}^+),\Phi_{\mr{cr},J}(x_{j0}^-)]=[J(x_{i}^+)-\zeta w_i^+,x_{j}^-]=\delta_{ij}J(h_{i})-\zeta[w_i^+,x_{j}^-]=\delta_{ij}\left(J(h_{i})-\zeta v_i\right)=\delta_{ij}\Phi_{\mr{cr},J}(h_{i1}).
\end{equation*}
We now verify $[\Phi_{\mr{cr},J}(h_{i1}),\Phi_{\mr{cr},J}(h_{i0})]=0$: 
\begin{equation*}
 [\Phi_{\mr{cr},J}(h_{i1}),\Phi_{\mr{cr},J}(h_{j0})]=[J(h_{i})-\zeta v_i,h_{j}]=J([h_{i},h_{j}])-\zeta[v_i,h_{j}]=0,
\end{equation*}
by the first relation in \eqref{hiwj}.

The last relation which follows from Lemma \ref{lem:GNW} is 
\begin{equation*}
[\Phi_{\mr{cr},J}(x^\pm_{i1}),\Phi_{\mr{cr},J}(x^\pm_{j0})] = [\Phi_{\mr{cr},J}(x^\pm_{i0}),\Phi_{\mr{cr},J}(x^\pm_{j1})] \pm \tfrac12\zeta\,(\al_i,\al_j)\, \big\{\Phi_{\mr{cr},J}(x^\pm_{i0}) , \Phi_{\mr{cr},J}(x^\pm_{j0}) \big\}.
\end{equation*}
To see why this is true, let us begin with the left-hand side: 
\begin{align*}
 [\Phi_{\mr{cr},J}(x^\pm_{i1}),\Phi_{\mr{cr},J}(x^\pm_{j0})]= {} &[J(x_{i}^\pm)-\zeta w_i^\pm,x^\pm_{j}]\\
                                          = {} &J([x_{i}^\pm,x^\pm_{j}])-\zeta [w_i^\pm,x^\pm_{j}]\\
                                          = {} &[x_{i}^\pm,J(x^\pm_{j})]-\zeta([x^\pm_{i}, w^\pm_j]\mp \tfrac{1}{2}(\alpha_i,\alpha_j) \left\{x^\pm_{i},x^\pm_{j}\right\}) \quad \text{ by }\;\eqref{wixjpm}\\
                                          = {} & [x_{i}^\pm,J(x^\pm_{j})-\zeta w^\pm_j]\pm \tfrac12\zeta\,(\al_i,\al_j)\, \big\{\Phi_{\mr{cr},J}(x^\pm_{i0}) , \Phi_{\mr{cr},J}(x^\pm_{j0}) \big\}\\
                                          = {} &[\Phi_{\mr{cr},J}(x^\pm_{i0}),\Phi_{\mr{cr},J}(x^\pm_{j1})] \pm \tfrac12\zeta\,(\al_i,\al_j)\, \big\{\Phi_{\mr{cr},J}(x^\pm_{i0}) , \Phi_{\mr{cr},J}(x^\pm_{j0}) \big\}.
\end{align*}

\noindent \textit{Step 1.2}: $\Phi_{\mathrm{cr},J}$ preserves the relation $[\tilde h_{i1},\tilde h_{j1}]=0$.

By definition of $\Phi_{\mathrm{cr},J}$, we have 
\begin{equation*}
 [\Phi_{\mathrm{cr},J}(\tilde h_{i1}),\Phi_{\mathrm{cr},J}(\tilde h_{j1})]=[J(h_{i})-\zeta \tilde v_i,J(h_{j})-\zeta \tilde v_j]=[J(h_{i}),J(h_{j})]-\zeta[\tilde v_i,J(h_{j})]-\zeta[J(h_{i}),\tilde v_j]+\zeta^2[\tilde v_i,\tilde v_j]. 
\end{equation*}
Using the relations \eqref{J0} and \eqref{J1}, it can be deduced that $[J(h_{i}),\tilde v_j] = [J(h_{j}),\tilde v_i]$, and hence 
\begin{equation}\label{h<->J-v}
 [\Phi_{\mathrm{cr},J}(\tilde h_{i1}),\Phi_{\mathrm{cr},J}(\tilde h_{j1})]=[J(h_{i}),J(h_{j})]-\zeta^2[\tilde v_j,\tilde v_i].
\end{equation}

From \eqref{J2} with $X_1=h_{i},\, X_2=x_{j}^+$ and $X_3=x_{j}^-$ we obtain 
\begin{equation*}
[J(h_{i}),J(h_{j})]= \zeta^2 \sum_{\la,\mu,\nu} \al_{\la\mu\nu} \{ X_\la, X_\mu, X_\nu \}  \;\text{ where }\;\al_{\la\mu\nu}=\left([h_{i},X_\lambda],\big[[x_{j}^+,X_\mu],[x_{j}^-,X_\nu] \big]\right).
\end{equation*}
Here we have used $\big[J(h_{i}),[J(x_{j}^+),x_{j}^-]\big] = [J(h_{i}),J(h_{j})]$ and $\big[h_{i},[J(x_{j}^+),J(x_{j}^-)]\big] =0$. As a consequence, we deduce from \eqref{h<->J-v} that 
\begin{equation*}
 [\Phi_{\mathrm{cr},J}(\tilde h_{i1}),\Phi_{\mathrm{cr},J}(\tilde h_{j1})]\in \mathfrak{Ug}.
\end{equation*}
Our present goal is to show that $[\Phi_{\mathrm{cr},J}(\tilde h_{i1}),\Phi_{\mathrm{cr},J}(\tilde h_{j1})]$ also belongs to the subspace of primitive elements of $\mathfrak{Ug}$, which is precisely the Lie algebra $\mfg$.

Let $\square: Y_\zeta(\mfg)\to Y_\zeta(\mfg)\otimes Y_\zeta(\mfg)$ be the linear map defined by $\square(X)=X\otimes 1 + 1\otimes X$ for all $X\in Y_\zeta(\mfg)$. Then, setting $\Omega_-=\sum_{\alpha\in \Delta^+}x_\alpha^+\otimes x_\alpha^-$ and $\Omega_+=\Omega-\Omega_-$, we have
\begin{equation*}
 \Delta(J(h_i))=\square(J(h_i))+\frac{\zeta}{2}[h_i\otimes 1,\Omega]\quad \text{ and }\quad \Delta(\tilde v_i)=\square(\tilde v_i)-\frac{1}{2}[h_i\otimes 1,\Omega_+-\Omega_-],
\end{equation*}
which implies that $\Delta(\Phi_{\mathrm{cr},J}(\tilde h_{i1}))=\square(\Phi_{\mathrm{cr},J}(\tilde h_{i1}))+\zeta[h_i\otimes 1,\Omega_+]$. Therefore, 
\begin{align*}
 (\Delta-&\square)([\Phi_{\mathrm{cr},J}(\tilde h_{i1}),\Phi_{\mathrm{cr},J}(\tilde h_{j1})])\\
 =&\zeta[\square(\Phi_{\mathrm{cr},J}(\tilde h_{i1})),[h_j\otimes 1,\Omega_+]]+\zeta[[h_i\otimes 1,\Omega_+],\square(\Phi_{\mathrm{cr},J}(\tilde h_{j1}))]+\zeta^2[[h_i\otimes 1,\Omega_+],[h_j\otimes 1,\Omega_+]].
\end{align*}
By Part II of the proof of \cite[Thm. 4.9]{GNW}, the right-hand side of the above equality vanishes. This proves that $[\Phi_{\mathrm{cr},J}(\tilde h_{i1}),\Phi_{\mathrm{cr},J}(\tilde h_{j1})]$ is primitive, and hence an element of $\mfg$. By \eqref{h<->J-v} and Lemma \ref{lem:GNW}, this element also commutes with the Cartan subalgebra $\mfh$ of $\mfg$, and hence must itself belong to $\mfh$: 
\begin{equation*}
 [\Phi_{\mathrm{cr},J}(\tilde h_{i1}),\Phi_{\mathrm{cr},J}(\tilde h_{j1})]=[J(h_{i}),J(h_{j})]-\zeta^2[\tilde v_j,\tilde v_i]\in \mfh.
\end{equation*}
 To complete the proof, we appeal to Lemma \ref{L:ext-aut} with $\phi$ specialized to the Chevalley involution $\varkappa$ of $\mathfrak{Ug}$:
\begin{equation}
 \varkappa(h_i)=-h_i,\quad \varkappa(x_i^\pm)=-x_i^\mp \; \text{ for all } \; i\in \mcI. \label{varkappa0}
\end{equation}
We will denote its extension to $Y_\zeta(\mfg)$, defined by \eqref{ext-aut}, again by $\varkappa$.

Note that $\varkappa$ fixes $[J(h_i),J(h_j)]$, and the same is true for $[\tilde v_j,\tilde v_i]$ as it is a linear combination of elements of the form $\{x_\al^+,x_\al^-\}$, all of which are fixed by $\varkappa$. Since $\varkappa(H)=-H$ for all $H\in \mfh$, we can conclude that
\begin{equation*}
 [\Phi_{\mathrm{cr},J}(\tilde h_{i1}),\Phi_{\mathrm{cr},J}(\tilde h_{j1})]=[J(h_{i}),J(h_{j})]-\zeta^2[\tilde v_j,\tilde v_i]=0,
\end{equation*}
as desired. \qed

\noindent \textit{Step 2:} $\Phi_{\mr{cr},J}$ is an isomorphism. 

\medskip 

Recall that $Y_\zeta(\mfg)$ and $Y_\zeta^\cur(\mfg)$ admit filtrations and, by definition, $\Phi_{\mr{cr},J}$ respects these and hence
we may consider the associated graded homomorphism $\gr \Phi_{\mr{cr},J}:\gr Y_\hbar^\cur(\mfg)\to \gr Y_\hbar(\mfg)$. After identifying $\gr Y_\hbar^\cur(\mfg)$ and $\gr Y_\hbar(\mfg)$ with the enveloping algebra of $\mfg[s]$, $\gr \Phi_{\mr{cr},J}$ becomes the identity map, so $\gr \Phi_{\mr{cr},J}$ is an isomorphism and hence so is $\Phi_{\mr{cr},J}$.  \end{proof}

Up to this point we have not equipped $Y_\zeta^{\mr{cr}}(\mfg)$ with the structure of a Hopf algebra. This is remedied by imposing on $Y_\zeta^{\mr{cr}}(\mfg)$
the unique Hopf algebra structure such that the algebra isomorphism $\Phi_{\mr{cr},J}$ of Theorem \ref{T:Ycr(g)-} becomes an isomorphism of Hopf algebras. Formulas 
for the coproduct applied to the generators $h_{i1}$ and $x_{i1}^\pm$ of Theorem \ref{smallerset} are not difficult to compute and can be found, for example, in (4.7) and (4.13) of \cite{GNW}, respectively. In general, it is not at all a simple task to compute $\Delta(h_{ir})$ and $\Delta(x_{ir}^\pm)$ for arbitrary $r\geq 2$. This has, however, been achieved for $\mfg=\mfsl_N$: see Theorem 4.5 of \cite{Cr}.


\section{Orthogonal and symplectic Yangians}\label{sec:sosp-Y}


The main result in this section is the proof of the equivalence between the $RTT$ and the $J$-presentations, and thus with the current presentation as well. The general approach was sketched in \cite{Dr1}: we will follow this approach and provide more details, but we will also be able to offer a second, more elementary proof which circumvents the use of the universal $R$-matrix.

Before stating and proving these results, we need to recall some basic facts and notation related to the orthogonal and symplectic Lie algebras.


\subsection{Orthogonal and symplectic Lie algebras}

Let $N=2n$ or $N=2n+1$. We will denote by $\mfg_N$ either $\mfsp_N$ or $\mfso_N$: the latter case is only possible if $N\ge 3$. 
The Lie algebra $\mfg_N$ can be realized as a Lie subalgebra of $\mfgl_N$ as follows. We label the rows and columns of $\mfgl_N$ by the indices $\{ \pm1,\ldots,\pm n\}$ if $N=2n$ and by $\{0, \pm 1,\ldots,\pm n \}$ if $N=2n+1$. Set $\theta_{ij}=1$ in the orthogonal case and $\theta_{ij}=\mr{sign}(i)\cdot \mr{sign}(j)$ in the symplectic case for $i,j\in\{ \pm 1, \pm 2, \ldots, \pm n  \}$. We denote by $(\ge)$ the symbol $>$ in the orthogonal case and the symbol $\ge$ in the symplectic case. We define $(\le)$ analogously.
Let $E_{ij}$ denote the usual elementary matrix of $\mfgl_N$. Define the transposition $t$ by $(E_{ij})^{t} = \theta_{ij} E_{-j,-i}$ and set $F_{ij} = E_{ij} - (E_{ij})^{t}$ so that 
\[
[F_{ij},F_{kl}] = \delta_{jk} F_{il} - \delta_{il} F_{kj} + \delta_{j,-l} \theta_{ij} F_{k,-i} - \delta_{i,-k} \theta_{ij} F_{-j,l} \qu\text{and}\qu F_{ij} + \theta_{ij}F_{-j,-i}=0 .
\] 
Then $\mfg_N = \mr{span}_{\C} \{ F_{ij}  \, : \, -n\le i,j\le n \}$. It is understood here and for the rest of this paper that, for inequalities of the form $-n \le i\le n$, the index $i$ can be 0 only when $N$ is odd. A vector space basis of $\mfg_N$ is provided by 
\[ 
\{ F_{ij} \, : \, -n \le  i, j \le n \text{ and } i+j \,\,(\ge)\,\, 0  \}. 
\] 

As invariant, non-degenerate, bilinear form $(\cdot , \cdot)$ on $\mfg_N$, we can take $(X_1,X_2) = \frac{1}{2} \mr{Tr}(X_1X_2)$ for any $X_1,X_2\in\mfg_N$. It~follows that $(F_{ij},F_{kl}) = 2^{\del_{i,-j}}\del_{il}\del_{jk}$ for $i+j\,(\ge)\,0$ and $k+l\,(\ge)\,0$. Since the $J$-presentation of the Yangian is given in terms of an orthonormal basis, we provide here such a basis for $\mfg_N$.

A basis $\{ X_{\lambda}\}_{\la\in\Lambda_N}$ of $\mfso_N$ orthonormal with respect to this bilinear form consists of $X_{ij}$ with $(i,j)\in\Lambda_N$ where  
\begin{equation}
X_{ij} = (F_{ij} + F_{ji})/\sqrt{2}, \;\; X_{ji} = (F_{ij} - F_{ji})/\sqrt{-2} \;\text{ for }\; i<j \;\text{ and }\; X_{ii} = F_{ii}; \;\; \Lambda_N = \{ (i,j) \, : \,  i+j>0 \}. \label{eq:orthoso} 
\end{equation}
For $\mfsp_N$, an orthonormal basis will include all these matrices along with, for $1\le i\le n$, 
\begin{equation} 
X_{i,-i} = (F_{-i,i}+F_{i,-i})/2, \;\; X_{-i,i} = \sqrt{-1}\,(F_{-i,i} - F_{i,-i})/2, \;\text{ so }\; \Lambda_N = \{ (i,j) \, : \,  i+j\,(\ge)\, 0 \}. \label{eq:orthosp} 
\end{equation}

Introduce the permutation operator $P$ and a one-dimensional projector $Q$ on $\C^N \ot \C^N$ by
\eq{ 
P=\msum_{i,j} E_{ij} \ot E_{ji}, \qq Q = \msum_{i,j} \theta_{ij} E_{ij} \ot E_{-i,-j} , \notag
}
where the sums are assumed to be over the range $-n,\ldots,n$, and it is understood that the index $0$ is omitted when $N$ is even. (We will always assume this sum rule for indices $i,j,\ldots$ and $a,b,\ldots$, if not specified otherwise.)  Let $I$ denote the identity matrix. Then $P^2=I$, $P^{t_1}=P^{t_2}=Q$, $PQ=QP=(\pm) Q$ and $Q^2=N Q$. Here (and further in this paper) $t_1$~and $t_2$ denote the partial transpositions on $\End(\C^N\ot \C^N)$, and the upper sign in $(\pm)$ or $(\mp)$ corresponds to the orthogonal case while the lower sign corresponds to the symplectic case. 

Using the orthonormal basis of $\mfg_N$ given in \eqref{eq:orthoso} and \eqref{eq:orthosp}, it can be computed that 
\eq{
\msum_{i+j\,(\ge)\,0} X_{ij} \ot X_{ij} = \tfrac 12 \msum_{i,j} F_{ij} \ot F_{ji} = P - Q , \qq
\msum_{i+j\,(\ge)\,0} X_{ij} X_{ij} = \tfrac12\msum_{i,j} F_{ij} F_{ji}  \label{Omega}
}
and the Casimir element $\sum_{i+j\,(\ge)\,0} X_{ij} X_{ij}$ of $\mfg_N$ operates in the adjoint representation of $\mfg_N$ by the eigenvalue $4\kappa$ where $\kappa =\tfrac{1}{2}N (\mp)1$.

We choose as Cartan subalgebra of $\mfg_N$ the abelian Lie subalgebra $\mfh_N$ spanned by the basis $\{ F_{11},\ldots,F_{nn} \}$ and we denote by $\{ \eps_1, \ldots, \eps_N \}$ its dual basis. The elements $F_{ij}$ with $i<j$ (and $j\neq -i$ when $\mfg_N \cong\mfso_N$) are chosen to be the positive root vectors as in \cite{Mo3}. With this choice, the simple roots are the following, for $i=1,2,\ldots,n-1$: 
\aln{
\mfg_N = {} & \mfso_{2n+1} &&: \;\; &-\eps_1,\; \eps_i - \eps_{i+1}, \\
\mfg_N = {} & \mfsp_{2n} &&: \;\; &-2\eps_1,\; \eps_i - \eps_{i+1},  \\
\mfg_N = {} & \mfso_{2n} &&: \;\; &-\eps_1-\eps_2,\; \eps_i - \eps_{i+1}.  
}
These simple roots will also be denoted $\al_0,\al_1, \ldots,\al_{n-1}$ in this order. (To obtain the standard description of the simple roots of $\mfg_N$ as given, for instance, in Appendix C in \cite{Kn},  apply the correspondence $\eps_i \leftrightarrow -\eps_{n+1-i}$.) We will denote by $\om_i$ the $i^{\mr{th}}$ fundamental weight. These are given by the following expressions: 
\aln{
& \mfg_N = \mfso_{2n+1} &&: \;\; \om_0 = -\tfrac{1}{2}\msum_{j=1}^n \eps_j, \;\; \om_i = -\msum_{j=i+1}^n \eps_j \; \text{ for }\; 1\le i\le n-1, 
\\
& \mfg_N = \mfsp_{2n} &&: \;\; \om_i = -\msum_{j=i+1}^n \eps_j \; \text{ for }\; 0\le i\le n-1,  
\\
& \mfg_N = \mfso_{2n} &&:  \;\;  \om_0 = -\tfrac{1}{2}\msum_{j=1}^n \eps_j, \;\; \om_1 = -\tfrac{1}{2}(-\eps_1 + \msum_{j=2}^n \eps_j), \;\; \om_i = -\msum_{j=i+1}^n \eps_j \; \text{ for } 2\le i\le n-1. 
}

The Lie algebra $\mfg_N$ can be presented using generators $x_i^{\pm}$, $h_i$ with $i\in\mcI$ and relations (restrict Definition \ref{D:Ycr(g)-} to generators with $r=0$). The map between the generators $x^\pm_i$, $h_i$ and the standard basis elements $F_{ij}$ is given by, 
\begin{alignat}{99}
& \text{for }\; 1\le i \le n\!-\!1 \; && :\;\; x^+_i \mapsto F_{i,i+1} , && x^-_{i} \mapsto F_{i+1,i}, && h_i\, \mapsto F_{ii} - F_{i+1,i+1} && \qu\text{and} \label{x->I-1} \\
& \mfg_N = \mfso_{2n+1} &&: \;\;x^+_0 \mapsto F_{01} , && x^-_0 \mapsto  F_{10} , && h_{0} \mapsto -F_{11} , \label{x->I-2}
\\
& \mfg_N = \mfsp_{2n} &&: \;\; x^+_0 \mapsto F_{-1,1}/\sqrt{2} , \qu && x^-_0 \mapsto F_{1,-1}/\sqrt{2} , \qu && h_0 \mapsto -2F_{11} , \label{x->I-3} 
\\
& \mfg_N = \mfso_{2n} &&: \;\;x^+_{0} \mapsto F_{-1,2} , && x^-_{0} \mapsto F_{2,-1} , && h_{0} \mapsto -F_{11}-F_{22} . \label{x->I-4}
\end{alignat}

The current Lie algebra $\mfg_N[s]$ (which equals $\mfg_N \ot \C[s]$) is generated by elements $F_{ij}^{(r)}$ ($= F_{ij} \ot s^r$) satisfying
\begin{equation*} 
[F_{ij}^{(r_1)},F_{kl}^{(r_2)}] = (\delta_{jk} F_{il} - \delta_{il} F_{kj} + \theta_{ij} \delta_{j,-l} F_{k,-i} - \theta_{ij} \delta_{i,-k} F_{-j,l}) \ot s^{r_1+r_2}, \qq 
F_{ij}^{(r)} + \theta_{ij} F_{-j,-i}^{(r)} = 0. 
\end{equation*}

Finally, we remark that the following notation will be used throughout this section. For an $N\times N$ matrix $X$ with entries $x_{ij}$ in an associative algebra $\mcA$ over $\C$ we write
\eq{
X_{\ell} = \msum_{i,j} I^{\ot \ell-1} \ot E_{ij} \ot I^{\ot k-\ell} \ot x_{ij} \in \End((\C^N)^{\ot k}) \ot \mcA , \notag
}
where $k\in \N_{\ge2}$ will always be clear from the context.


\subsection{Extending certain representations of \texorpdfstring{$\mfg_N$}{} to \texorpdfstring{$Y_{\zeta}(\mfg_N)$}{}}

We explain how to obtain a representation of the Yangian on certain fundamental representations of the Lie algebra $\mfg_N$.
In particular, to prove the main theorem of this section (which is Theorem \ref{T:YR-iso}), we need to turn $\C^N$ into a representation of  $Y_\zeta(\mfg_N)$. 

\begin{proposition} \label{P:Y(g)-rep}
The natural representation of $\mfg_N$ on $\C^N$ extends to a representation of $Y_\zeta(\mfg_N)$ by letting $J(X)$ act by~$0$ for all $X\in\mfg_N$. 
\end{proposition}

To prove the proposition, we will use the following lemma which states the relevant representation of $Y^\cur_\zeta(\mfg_N)$. 

\begin{lemma} \label{L:Y(g)-rep}
Let $a=-\zeta/4$ if $\mfg_N=\mfso_{2n+1}$, $a=\zeta/2$ if $\mfg_N=\mfsp_{2n}$ and $a=-\zeta/2$ if $\mfg_N=\mfso_{2n}$.
Then the following formulas provide an algebra homomorphism $\varrho: Y_{\zeta}^\cur(\mfg_N) \rightarrow \End(\C^N)$, for $1\le i \le n-1$ and $r\ge0$:
\eq{ \label{rhoi}
\begin{gathered}\varrho(x_{ir}^+) = (a+i\zeta/2)^r E_{i,i+1} - (-a-i\zeta/2)^r E_{-i-1,-i}, \qu 
\varrho(x_{ir}^-) = (a+i\zeta/2)^r E_{i+1,i} - (-a-i\zeta/2)^r E_{-i,-i-1}, \\
\varrho(h_{ir}) = (a+i\zeta/2)^r (E_{i,i} - E_{i+1,i+1}) - (-a-i\zeta/2)^r (E_{-i,-i} - E_{-i-1,-i-1})
\end{gathered}
}
and
\eqa{
 \label{rho0B}
\qq & \mfg_N = \mfso_{2n+1} &&: \;\; \begin{gathered}
\varrho(x_{0r}^+) = (-\zeta/4)^r E_{01} - (\zeta/4)^r E_{-1,0}, \qq 
\varrho(x_{0r}^-) = (-\zeta/4)^r E_{10} - (\zeta/4)^r E_{0,-1} , \\ 
\varrho(h_{0r}) = (-\zeta/4)^r ( E_{00} - E_{11}) + (\zeta/4)^r (E_{-1,-1} - E_{00} ),
\end{gathered}
\\
\label{rho0C}
& \mfg_N = \mfsp_{2n} &&: \;\; \varrho(x_{0r}^+) = \del_{r0} F_{-1,1}/\sqrt{2}, \qq 
\varrho(x_{0r}^-) = \del_{r0} F_{1,-1}/\sqrt{2} , \qq \varrho(h_{0r}) = -2 \del_{r0}F_{11},
\\
 \label{rho0D}
& \mfg_N = \mfso_{2n} &&: \;\; \varrho(x_{0r}^+) = \del_{r0} F_{-1,2}, \qq 
\varrho(x_{0r}^-) = \del_{r0} F_{2,-1} , \qq 
\varrho(h_{0r}) = -\del_{r0} (F_{11}+F_{22}). \qq
}
\end{lemma}
\begin{proof}
It is known that the formulas \eqref{rhoi} define a representation of the subalgebra $Y_\zeta^\cur(\mfsl_n)\subset Y_\zeta^\cur(\mfg_N)$ on the space $\C^N$. For instance, they have appeared in Section 3 in \cite{ReSp} and can also be deduced using the evaluation homomorphism $\mathrm{ev}:Y_\zeta(\mfsl_n)\to \mfU\mfsl_n$ (see Proposition 12.1.15 in \cite{ChPr2}) together with the isomorphism $\Phi_{\cur,J}$ of Theorem \ref{T:Ycr(g)-}. Thus, proving the lemma amounts to verifying that this representation of $Y_\zeta^\cur(\mfsl_n)$ can be extended to all of $Y_\zeta^\cur(\mfg_N)$ via the assignments \eqref{rho0B}, \eqref{rho0C} and \eqref{rho0D}. 
It is possible to check that these formulas respect all the relations in Definition \ref{D:Ycr(g)-}, or to use Theorem \ref{smallerset} along with the inductive formula $\varrho(x_{i,r+1}^{\pm}) = \pm \frac{1}{2}[\varrho(\tilde{h}_{i1}),\varrho(x_{ir}^{\pm})]$ and $\varrho(h_{ir}) = [\varrho(x_{ir}^+),\varrho(x_{i0}^-)]$. Let us verify only one of these relations, namely \eqref{Ycr(g)-3} when $\mfg_N=\mfso_{2n+1}$ and $i=0,j=1,\pm=+$. We have
\begin{align}
[\varrho(x_{0,r+1}^+),\varrho(x_{1s}^{+})] & - [\varrho(x_{0r}^+),\varrho(x_{1,s+1}^{+})]  \notag \\
 = {} & [(-\zeta/4)^{r+1} E_{01} - (\zeta/4)^{r+1}E_{-1,0},(\zeta/4)^s E_{12} - (-\zeta/4)^s E_{-2,-1}] \notag  \\
 & \qquad \qquad  - [(-\zeta/4)^{r} E_{01} - (\zeta/4)^{r}E_{-1,0},(\zeta/4)^{s+1} E_{12} - (-\zeta/4)^{s+1} E_{-2,-1}] \notag \\
= {} & (-1)^{r+1} (\zeta/4)^{r+1+s} E_{02} - (-1)^s (\zeta/4)^{r+1+s} E_{-2,0} - (-1)^{r} (\zeta/4)^{r+1+s} E_{02} + (-1)^{s+1} (\zeta/4)^{r+1+s} E_{-2,0} \notag \\
= {} & 2((\zeta/4)^{r+1+s}) ((-1)^{r+1}  E_{02} + (-1)^{s+1} E_{-2,0}) \label{varrhoso1}
\end{align}
and
\begin{align}
-\tfrac{\zeta}{2}\,\{ \varrho(x_{0r}^+),\varrho(x_{1s}^{+}) \} = {} & -\tfrac{\zeta}{2}\, \{  (-\zeta/4)^{r} E_{01} - (\zeta/4)^{r}E_{-1,0}, (\zeta/4)^s E_{12} - (-\zeta/4)^s E_{-2,-1} \} \notag \\
= {} & -\tfrac{\zeta}{2}\, (\zeta/4)^{r+s} ((-1)^{r}E_{02} + (-1)^s E_{-2,0}) . \label{varrhoso2}
\end{align}
The equalities \eqref{varrhoso1} and \eqref{varrhoso2} show that applying $\varrho$ to the left-hand side and the right-hand side of \eqref{Ycr(g)-3} yields the same result. It is not more difficult to check all the other relations in all the cases.
\end{proof}

\begin{proof}[Proof of Proposition \ref{P:Y(g)-rep}]
We will denote by $\rho$ the representation of $Y_{\zeta}(\mfg_N)$ on $\C^N$ obtained from $\varrho$ via the isomorphism of Theorem \ref{T:Ycr(g)-}. Since $\mfg_N$ is a simple Lie algebra (except when $\mfg_N = \mfso_4$), we only need to show that $J(X)$ acts by $0$ on $\C^N$ for a single $X\in\mfg_N$ by \eqref{J0}. The case $\mfg_N=\mfso_4\simeq \mfso_2\oplus\mfso_2$ will be treated separately at the end. 
We choose $X=h_0$ and thus we need to compute the right-hand side of 
$
\rho(J(h_0))= \varrho(h_{01}) + \tfrac14 \zeta \sum_{\al\in\Delta^+} (\al,\al_0)\, \big\{\varrho(x^+_\al),\varrho(x^-_\al)\big\} - \tfrac12 \zeta\,(\varrho(h_{00}))^2
$
as an operator on $\C^N$. We check each case of $\mfg_N$ individually.

\noindent {\it Case I: $\mfg_N = \mfso_{2n+1}$}. From Lemma \ref{L:Y(g)-rep}, we obtain that $\varrho(h_{01}) - \tfrac12 \zeta\, (\varrho(h_{00}))^2 = -\tfrac14\zeta(E_{-1,-1}+2E_{00}+E_{11})$. Moreover,
\eqn{
\msum_{\al\in\Delta^+} (\al,\al_0)\, \big\{\varrho(x^+_\al),\varrho(x^-_\al)\big\} = {} & (\al_0,\al_0)\, \big\{\varrho(x^+_{00}),\varrho(x^-_{00})\big\} + \msum_{2\le i\le n} (\pm\eps_{1}-\eps_{i},-\eps_1)\, \big\{\varrho(x^+_{\pm\eps_1-\eps_i}),\varrho(x^-_{\pm\eps_1-\eps_i})\big\} 
\\
= {} & \{F_{01},F_{10}\} - \msum_{2\le i \le n} ( \{F_{1i} , F_{i1} \} - \{F_{-1,i} , F_{i,-1} \} ) = E_{-1,-1} + 2 E_{00} + E_{11}.
}

\noindent {\it Case II: $\mfg_N = \mfsp_{2n}$}. We now have $\varrho(h_{01}) - \tfrac12 \zeta\, (\varrho(h_{00}))^2 = -2\zeta(E_{-1,-1}+E_{11})$ and
\eqn{
\msum_{\al\in\Delta^+} (\al,\al_0)\, \big\{\varrho(x^+_\al),\varrho(x^-_\al)\big\} = {} & (\al_0,\al_0)\, \big\{\varrho(x^+_{00}),\varrho(x^-_{00})\big\} + \msum_{2\le i\le n} (\pm\eps_{1}-\eps_{i},-2\eps_1)\, \big\{\varrho(x^+_{\pm\eps_1-\eps_i}),\varrho(x^-_{\pm\eps_1-\eps_i})\big\} 
\\
= {} & 2\{F_{-1,1},F_{1,-1}\} - 2\msum_{2\le i \le n} ( \{F_{1i} , F_{i1} \} - \{F_{-1,i} , F_{i,-1} \} ) = 8(E_{-1,-1} + E_{11}).
}

\noindent {\it Case III: $\mfg_N = \mfso_{2n}$}. This time, we compute $\varrho(h_{01}) - \tfrac12 \zeta\, (\varrho(h_{00}))^2 = -\tfrac12\zeta(E_{-2,-2}+E_{-1,-1}+E_{11}+E_{22})$ and
\eqn{
\msum_{\al\in\Delta^+} (\al,\al_0)\, \big\{\varrho(x^+_\al),\varrho(x^-_\al)\big\} = {} & (\al_0,\al_0)\, \big\{\varrho(x^+_{00}),\varrho(x^-_{00})\big\} + \msum_{3\le i\le n} \msum_{j=1,2} (\pm\eps_{j}-\eps_{i},-\eps_1-\eps_2)\, \big\{\varrho(x^+_{\pm\eps_j-\eps_i}),\varrho(x^-_{\pm\eps_j-\eps_i})\big\} 
\\
= {} & 2\{F_{-1,2},F_{2,-1}\} - \msum_{3\le i \le n} \msum_{j=1,2}( \{F_{ji} , F_{ij} \} - \{F_{-j,i} , F_{i,-j} \} ) = 2 (E_{-2,-2}+E_{-1,-1}+E_{11}+E_{22}).
}
The expressions above yield $\rho(J(h_0))=0$ for each case of $\mfg_N$, as required. Finally, for $\mfg_N=\mfso_4$ we additionally need to compute $\rho(J(h_1))$:
\eqn{
\varrho(J(h_1)) = {} & \varrho(h_{11}) + \tfrac12 \zeta\, \big\{\varrho(x^+_{10}),\varrho(x^-_{10})\big\} - \tfrac12 \zeta\,(\varrho(h_{10}))^2 =  \tfrac12 \zeta \big\{F_{12},F_{21}\big\} -\tfrac12\zeta(E_{-2,-2}+E_{-1,-1}+E_{11}+E_{22}) = 0. \qedhere
}
\end{proof}

That $\C^N$ can be made into a representation of $Y_{\zeta}(\mfg_N)$ is also a consequence of the more general Proposition \ref{P:fundrepYang}, stated below, when $\omega_i=\omega_{n-1}$ and thus $V(\omega_i) = \C^N$. (\cite[Example 2]{Dr1} is Proposition \ref{P:fundrepYang} specialized to $\C^N$ and $\mfso_N$.) We presented the calculations above to show that, via the realization of the Yangian given by Definition \ref{D:Ycr(g)-}, it is not difficult to turn $\C^N$ into a representation of $Y_{\zeta}^\cur(\mfg_N)$ (and thus $Y_\zeta(\mfg_N)$) using rudimentary means, especially when starting with the known representation of $Y_{\zeta}^\cur(\mfsl_N)$ on $\C^N$.  Moreover, it is rare that one can give an explicit formula for the action of $x_{ir}^{\pm}$ on a representation of $Y_{\zeta}^\cur(\mfg_N)$ for all $r\ge 0$ and the formulas in the previous proof could be useful. The more general results of Proposition \ref{P:fundrepYang} are not needed in this section, but will be required in Section \ref{sec:rep}. 

Before stating Proposition \ref{P:fundrepYang}, we need to recall what a fundamental representation of the Yangian is. 
\begin{definition}\label{Funrep}
Let $a\in\C$ and $i\in \mcI$. We denote by $V(i;a)$ the irreducible finite-dimensional representation of $Y_{\zeta}^\cur(\mfg_N)$ with Drinfeld polynomials $(P_j(u))_{j\in \mcI}$ given by  $P_j(u)=1$ if $j\neq i$ and $P_i(u) = u-a$. We are referring here to the classification theorem for finite-dimensional irreducible representations of Yangians given in \cite{Dr3} - see also Theorem 12.1.11 in \cite{ChPr2} and \eqref{Dr-class} below.
\end{definition}
We also use the same terminology and notation when $V(i;a)$ is viewed as a $Y_\zeta(\mfg_N)$-module via the isomorphism of Theorem \ref{T:Ycr(g)-}. 
\begin{proposition}[\cite{Dr2} Theorem 7, \cite{ChPr2} Proposition 12.1.17]\label{P:fundrepYang}
Let $m_i$ be the multiplicity of the simple root $\al_i$ in the highest root $\theta$ of $\mfg_N$. If $m_i=1$ or if $m_i = \frac{(\theta,\theta)}{(\al_i,\al_i)}$, the fundamental representation $V(\omega_i)$ of $\mfg_N$ can be made into a fundamental representation of $Y_{\zeta}(\mfg_N)$ by letting $J(X)$ act by $aX$ for any $a\in\C$. 
\end{proposition}
\begin{remark}
This proposition holds more generally for an arbitrary semisimple Lie algebra $\mfg$, but we restrict ourselves to $\mfg_N$ because this is the case which interests us the most in this paper and a more general proof would require a case-by-case analysis for the exceptional Lie algebras.
\end{remark}
This proposition is not proved in \cite{Dr2}, but a sketch of a proof is given in \cite{ChPr2}. In the particular case when $\mfg_N = \mfsp_N$, there is a proof provided in \cite{AMR} (see Theorem 5.31) except that it omits the details for one important step, so we provide more explanations below. 

\begin{proof}
We start by repeating the beginning of the proof of Proposition 12.1.17 given in \cite{ChPr2}. Consider $V(i;0)$ which is generated by a highest weight vector, so it follows that, as a $\mfg_N$-module, it decomposes as \[ V(\omega_i) \oplus \bigoplus_{\mu < \omega_i} V(\mu)^{\oplus m_{\mu}}, \; m_{\mu} \ge 0, \] where $V(\mu)$ is the irreducible finite-dimensional representation of $\mfg_N$ with highest weight $\mu$. The map $\mfg_N \ot V(\omega_i) \lra V(i;0)$ given by $X\ot v \mapsto J(X)(v)$ is a $\mfg_N$-module homomorphism; it is enough to show that the image of $\mfg_N \ot V(\omega_i)$ under this homomorphism equals $V(\omega_i)$ and that the action of $J(X)$ is given by $aX$.  We need to prove the following two claims:

\noindent (i) If $m_i=1$ or $m_i = \frac{(\theta,\theta)}{(\al_i,\al_i)}$ and $\mu$ is a dominant integral weight less than $\omega_i$, there is no non-trivial $\mfg_N$-module homomorphism $\mfg_N \ot V(\omega_i) \lra V(\mu)$. 

\noindent (ii) The space of $\mfg_N$-module homomorphisms $\mfg_N \ot V(\omega_i) \lra V(\omega_i)$ has dimension one and hence consists of the scalar multiples of the action of $\mfg_N$ on $V(\omega_i)$.

Let us explain now how to prove (i) which, as suggested in \cite{ChPr2}, can be checked using a case-by-case analysis. We will need the following classical fact (see, for instance, exercise 25.33 in \cite{FuHa} or Proposition 3.2 in \cite{Ku}): if $V(\nu_1)$ and $V(\nu_2)$ are finite-dimensional, irreducible representations of a semisimple Lie algebra $\mfg$ with highest weights $\nu_1$ and $\nu_2$, and if $V(\nu)$ is an irreducible component of $V(\nu_1) \ot V(\nu_2)$, then $\nu = \nu_2 + \eta$ where $\eta$ is a weight of $V(\nu_1)$. It follows that, to prove (i), we only have to consider weights $\mu$ of the form $\omega_i - \al$ where $\al$ is a positive root of $\mfg_N$. Information about the multiplicities $m_i$ can be found, for instance, in Appendix C of \cite{Kn}.

{\it Proof of (i) for $\mfg_N=\mfso_{2n+1}$.} In this case, $\theta = - \eps_{n-1} - \eps_n$. The only value of $i$ for which $m_i =1$ is $i=n-1$ and the only one for which $m_i = \frac{(\theta,\theta)}{(\al_i,\al_i)}$ is $i=0$. For these two values $i$, $\omega_i - \al$ is never dominant for any positive root $\al$ of $\mfg_N$,  except that $\omega_i-\alpha=0$ when $i=n-1$ and $\alpha=-\eps_n$. However, $\mfg_N\ot V(\omega_i)$ does not contain the trivial representation as an irreducible component because $V(\omega_i)$ is not isomorphic to the coadjoint representation. Therefore, we can conclude that, provided $i$ is such that $m_i=1$ or $m_i=\frac{(\theta,\theta)}{(\al_i,\al_i)}$, no irreducible component of $\mfg_N \ot V(\omega_i)$ has highest weight $\omega_i -\al$, and thus (i) holds for 
$\mfg_N=\mfso_{2n+1}$.

{\it Proof of (i) for $\mfg_N=\mfso_{2n}$.} All the roots have the same length, so $\frac{(\theta,\theta)}{(\al_i,\al_i)}=1$ for all $i$. $\theta = - \eps_{n-1} - \eps_n$ and the only values of $i$ for which $m_i=1$ are $i=0,1,n-1$.  Observe that  $\omega_i - \al$ is never dominant for $i=0,1,n-1$ and any positive root $\al$ because $(\omega_i,\al)=0$ or $1$ for these values of $i$ and $\al$.

{\it Proof of (i) for $\mfg_N=\mfsp_{N}$.} In this case, $\theta=-2\eps_n$ and the condition $m_i = 1$ or $m_i = \frac{(\theta,\theta)}{(\al_i,\al_i)}$ is satisfied for all $i$.  In general, $\omega_i - \al$ is not dominant except when $\al = -\eps_{i+1} - \eps_{i+2}$ and $0\le i\le n-2$, in which case $\omega_i - \al = \omega_{i+2} \text{ or } 0$. When $i=n-2$ and $\al = -\eps_{n-1} - \eps_n$, we have $\omega_i - \al=0$, but the trivial representation is not a submodule of $\mfg_N \ot V(\omega_i)$ for the same reason as given in the $\mfg_N=\mfso_{2n+1}$ case. Let us suppose that $0\le i \le n-3$. We have to see why $V(\omega_{i+2})$ is not a submodule of $\mfg_N \ot V(\omega_i)$. Suppose that, indeed, it is a submodule. Then it admits a highest weight vector $v_{\mr{max}}$ of the form 
\begin{equation*}
 X_{\eps_{i+1}+\eps_{i+2}} \ot v_{\omega_i} + \msum_{k=0}^{n-1} h_k \ot v^k_{\omega_i + \eps_{i+1} + \eps_{i+2}}  +  \msum_{\substack{\beta\in\Z_{\ge 0}\Delta^+ \\ \be\neq \eps_{i+1} + \eps_{i+2}}} c_{\beta} X_{\eps_{i+1}+\eps_{i+2} + \beta} \ot v_{\omega_i - \beta}, \quad c_{\beta}\in\C. 
 \end{equation*}
Here, $c_{\beta}=0$ if $\eps_{i+1}+\eps_{i+2} + \beta$ is not a root or $\omega_i -\beta$ is not a weight of $V(\omega_i)$,   $X_{\eps_{i+1}+ \eps_{i+2}}$  and $X_{\eps_{i+1}+ \eps_{i+2} + \beta}$ are root vectors for the roots $\eps_{i+1}+ \eps_{i+2}$ and $\eps_{i+1}+ \eps_{i+2} + \beta$,  and $v_{\omega_i - \beta}$ is a weight vector of weight $\omega_i-\beta$. The $v^k_{\omega_i + \eps_{i+1} + \eps_{i+2}}$ are also some weight vectors of weight $\omega_i + \eps_{i+1} + \eps_{i+2}$. Moreover, the weight space $ V(\omega_i)[\omega_i - \al_j]$ is zero unless $\al_j = \al_i$.

Since $v_{\mr{max}}$ is a highest weight vector, $ 0 = X_{\al_{i+1}} (v_{\max}) =  [X_{\al_{i+1}},X_{\eps_{i+1}\eps_{i+2}}] \ot v_{\omega_i} +  X$ where $X$ is a linear combination of tensors of the form $X_{\al} \ot v$ with $v$ a weight vector of $V(\omega_i)$ of weight $<\omega_i$ because $ V(\omega_i)[\omega_i - \al_{i+1}]=0$. This is a contradiction because, for some non-zero factor $c$, $[X_{\al_{i+1}},X_{\eps_{i+1}+\eps_{i+2}}] = c[E_{i+1,i+2} - E_{-i-2,-i-1},E_{i+2,-i-1} + E_{i+1,-i-2}] = 2cE_{i+1,-i-1} \neq 0$, so there is no highest weight vector in $\mfg_N \ot V(\omega_i)$ of weight $\omega_i + (\eps_{i+1} + \eps_{i+2})$.

{\it Proof of (ii).} Since the map $\mfg_N \ot V(\omega_i) \rightarrow V(\omega_i)$ given by $X \ot v \mapsto X(v)$ is a non-zero $\mfg_N$-module homomorphism, $V(\omega_i)$ occurs with multiplicity at least one as a $\mfg_N$-submodule of $\mfg_N \ot V(\omega_i)$. The proof that its multiplicity is exactly one is given by Lemma 2.3 in \cite{ChPr1}.  It follows that $J(X)$ must act as $bX$ for some $b\in\C$ in $V(i;0)$. To obtain a representation where $J(X)$ acts as $aX$ for an arbitrary fixed $a\in\C$, we twist the representation $V(i;0)$ by the automorphism $\tau_{(a-b)\zeta^{-1}}$ (see \eqref{isotau}), which coincides with $V(i;a-b)$.
\end{proof}

 Let $d_i=1$ for $1\leq i\leq n-1$, and set $d_0=\tfrac{1}{2}$ if $\mfg_N=\mfso_{2n+1}$, $d_0=2$ if $\mfg_N=\mfsp_{2n}$, and $d_0=1$ if $\mfg_N=\mfso_{2n}$. The proof of Proposition \ref{P:fundrepYang} leads to the following corollary.
\begin{corollary}\label{C:J(X)-b}
Suppose that $m_i$ with $i\in \mcI$ satisfies the conditions of Proposition \ref{P:fundrepYang}, and let $a\in \C$. Then, viewed as a $\mfg_N$-module, the fundamental representation $V(i;a)$  is isomorphic to $V(\omega_i)$. Its $Y_\zeta(\mfg_N)$-module structure is obtained by  allowing $J(X)$ to operate as $bX$ for all $X\in \mfg_N$, where 
\begin{equation}
 b=d_i a+ \tfrac{\zeta d_i}{2}(\ka-d_i). \label{J(X)-b}
\end{equation}
\end{corollary}

\begin{proof}
Aside from the specific relationship \eqref{J(X)-b}, the corollary has been proven when $a=0$ in the proof of Proposition \ref{P:fundrepYang}, and in the general case the same argument applies. It is left to show that $b$ is related to $a$ via the relation \eqref{J(X)-b}. Let $(\lambda_i^r)_{i\in \mcI,r\geq 0}$ denote the highest weight of the $Y_\zeta^\cur(\mfg_N)$-module $V(i;a)$, and $\xi$ the highest weight vector (see Section \ref{sec:rep}). That is, $\xi$ generates $V(i;a)$, while $x_{ir}^+\xi=0$ and $h_{ir}\xi=\lambda_i^r \xi$ are satisfied for all $i\in \mcI$ and $r\geq 0$. The fact that $V(i;a)$ corresponds to the Drinfeld tuple $(P_j(u))_{i\in\mcI}$ with $P_i(u)=u-a$ and $P_j(u) = 1$ for $j\neq i$ means precisely that 
\begin{equation*}
 1+\zeta\msum_{r=0}^\infty \lambda_j^ru^{-r-1}=\begin{cases}
                                               \frac{u-a+\zeta d_{j}}{u-a} & \text{ if } i=j,\\
                                                1 & \text{ otherwise}.
                                              \end{cases}
\end{equation*}
Expanding the right-hand side of the above equality as an element of $\C[[u^{-1}]]$ and comparing coefficients, we deduce that $\lambda_j^r=\delta_{ij} d_j a^r$ for all $j\in \mcI$ and $r\geq 0$. Via the isomorphism $\Phi_{\mr{cr},J}$ of Theorem \ref{T:Ycr(g)-}, $J(h_i)$ corresponds to $h_{i1}+\zeta v_i$, and thus to complete the proof we just need to show $(h_{i1}+\zeta v_i)\xi=b \xi$, where $b$ is defined by \eqref{J(X)-b}. 

In \cite[3(ii)]{GNW}, it was shown that $v_j=\tfrac{1}{4}h^\vee h_j+\frac{1}{2}\sum_{\al\in \Delta^+}(\al,\al_j)x_\al^- x_\al^+ -\tfrac{1}{2}h_j^2$ for all $j\in \mcI$, provided $(\cdot,\cdot)$ has been normalized such that a long root has length 2. In the $\mfg_N=\mfsp_N$ case, the latter property does not hold so we must replace $h^\vee$ with $2h^\vee$ to account for this discrepancy. Here $h^\vee$ is the dual Coxeter number of $\mfg_N$: it is equal to $2n-1$ if $\mfg_N=\mfso_{2n+1}$, to $2n-2$ if $\mfg_N=\mfso_{2n}$, and to $n+1$ if $\mfg_N=\mfsp_{2n}$. After multiplying the $h^\vee$ in the $\mfsp_{2n}$ case by $2$, we see that these scalars coincide with $2\ka$. As $h_{i1}\xi=d_i a \xi$, 
$h_{i0}\xi=d_i \xi $
and $x_\al^+ \xi=0$ for all $\al\in \Delta^+$, we obtain 
\begin{equation*}
 (h_{i1}+\zeta v_i)\xi=d_i a\xi +\zeta (\tfrac{1}{4}2\ka h_i -\tfrac{1}{2}h_i^2)\xi= d_i a \xi+\tfrac{\zeta d_i}{2}(\ka-d_i)\xi. \qedhere
\end{equation*}
\end{proof}


\subsection{\texorpdfstring{$RTT$}{RTT}-presentation}

There is yet another presentation of the Yangian of $\mfg_N$, which we denote by $Y^R_{\zeta}(\mfg_N)$ and which is based on the $R$-matrix $R(u)\in\End((\C^N)^{\ot2})[[u^{-1}]]$ associated to its representation on $\C^N$ \cite{KuSk,ZaZa}:
\eq{
R(u) = I - \frac{\zeta}{u}\, P + \frac{\zeta}{u-\zeta\,\ka}\, Q , \qu\text{where}\qu \ka = \frac N2 (\mp) 1. \label{R(u)}
}
It is a solution of the quantum Yang-Baxter equation on $(\C^N)^{\ot3}$, that is, 
\eq{
R_{12}(u-v)\,R_{13}(u)\,R_{23}(v) = R_{23}(v)\,R_{13}(u)\,R_{12}(u-v) . \notag
}

\begin{definition} \cite{AACFR}\label{D:Y(g)-RTT}
The Yangian $Y^R_{\zeta}(\mfg_N)$ is the unital associative $\C$-algebra generated by the coefficients $t_{ij}^{(r)}$, with $-n \le i,j \le n$ and $r\in\Z_{\ge 0}$, of the formal series $t_{ij}(u) = \del_{ij} + \sum_{r\ge1} t_{ij}^{(r)} \zeta^r u^{-r}$ satisfying the following relations in $Y^R_\zeta(\mfg_N)[[u^{-1},v,v^{-1}]]$ and $Y^R_\zeta(\mfg_N)[[u^{-1}]]$, respectively:
\begin{gather} 
[\, t_{ij}(u),t_{kl}(v)] =\frac{\zeta}{u-v}
\Big(t_{kj}(u)\, t_{il}(v)-t_{kj}(v)\, t_{il}(u)\Big)-\frac{\zeta}{u-v-\zeta\,\ka}
\msum_{a} \Big(\delta_{k,-i}\,\theta_{ia}\, t_{aj}(u)\, t_{-a,l}(v)-
\delta_{l,-j}\,\theta_{ja}\, t_{k,-a}(v)\, t_{ia}(u)\Big) , \label{RTT:BCD}
\\
\msum_a t_{ia}(u)\,t_{-j,-a}(u+\zeta\,\ka) = \msum_a t_{-a,-i}(u+\zeta\,\ka)\,t_{aj}(u) = \del_{ij} . \label{unitary}
\end{gather}
\end{definition}

\begin{proposition} \label{P:RTT} \cite{AACFR}
We can collect the series $t_{ij}(u)$ into the generating matrix $T(u)= \sum_{i,j} E_{ij}\ot t_{ij}(u)$. Then the defining relations of $Y^R_\zeta(\mfg_N)$ are equivalent to
\begin{gather} 
R(u-v)\,T_1(u)\,T_2(v) = T_2(v)\,T_1(u)\,R(u-v) , \label{RTT} \\
T^{t}(u+\zeta\,\ka)\, T(u) = T(u)\, T^{t}(u+\zeta\,\ka) = I. \label{TT=TT=I}
\end{gather} 
Furthermore, $Y^R_\zeta(\mfg_N)$ admits a Hopf algebra structure with coproduct $\Delta : t_{ij}(u) \mapsto \sum_k t_{ik}(u) \ot t_{kj}(u)$, antipode $S : T(u) \mapsto T^{-1}(u)$ and counit $\eps: T(u) \mapsto I$.
\end{proposition}

Introduce an ascending filtration on the Yangian $Y^R_\zeta(\mfg_N)$ by setting $\deg t^{(r)}_{ij} = r-1$ for $r\ge 1$ and denote by $\bar t^{(r)}_{ij}$ the images of the elements $t^{(r)}_{ij}$ in the $(r-1)$-th component of the associated graded algebra $\gr Y^R_\zeta(\mfg_N)$. Then, by Theorem~3.6 of \cite{AMR}, the map defined by
\eq{
\psi \;:\; \mfU\mfg_N[s] \to  \gr Y^R_\zeta(\mfg_N), \qu F^{(r-1)}_{ij} \mapsto \bar t^{\,(r)}_{ij}  \label{iso:grYR-Ug}
}
for all $-n \le i,j \le n$ and $r\ge1$ is an isomorphism of algebras. As a consequence, we have an embedding $\mfU\mfg_N \into Y^R_\zeta(\mfg_N)$ which identifies $F_{ij}$ with $t_{ij}^{(1)}$ - see Proposition 3.11 of \cite{AMR}.


\subsection{Universal \texorpdfstring{$R$}{R}-matrix}

Our main goal in this subsection is to relate $R(u)$ to the universal $R$-matrix of $Y_{\zeta}(\mfg_N)$.  A precise relation will be given in Proposition \ref{C:RmcR}, which itself can be viewed as a particular case of Theorem \ref{RmcR}. The latter of these two results is valid for any $\mfg$ and first appeared in \cite{Dr1} without proof: see Theorem 4 therein. Let us begin by introducing some useful notation and the universal $R$-matrix $\mcR(u)$ of $Y_\zeta(\mfg)$ for any complex simple Lie algebra $\mfg$.  

For each $z\in \C$, denote by $\tau_z$ the automorphism of $Y_\zeta(\mfg)$ defined analogously to that in \eqref{isotau} with $\hbar$ specialized to~$\zeta$. Given an arbitrary $Y_\zeta(\mfg)$-module $V$ with corresponding homomorphism $\varrho:Y_\zeta(\mfg)\to \End(V)$, we set $V_z = \tau_z^*(V)$ and denote by $\varrho_z$ the resulting algebra homomorphism $Y_\zeta(\mfg) \rightarrow \End(V_z)$. We also set $\tau_{a,b} = \tau_{a} \ot \tau_{b}$ and $\varrho_{a,b} = \varrho_a \ot \varrho_b$ for all $a,b\in \C$.

Let $\si : w_1 \ot w_2  \mapsto w_2 \ot w_1$ denote the permutation operator on the tensor product $W \ot W$ of any vector space $W$. 

\begin{theorem}[\cite{Dr1}, Theorem 3]\label{uniR}
There is a unique formal series $\mcR(u) = 1 + \sum_{k=1}^{\infty} \mcR_k \,\zeta^k u^{-k}$ with $\mcR_k \in Y_\zeta(\mfg) \ot Y_\zeta(\mfg)$ such that $(id \ot \Delta)\,\mcR(u) = \mcR_{12}(u)\,\mcR_{13}(u)$, $ (id\ot S)\,\mcR(u)=\mcR^{-1}(u)$ and 
\eq{ 
\tau_{0,u\zeta^{-1}} \Delta'(Y) = \mcR(u)^{-1} (\tau_{0,u\zeta^{-1}}\Delta(Y))\, \mcR(u) \qu\text{for all}\qu Y \in Y_\zeta(\mfg). \label{UR:1}  
}
where $\Delta' = \si \circ \Delta$.  This $\mcR(u)$ satisfies the universal quantum Yang-Baxter equation 
\eq{ 
\mc{R}_{12}(u-v)\,\mc{R}_{13}(u)\,\mc{R}_{23}(v) = \mc{R}_{23}(v)\,\mc{R}_{13}(u)\,\mc{R}_{12}(u-v) \label{UYBE}
}
and is called the universal $R$-matrix. Moreover, 
\begin{gather} 
\mcR_{12}(u)\, \mcR_{21}(-u)=1, \qq \tau_{a\zeta^{-1},b\zeta^{-1}}\mcR(u)=\mcR(u+b-a), \label{RR=I} \\
\ln \mcR(u) = \zeta u^{-1} \msum_{\la\in\Lambda} X_{\la} \ot X_{\la} + \zeta u^{-2} \msum_{\la\in\Lambda} \big(J(X_{\la}) \ot X_{\la} - X_{\la} \ot J(X_{\la})\big) + O(u^{-3}), \label{lnR}
\end{gather} 
where $\{X_\la\}_{\la\in\Lambda}$ is an orthonormal basis of $\mfg$ with respect to the given bilinear form and $\Lambda$ is an indexing set. 
\end{theorem}


Theorem 4 in \cite{Dr1} is the special case of the next theorem when $V=W$. Its proof relies on Theorem \ref{uniR} above.

\begin{theorem}[\cite{Dr1}, Theorem 4]\label{RmcR}
Let $V,W$ be two irreducible representation of $Y_\zeta(\mfg)$ with corresponding homomorphisms $\varrho^V:Y_\zeta(\mfg)\to \End(V)$ and $\varrho^W:Y_\zeta(\mfg)\to \End(W)$. Then, up to multiplication by a formal power series in $u^{-1}$, $R^{VW}(u)=(\varrho^V\ot \varrho^W)(\mathcal{R}(-u))$ is the unique solution $\ms{R}(u)\in \End(V\ot W)[[u^{-1}]]$ of the equation 
\begin{equation}
(\varrho^V_{u\zeta^{-1}}\otimes \varrho^W_{v\zeta^{-1}}) (\Delta(J(X)))\,\ms{R}(u-v)=\ms{R}(u-v)\,(\varrho^V_{u\zeta^{-1}}\otimes \varrho^W_{v\zeta^{-1}}) (\Delta'(J(X)))\; \text{ for all }\; X\in \mfg.\label{PRRP}
\end{equation}
 Moreover, there exists $f(u)\in 1+u^{-1}\C[[u^{-1}]]$ such that the rescaled $R$-matrix $f(u)\,R^{VW}(u)$ is rational. 
\end{theorem}

\begin{remark}\label{RX=XR}
If $\ms{R}(u)\in \End(V\ot V)[[u^{-1}]]$ is a solution of \eqref{PRRP}, then $\ms{R}(u)$ intertwines the action of $\mfg$ on $V\ot W$. To see this, first write $\ms{R}(u)= 1 + \sum_{k=1}^\infty R^{(k)}u^{-k}$. Then, for any fixed $m\geq 1$, multiply \eqref{PRRP} by $(u-v)^{m-1}$ and expand both sides as elements of $(\End(V\ot W)[v])((u^{-1}))$. In particular $\frac{1}{u-v}$ should be expanded as the formal series $\sum_{k=0}^\infty v^k u^{-k-1}$.
By comparing the coefficient of $vu^{-1}$ on both sides of the resulting expression, we obtain $(\varrho^V\ot \varrho^W)(\Delta(X))\,R^{(m)}=R^{(m)}\,(\varrho^V\ot \varrho^W)(\Delta(X))$. 
\end{remark}

Since a proof of Theorem \ref{RmcR} has not appeared in the literature, we provide one here. We note however that a proof of the analogous result for quantum affine algebras has appeared in \cite{Ji} and \cite{FrRe} (see also \cite{EFK}). The proof which we present is in the same spirit as that of \cite{FrRe}: to this end, see Remark \ref{R:FrRe} below.
\begin{proof}
First observe that $R^{VW}(u)$ is in fact a solution of the equation \eqref{PRRP}. Indeed, after replacing $u$ by $v$, we can rewrite \eqref{UR:1} with $Y=J(X)$ as $\mcR(v)\,\tau_{0,v\zeta^{-1}}\Delta^\prime(J(X))=(\tau_{0,v\zeta^{-1}}\Delta(J(X)))\,\mcR(v)$; applying $\varrho^V_{u\zeta^{-1}}\otimes \varrho^W_0$ to both sides of this equality, we obtain, by \eqref{RR=I}, that the left-hand side equals
\[
 (\varrho^V_{u\zeta^{-1}}\otimes\varrho^W_0)(\mcR(v))\,(\varrho^V_{u\zeta^{-1}}\otimes\varrho^W_0)(\tau_{0,v\zeta^{-1}}\Delta^\prime(J(X))) = (\varrho^V\ot \varrho^W)(\mcR(v-u))\,(\varrho^V_{u\zeta^{-1}}\ot\varrho^W_{v\zeta^{-1}})(\Delta^\prime(J(X))).
\]
Similarly, the right-hand side becomes $(\varrho^V_{u\zeta^{-1}}\otimes \varrho^W_{v\zeta^{-1}}) (\Delta(J(X)))(\mcR(v-u))$. 

 To show that the solution space of \eqref{PRRP} is in fact equal to $\C[[u^{-1}]]R^{VW}(u)$, we will proceed in a few steps, beginning with a proof of the theorem in the case where $V$ and $W$ are assumed to be fundamental representations. 

\noindent \textit{Step 1}: The statement of the theorem holds whenever $V$ and $W$ are fundamental representations of the form $V=V(i_1;0)$ and $W=V(i_2;0)$ with $i_1,i_2\in \mcI$. 

Fix $V$ and $W$ of this form and suppose that  $\ms{R}(u)\in \End(V\ot W)[[u^{-1}]]$ is any solution of \eqref{PRRP}.  Since $\mathcal{R}(u)$ has constant term $1$, $R^{VW}(u)^{-1}$ is also an element of $\End(V\ot W)[[u^{-1}]]$. As $\ms{R}(u)$ and $R^{VW}(u)$ are both solutions of \eqref{PRRP}, we have that 
\begin{equation*}
(\varrho^V_{u\zeta^{-1}}\otimes \varrho^W_{v\zeta^{-1}}) (\Delta(J(X)))\,\ms{R}(u-v)\,R^{VW}(u-v)^{-1} =  \ms{R}(u-v)\,R^{VW}(u-v)^{-1}(\varrho^V_{u\zeta^{-1}}\otimes\varrho^W_{v\zeta^{-1}}) (\Delta(J(X))) \qu\text{for all}\qu X \in \mfg.  
\end{equation*}
To be able to conclude that $\ms{R}(u-v)\,R^{VW}(u-v)^{-1}$ is a formal power series, it is thus enough to prove the following claim:  

\noindent \textit{Claim:} any $\mfI(u) \in \End(V \ot W)[[u^{-1}]]$ satisfying the relation 
\begin{equation}
(\varrho^V_{u\zeta^{-1}}\otimes \varrho^W_{v\zeta^{-1}}) (\Delta(J(X)))\,\mfI(u-v)=  \mfI(u-v) (\varrho^V_{u\zeta^{-1}}\otimes \varrho^W_{v\zeta^{-1}}) (\Delta(J(X))) \label{PIIP}
\end{equation}
for all $X \in \mfg$ is of the form $f(u)\cdot \mathrm{id}_{V\ot W}$ for some $f(u)\in \C[[u^{-1}]]$. 
\begin{proof}[Proof of claim] \let\qed\relax
Suppose that $\mfI(u)$ is as in the statement of the claim. 
By the same argument as given in Remark \ref{RX=XR}, the equality \eqref{PIIP} also holds if we replace $J(X)$ by $X$, and hence $\mfI(u-v)$ intertwines the representation $\varrho^V_{u\zeta^{-1}}\otimes \varrho^W_{v\zeta^{-1}}$. Equation \eqref{PIIP} is thus equivalent (after setting $w=u-v$) to 
\begin{equation}
(\varrho^V_{w\zeta^{-1}}\otimes \varrho^W_{0}) (\Delta(J(X)))\,\mfI(w)=  \mfI(w) (\varrho^V_{w\zeta^{-1}}\otimes \varrho^W_{0}) (\Delta(J(X)))\qu\text{for all}\qu X \in \mfg. \label{PIIP2}
\end{equation}
After choosing bases of $V$ and $W$, we can view $\mfI(w)$ as a matrix whose entries are unknown variables and, after restricting $X$ to a basis of $\mfg$, the previous equation becomes equivalent to a finite system of linear equations where the variables are the entries of $\mfI(w)$ and the coefficients are polynomials of degree $\le 1$ in $w$. The space of solutions is thus of the form 
\begin{equation} \label{Sol}
 S=\left\{\mfI(w)=\msum_{\ell=1}^m \mfI_{\ell}(w)\, f_{\ell}(w)\,:\, f_{\ell}(w)\in \C[[w^{-1}]] \; \text{ for all }\; 1\leq \ell\leq m\right\}, 
\end{equation}
where $m\geq 0$ is a fixed integer and $\mfI_1(w),\ldots,\mfI_\ell(w)$ are rational intertwiners which are linearly independent over $\C[[w^{-1}]]$. As any element of $\C[[u^{-1}]]\cdot\mathrm{id}_{V\ot W}$ provides a solution of \eqref{PIIP2}, the integer $m$ must be at least one. If $m$ is exactly one, then 
it will follow that $S=\C[[u^{-1}]]\cdot\mathrm{id}_{V\ot W}$, concluding the proof of the claim. To see that this is indeed the case, it is enough to show that any rational solution 
$\mfI(w)$ of \eqref{PIIP2} belongs to $\C(w)\cdot \mathrm{id}_{V\ot W}$. 
Let $\mfI(w)$ be a rational solution of \eqref{PIIP2}. Then for any $z\in\C$ which is not a pole of $\mfI(w)$, we obtain an intertwiner $\mfI(z): V_{z\zeta^{-1}} \ot W \lra V_{z \zeta^{-1}} \ot W$. 
The tensor product $V_{z\zeta^{-1}} \ot W $ is an irreducible representation of $Y_\zeta(\mfg)$ for all but finitely many values of $z\in \C$: this follows from the results proven in \cite{GuTa} and \cite{Ta}, but was certainly known a long time ago to experts. It follows by Schur's Lemma that $\mfI(z)$ is a scalar multiple of $\mathrm{id}_{V\otimes W}$ for generic values of $z$. This shows that $\mfI(w)$ must be a multiple of the identity by a rational function in $w$, concluding the proof of the claim. 
\end{proof}
As a consequence of the claim we can deduce that the first statement of the theorem holds  when  $V=V(i_1;0)$ and $W=W(i_2;0)$ with $i_1,i_2\in \mcI$. 

To finish the proof of Step 1, it remains to be determined that there is $f(u)\in 1+u^{-1}\C[[u^{-1}]]$ such that $f(u)\,R^{VW}(u)\in \End(V\otimes W)\otimes \C(u)$. The same argument that led us to conclude that the solution space $S$ of \eqref{PIIP2} took the form \eqref{Sol} allows us to conclude that the solution space of \eqref{PRRP} is the $\C[[u^{-1}]]$-linear span of finitely many independent rational solutions $\msR_1(u),\ldots,\msR_k(u)$ (with $k\geq 1$). Conversely, as a consequence of the first part of Step 1  we must have $k=1$ and $\msR_1(u)=h(u)R^{VW}(u)$ for some $h(u)\in \C[[u^{-1}]]$. Letting $z\in\C^{\times}$ and $m\in\Z_{\ge 0}$ be such that $f(u) = zu^mh(u) \in 1 + u^{-1}\C[[u^{-1}]]$, we obtain the desired result.

We now turn towards proving the theorem for arbitrary finite-dimensional irreducible modules $V$ and $W$. 

\noindent \textit{Step 2}: Let $V$ and $W$ be any two finite-dimensional irreducible modules of $Y_\zeta(\mfg)$. Then $V_{\zeta^{-1}z}\otimes W$ is irreducible 
for all but finitely many values of $z\in \C$. 

A proof of this statement for quantum loop algebras was given in Corollary 2.5 of \cite{AkKa}, pending the proof of Conjecture 1 therein (which was obtained in \cite{Kas}). For Yangians, the analogue of \cite[Conjecture 1]{AkKa} is proven in \cite{GuTa} and \cite{Ta}, following earlier work on tensor products of representations of the Yangian of $\mfgl_n$ by A. Molev \cite{Mo1}, M. Nazarov and V. Tarasov \cite{NaTa}. A careful reading of the proof of \cite[Corollary 2.5]{AkKa} reveals that the same argument will apply in the Yangian setting provided the following fact holds: given $V=V(i_1;0)$ and $W=V(i_2;0)$, there exists an intertwiner $\mfI_{WV}(u-v):W_{v\zeta^{-1}}\otimes V_{u\zeta^{-1}}\to V_{u\zeta^{-1}}\otimes W_{v\zeta^{-1}}$ which is rational in $u-v$. By Step 1, there exists a rational solution $\msR(u)$ of \eqref{PRRP} when $V=V(i_1;0)$ and $W=V(i_2;0)$. If 
$\sigma_{WV}$ denotes the permutation operator $W\otimes V\to V\otimes W$, then it follows from \eqref{PRRP} that $\mfI_{VW}(u-v)=\msR(u-v)\circ \sigma_{WV}$ is an intertwiner of the desired form. 

\noindent \textit{Step 3}: The theorem holds in full generality.

The only barrier to carrying out the argument of Step 1 when $V$ and $W$ are arbitrary finite-dimensional irreducible modules is that it requires the assumption that 
$V_{z\zeta^{-1}}\otimes W$ is irreducible  except at finitely many values of $z\in \C$. By Step~2, this assumption is always satisfied, and thus we may in fact apply the argument of Step~1 to conclude the proof of the theorem. \qedhere
\end{proof}

\begin{remark}\label{R:FrRe}
It was mentioned before the proof of Theorem \ref{RmcR} that the analogous result for the quantum affine algebra $U_q(\hat \mfg)$ was proven in \cite[Theorem 4.2]{FrRe}. It was known to the authors of \textit{loc. cit.} that, given finite-dimensional irreducible modules $V$ and $W$, the modules $V_u\otimes W_v$ and $W_v\otimes V_u$ are irreducible when $u$ and $v$ are regarded as formal variables. Let us assume that this is true for $Y_\zeta(\mfg)$ after extending the base ring appropriately. Then, any nonzero solution $\msR(u)$ of \eqref{PRRP} yields a module homomorphism 
 \begin{equation*}
  \msR(u-v)\circ \sigma_{WV}:W_{v\zeta^{-1}}\otimes V_{u\zeta^{-1}}\to V_{u\zeta^{-1}}\otimes W_{v\zeta^{-1}}.
 \end{equation*}
By Schur's Lemma (as stated, for instance, in Lemma 2.1.3 in \cite{ChGi}), this property uniquely determines  $\msR(u)$ up to multiplication by $f(u)\in \C[[u^{-1}]]$, and hence the solution space of \eqref{PRRP} is equal to $\C[[u^{-1}]]\msR(u)$. This provides a proof of the first statement of Theorem \ref{RmcR}. As we were unable to locate a proof of the irreducibility of the modules $V_u\otimes W_v$ and $W_v\otimes V_u$ in the formal setting, we have taken a different approach relying on \cite{AkKa} as well as the more recent papers \cite{GuTa} and \cite{Ta}. It should be possible to prove the irreducibility of those tensor products following arguments analogous to those used in \cite{KaSo} for quantum affine algebras. Another roundabout way to deduce that $V_u\otimes W_v$ and $W_v\otimes V_u$ are irreducible may be to combine the corresponding result of \cite{KaSo} for quantum affine algebras with the meromorphic tensor equivalence constructed in \cite{GTL3}, taking into account that the coproduct considered in \text{loc. cit.} is related to the standard coproduct by a meromorphic twist - see Subsection 2.13 therein.
\end{remark}

We now return to the specialized setting where $\mfg=\mfg_N$ and $\varrho=\rho$ is the natural representation on $\C^N$. 
\begin{proposition}\label{C:RmcR}
When $V=W=\C^N$, the solutions of \eqref{PRRP} are precisely the elements $\ms{R}_f(u)\in \End(\C^N\ot \C^N)[[u^{-1}]]$ which are of the form 
\begin{equation*}
  \ms{R}_f(u)=f(u)\,R(u) \; \text{ with }\; f(u)\in \C[[u^{-1}]]
\end{equation*}
where $R(u)$ is given by \eqref{R(u)}. In particular,  $(\rho\ot \rho)(\mcR(-u))=h(u)\,R(u)$ where $h(u)\in 1+u^{-1}\C[[u^{-1}]]$ is uniquely determined by the property 
\begin{equation}
 h(u)\,h(u+\zeta\ka)=(1-\zeta^2u^{-2})^{-1}. \label{h-hk}
\end{equation}
\end{proposition}
Note that the second statement of the above proposition was observed in Example 2 of \cite{Dr1} in the case when $\mfg_N=\mfso_N$. 
\begin{proof}
One can verify  that the matrix $R(u)$ from \eqref{R(u)} is a solution of \eqref{PRRP} when $V=W=\C^N$. The first statement of the proposition then becomes an immediate
 consequence of Proposition \ref{RmcR}. However, in this specialized case it is also not too difficult to determine the solution space \eqref{PRRP} directly and this approach has the benefit of allowing one to recover $R(u)$ as a solution of \eqref{PRRP} explicitly. With this in mind, we proceed along this alternate route. 

\noindent \textit{Step 1: } A direct proof that the solution space of \eqref{PRRP} is equal to $\C[[u^{-1}]]R(u)$. 

Suppose first that $N>2$. Let us begin by decomposing $\C^N\ot \C^N$ as a direct sum of irreducible $\mfg_N$-modules and determining the corresponding projection maps. Let $v_Q=\sum_k \theta_{1k}\,e_k\ot e_{-k}$.
 
Consider first the case where $\mfg_N=\mfsp_N$. We have the decomposition $\C^N\ot \C^N=\Lambda^2(\C^N)\oplus \mr{Sym}^2(\C^N)$, where $\mr{Sym}^2(\C^N)$ 
 is an irreducible $\mfsp_N$-module, but $\Lambda^2(\C^N)$ is not. Indeed, $V_{\rm tr}=\C v_Q\subset \Lambda^2(\C^N)$ is isomorphic to the trivial representation of $\mfsp_N$, and $\frac{1}{N}Q$ is the projection operator of $\C^N\ot \C^N$ onto $V_{\rm tr}$. Let $V_{\Lambda}=\mr{Ker}(Q)\cap \Lambda^2(\C^N)$. Then  $V_\Lambda$
 is an irreducible $\mfsp_N$-module with the highest weight vector $\frac{1}{2}(e_{-n}\ot e_{-n+1}-e_{-n+1}\ot e_{-n})=e_{-n}\wedge e_{-n+1}$, and we have the decomposition 
 $\C^N\ot \C^N=V_{\rm tr}\oplus V_{\Lambda}\oplus  \mr{Sym}^2(\C^N)$. The corresponding projection maps are $\pi_1^-=\tfrac{1}{N}Q$, $\pi_2^-=\tfrac{1}{2}(I-P)-\tfrac{1}{N}Q$, and $\pi_3^-=\tfrac{1}{2}(I+P)$, respectively.
 
Suppose now that $\mfg_N=\mfso_N$. Then, we have $\C^N\ot \C^N=\Lambda^2(\C^N)\oplus \mr{Sym}^2(\C^N)$, where $\Lambda^2(\C^N)$ is irreducible, but $\mr{Sym}^2(\C^N)$ decomposes as the direct sum $V_{\mr{Sym}}\oplus V_{\rm tr}$. Here $V_{\rm tr}=\C v_Q$ is isomorphic to the trivial representation, and the projection of $\C^N\ot \C^N$ onto $V_{\rm tr}$ is again given by $\pi_1^+=\frac{1}{N}Q$. The $\mfso_N$-module $V_{\mr{Sym}}$ is the submodule of $\mr{Sym}^2(\C^N)$ defined by $V_{\mr{Sym}}=\mr{Sym}^2(\C^N)\cap \mr{Ker}(Q)$, and the corresponding projection map is $\pi_2^+=\tfrac{1}{2}(I+P)-\frac{1}{N}Q$. Finally, the projection of $\C^N\ot \C^N$ onto $\Lambda^2(\C^N)$ is given by $\pi_3^+=\tfrac{1}{2}(I-P)$. 
  
Next, assume that $\ms{R}(u)\in \End(\C^N\ot \C^N)[[u^{-1}]]$ is any solution of \eqref{PRRP}. Then, by Remark \ref{RX=XR}, $\ms{R}(u)$ intertwines the action of $\mfg_N$, and thus we can express it as a $\C[[u^{-1}]]$-linear combination of 
 the projection operators $\pi_1^{(\pm)},\pi_2^{(\pm)}$ and $\pi_3^{(\pm)}$. As a consequence of the definition of these projection operators, this means that there exists $A(u),B(u)$ and $C(u)$ in $\C[[u^{-1}]]$ such that 
\begin{equation}
  \ms{R}(u)=A(u) I + B(u) P+C(u) Q\in \mr{End}(\C^N\ot \C^N)[[u^{-1}]]. \label{R:A,B,D}
\end{equation}
Since the Casimir element $\Omega$ operates as $P-Q$ on $\C^N\ot \C^N$ (see \eqref{Omega}), we can rewrite \eqref{PRRP} explicitly as
\begin{align}
\big(uX\ot 1 +1\ot& vX+\tfrac{1}{2}\zeta[X\ot 1,P-Q]\big)\,\big(A(u-v) I + B(u-v) P+C(u-v) Q\big) \nonumber\\
  = {} &\big(A(u-v) I + B(u-v) P+C(u-v) Q\big)\,\big(uX\ot 1 +1\ot vX-\tfrac{1}{2}\zeta[X\ot 1,P-Q]\big) \label{R-solve}
 \end{align}
for all $X\in \mfg_N$ (here we identify $X$ with $\rho(X)$). We now proceed to solve \eqref{R-solve} for $A(u-v)$, $B(u-v)$ and $C(u-v)$. Since 
 $P^2=I, Q^2=NQ$ and $PQ=QP=(\pm) Q$, we have the following set of equalities:
\begin{align*}
 [X\ot 1, P-Q]P = {} & (X\ot 1 - 1\ot X) + (\mp)(X\ot 1)Q + (\pm)Q(1\ot X),\\
 [X\ot 1,P-Q]Q  = {} & (\pm)(X\ot 1-1\ot X)Q-N(X\ot 1)Q+Q(X\ot 1)Q,\\
 P[X\ot 1, P-Q] = {} & (1\ot X-X\ot 1) + (\mp)(1\ot X)Q + (\pm) Q(X\ot 1),\\
 Q[X\ot 1, P-Q] = {} & (\pm)Q(1\ot X-X\ot 1)-Q(X\ot 1)Q+NQ(X\ot 1).
\end{align*}
Therefore, after expanding \eqref{R-solve} and subtracting the expressions $A(u-v)\big(uX\ot 1 +1\ot vX-\tfrac{1}{2}\zeta[X\ot 1,P-Q]\big)$, $ B(u-v)\tfrac{\zeta}{2}(X\ot 1 - 1\ot X)$ and $C(u-v)\frac{\zeta}{2}Q(X\ot 1)Q$ from both sides, we obtain the equivalent relation 
 \begin{align}\label{eq0}
 \begin{split}
 A(u-v) & \,\zeta\, P\,\big(1\ot X-X\ot 1\big)+A(u-v)\,\zeta\,\big(Q(X\ot 1)-(X\ot 1)Q\big)\\
 &+  B(u-v)\big( uX\ot 1 + 1\ot vX\big) P \, + (\mp) B(u-v)\tfrac{1}{2}\zeta\,\big((X\ot 1)Q- Q(1\ot X) \big)\\
 &+ C(u-v)\big(uX\ot 1 + 1\ot vX) Q + C(u-v)\tfrac{1}{2}\zeta \big((\pm) (X\ot 1-1\ot X)Q-N(X\ot 1)Q \big)
 \\
 = {} &B(u-v)P\big(uX\ot 1 + 1\ot vX)\, + (\pm)B(u-v)\tfrac{1}{2}\zeta\big((1\ot X)Q-Q(X\ot 1) \big)\\
 &+C(u-v) Q \big(uX\ot 1 +1\ot vX \big)-C(u-v)\tfrac{1}{2}\zeta\big((\pm) Q(1\ot X-X\ot 1)+NQ(X\ot 1) \big).
 \end{split}
\end{align}
Since $\Delta(X)\,Q=0=Q\,\Delta(X)$ for all $X\in \mfg_N$, we have the the equalities 
\begin{equation*}
(1\ot X)\,Q=-(X\ot 1)\,Q \; \text{ and }\; Q\,(1\ot X)=-Q\,(X\ot 1) \; \text{ for all }\; X\in \mfg_N.
\end{equation*}
Using these relations repeatedly, together with the identity $PT=PT P^2=\sigma(T)P$ for all $T\in \End(\C^N\ot \C^N)$, we find that \eqref{eq0} can be rewritten in the 
simple form
\begin{align}\label{eq0-new}
\begin{split}
 \big(A(u-v)\zeta+B(u-v)(u-v)\big)&\,\big(X\ot 1- 1\ot X\big)\,P\\
                          = {} & \big( C(u-v)(u-v-\zeta \ka)-A(u-v)\zeta\big)\,\big(Q(X\ot 1)-(X\ot 1)Q\big).
\end{split}                          
\end{align}
Multiplying on the left by the projector $\pi_3^{(\pm)}=\tfrac{1}{2}(I(\mp) P)$, and on the right by the projector $\pi_2^{(\pm)}=\frac{1}{2}(I(\pm) P)-\tfrac{1}{N}Q$, we obtain the relation 
\begin{equation*}
\big(A(u-v)\zeta+B(u-v)(u-v)\big)\,\tfrac12\,\big(I(\mp) P\big)\,\big(X\ot 1- 1\ot X\big)\,\big(\tfrac12(I(\pm) P)-\tfrac1N Q \big)=0.
\end{equation*}
If we apply both sides to, say, $e_{n-1} \ot e_n \,(\pm)\, e_n \ot e_{n-1}$, and take $X= F_{nn}$, we obtain the equality 
\begin{equation*}
  \big(A(u-v)\zeta+B(u-v)(u-v)\big)\,\big(e_{n-1}\ot e_n(\mp) e_n\ot e_{n-1}\big)=0,
\end{equation*}
which implies that $A(u-v)\zeta=B(u-v)\,(v-u)$. Reinserting this into \eqref{eq0-new} and left multiplying by the projector $\tfrac{1}{2}(I(\mp) P)$, we arrive at the equation
\begin{equation*}
 (C(u-v)(u-v-\zeta\ka)-A(u-v)\zeta)\tfrac{1}{2}(I(\mp) P)(X\ot 1)Q=0.
\end{equation*}
Applying both sides to the vector $\tfrac{1}{N}v_Q$ and taking $X=F_{11}$ yields 
\begin{equation*}
 (C(u-v)(u-v-\zeta\ka)-A(u-v)\zeta)(e_1\ot e_{-1}(\mp) e_{-1}\ot e_1)=0,
\end{equation*}
which implies that $C(u-v)(u-v-\zeta\ka)=A(u-v)\zeta$. Hence, we have shown that if $(A(u),B(u),C(u))$ satisfies \eqref{eq0}, then the relations
\begin{equation*}
 A(u-v)\zeta=B(u-v)(v-u)\; \text{ and } \; C(u-v)(u-v-\zeta\ka)=A(u-v)\zeta
\end{equation*}
must hold. Conversely, it follows immediately from \eqref{eq0-new} that if $A(u),B(u)$ and $C(u)$ satisfy the above relations, then they satisfy \eqref{R-solve}. Therefore, we 
have shown that the solutions of \eqref{PRRP} are exactly the elements $\ms{R}(u)\in \End(\C^N\ot \C^N)[[u^{-1}]]$ which are of the form 
\begin{equation*}
 \ms{R}(u)=A(u)\left(I-\frac{\zeta}{u}\,P + \frac{\zeta}{u-\zeta \ka}\,Q \right),\; \text{ with }\; A(u)\in \C[[u^{-1}]].  
 \end{equation*}
If instead $N=2$ ({\it i.e.}~$\mfg_N=\mfsp_2$), we have $\Lambda^2(\C^2)=\C v_Q$, and $P+Q=I$, so we may assume that, for instance, $C(u)=0$ in \eqref{R:A,B,D}. 
Since $Q(X\ot 1)-(X\ot 1)Q=(X\ot 1-1\ot X)P$, we can rewrite \eqref{eq0-new} as 
\begin{equation}
\big(2A(u-v)\zeta+B(u-v)(u-v)\big)\,\big(X\ot 1-1\ot X\big)\,P=0. \label{eq0-new'}
\end{equation}
Applying this to, say, $e_1\ot e_1$, and taking $X=F_{-1,1}$, we obtain $2A(u-v)\zeta=B(u-v)(v-u)$, and any pair $(A(u),B(u))$ satisfying 
this relation will solve \eqref{eq0-new'}. Hence, we obtain that $\ms{R}(u)$ is a solution of \eqref{PRRP} if and only if it is of the form 
\begin{equation*}
 \ms{R}(u)=A(u)\left(I-\frac{2\zeta}{u}\,P\right)=A(u)\frac{u-2\zeta}{u-\zeta}\left(I-\frac{\zeta}{u}\,P+\frac{\zeta}{u-\zeta \ka}\,Q  \right),\; \text{ with }\; A(u)\in \C[[u^{-1}]].  \qedhere
\end{equation*} 
We thus recover (4.21) from \cite{GRW1}.
Now let us turn to proving the second statement of the proposition. 

\noindent \textit{Step 2: } $(\rho\otimes \rho)(\mcR(-u))=h(u)R(u)$ where $h(u)\in 1+u^{-1}\C[[u^{-1}]]$ is uniquely determined by \eqref{h-hk}. 

Set $\mcR_\rho(u)=(\rho\ot \rho)(\mcR(-u))$.  Since $\mcR_\rho(u)$  satisfies \eqref{PRRP}, there is $h(u)\in \C[[u^{-1}]]$ such that $\mcR_\rho(u)=h(u)\,R(u)$. As $R(u)$ and $\mcR_\rho(u)$ belong to  $1+u^{-1}\End(\C^N\otimes \C^N)[[u^{-1}]]$, $h(u)$ belongs to $1+u^{-1}\C[[u^{-1}]]$. It will be proven in the proof of Theorem \ref{T:YR-iso} that 
$\mcR_\rho(-u-\zeta\ka)^t=\mcR_\rho(-u)^{-1}$. Conversely, $R(u)$ satisfies $R(u)R(u+\zeta\ka)^t=R(u)R(-u)=1-\zeta^2 u^{-2}$ (see (2.23) in \cite{AMR}). Therefore, from the equality $1=\mcR_\rho(-u)\mcR_\rho(-u-\zeta\ka)^t=h(u)h(u+\zeta\ka)R(u)R(u+\zeta\ka)^t$ we obtain \eqref{h-hk}. As there is a unique series in $1+u^{-1}\C[[u^{-1}]]$ satisfying \eqref{h-hk}, this relation uniquely determines $h(u)$. 
\end{proof}

\begin{remark} 
The appearance of the scalar $\ka$ in $R(u)$ can be explained by the fact that the adjoint representation $\ad(\mfg_N)$ appears as a summand of $\C^N\ot \C^N$ (note that this is not the case for $\mfsl_N$, where we have $R(u) = I - {\zeta}u^{-1}P$ instead) and that the Casimir element operates as $4\ka$ in this representation (see the line below \eqref{Omega}).
\end{remark}
 

\subsection{Equivalence between the \texorpdfstring{$J$}{J} and $RTT$ presentations of the Yangian}
\hspace{0mm} In this subsection, we prove Theorem 6 of \cite{Dr1} when $\mfg=\mfg_N$, which provides an isomorphism of Hopf algebras between 
$Y_\zeta^R(\mfg_N)$ and $Y_\zeta(\mfg_N)$ (see Theorem \ref{T:YR-iso}). Our proof will employ the following lemma:

\begin{lemma} \label{L:dual}
For any $X \in Y_\zeta(\mfg_N)$ we have that
\eq{
\rho_s ( S^{\pm1}(X) ) = (\rho_{s\mp\ka}(X))^t  . \label{dual}
}
\end{lemma}

\begin{proof}
It is enough to demonstrate \eqref{dual} for all $F_{ij}$ and all $J(F_{ij})$. This is immediate for $F_{ij}$, since $S(F_{ij})= - F_{ij}$ and $F_{ij}^t = - F_{ij}$. For $J(F_{ij})$, we have $\rho_s(S^{\pm1}(J(F_{ij}))) = \rho_s(-J(F_{ij}) \pm\zeta \ka F_{ij}) = - \zeta\, (s \mp\ka) \, F_{ij}$, which coincides with $(\rho_{s\mp\ka}(J(F_{ij})))^t$. 
\end{proof}

\begin{theorem}[\cite{Dr1}, Theorem 6] \label{T:YR-iso}
The map 
\eq{
\varphi_{R,J} \;:\; Y^R_\zeta(\mfg_N) \to Y_\zeta(\mfg_N) , \qu T(u) \mapsto (\rho\ot id)(\mcR(-u)) \notag
}
is an isomorphism of Hopf algebras. 
\end{theorem}
The original formulation of this result, which appeared in \cite{Dr1} without proof, states that there exists a surjective Hopf algebra homomorphism $X_\zeta^R(\mfg_N)\onto Y_\zeta(\mfg_N)$ whose kernel is generated by the coefficients of a distinguished central series $z(u)=1+\sum_{r=1} z_r u^{-r}$ satisfying $\wt T(u)\,\wt T(u+\zeta\ka)^t=z(u)\cdot I$. Here $X_\zeta^R(\mfg_N)$ is the \textit{extended Yangian} of $\mfg_N$ whose definition can be obtained from that of $Y_\zeta^R(\mfg_N)$ (Definition \ref{D:Y(g)-RTT}) by omitting the relation \eqref{unitary}, and $\wt T(u)$ denotes its generating matrix: see \cite{AMR}.

\begin{proof}
We know that the universal $R$-matrix $\mcR(u)$ defined by Theorem \ref{uniR} satisfies the universal quantum Yang-Baxter equation \eqref{UYBE}. Replacing $u$ and $v$ by $-u$ and $-v$, and then applying $\rho \ot \rho \ot id$ to it gives, combined with Proposition~\ref{C:RmcR},
\[ 
R_{12}(u-v)\, \varphi_{R,J}(T(u))_{13} \, \varphi_{R,J}(T(v))_{23} = \varphi_{R,J}(T(v))_{23} \, \varphi_{R,J}(T(u))_{13}\, R_{12}(u-v) , 
\] 
which shows that $\varphi_{R,J}$ respects the defining relation \eqref{RTT}. 
We also know that $(S^{-1} \ot id)(\mcR(-u)) = \mcR(-u)^{-1}$. (This is a consequence of $( id \ot S)(\mcR(u)) = \mcR(u)^{-1}$ and \eqref{RR=I}.) Left multiplying both
sides of this equality by $\mcR(-u)$ and applying $\rho \ot id$ gives, combined with Lemma \ref{L:dual}, $\varphi_{R,J}(T(u))\,\varphi_{R,J}(T(u+\zeta \ka))^t  = I$. This shows that $\varphi_{R,J}$ respects \eqref{TT=TT=I}. 

Next, we demonstrate the surjectivity of $\varphi_{R,J}$. Using \eqref{Omega} we rewrite \eqref{lnR} for $\mfg=\mfg_N$ as
\eqa{
\mcR(u) = 1 &+ \tfrac12 \zeta u^{-1} \msum_{i,j} F_{ij} \ot F_{ji} + \tfrac12 \zeta u^{-2}  \msum_{i,j} (J(F_{ij}) \ot F_{ji} - F_{ij} \ot J(F_{ji})) + \tfrac18 \zeta^2 u^{-2} \, \Big(\msum_{i,j} F_{ij} \ot F_{ji} \Big)^2 + O(u^{-3}) . \label{R(u)-series}
}
Notice that
\eqn{
(\rho\ot\id)\,\Big( \msum_{i,j} F_{ij} \ot F_{ji} \Big) = {} & \msum_{i,j} (E_{ij} - \theta_{ij}E_{-j,-i}) \ot F_{ji} = 2\msum_{i,j} E_{ij} \ot F_{ji} \\
\intertext{and}
(\rho \ot \id)\,\Big(\msum_{i,j} F_{ij} \ot F_{ji} \Big)^2 = {} & \Big( 2\msum_{i,j} E_{ij} \ot F_{ji} \Big)^2 = 4 \msum_{i,j,k} E_{ij} \ot F_{ki} F_{jk}.
}
Thus, upon substituting $u\mapsto -u$ in \eqref{R(u)-series} and then applying $(\rho \ot \id)$ we find that $\varphi_{R,J}(t_{ij}^{(1)}) = -F_{ji}$ and $\varphi_{R,J}(t_{ij}^{(2)}) = -  \zeta^{-1}J(F_{ji}) + \tfrac12 \sum_{k} F_{ki} F_{jk}$, which follow from the definition of $T(u)$ in Proposition~\ref{P:RTT} and Proposition \ref{P:Y(g)-rep}. Since $F_{ij}$ and $J(F_{ij})$ generate the whole Yangian $Y_\zeta(\mfg_N)$, the map $\varphi_{R,J}$ is surjective.

It remains to show that it is injective.  We have filtrations on both $Y_\zeta(\mfg_N)$ and $Y^R_\zeta(\mfg_N)$ given by $\deg(F_{ij})=0$, $\deg(J(F_{ij}))=1$ and $\deg(t^{(r)}_{ij})=r-1$, respectively, with respect to which $\varphi_{R,J}$ becomes a filtered homomorphism (because $t_{ij}^{(1)}$ and $t_{ij}^{(2)}$ for all $i,j$ generate $Y^R_\zeta(\mfg_N)$). We can thus consider the associated graded homomorphism $\gr \varphi_{R,J}$ which is given by:
\eq{
\gr \varphi_{R,J} \;:\; \gr Y^R_\zeta(\mfg_N) \to \gr Y_\zeta(\mfg_N), \qu 
\bar t^{\,(1)}_{ij} \mapsto -\overline{F}_{ji} , \qu
\bar t^{\,(2)}_{ij} \mapsto  -\zeta^{-1}\overline{J(F_{ji})} . 
}
By \eqref{iso:grYR-Ug} and Proposition \ref{grYzeta}, we already know that the graded algebras $\gr Y^R_\zeta(\mfg_N)$ and $\gr Y_\zeta(\mfg_N)$ are both  isomorphic to $\mfU\mfg_N[s]$. 
After identifying both associated graded algebras with $\mfU \mfg[s]$, $\gr \varphi_{R,J}$ becomes the automorphism $F_{ij} \ot s^r \mapsto -F_{ji} \ot s^r$ of $\mfU\mfg[s]$. Thus, $\varphi_{R,J}$ is also an isomorphism of algebras.

To complete the proof that $\varphi_{R,J}$ is an isomorphism of Hopf algebras, we need to show that it is a coalgebra homomorphism commuting with the antipodes of $Y_\zeta(\mfg_N)$ and $Y_\zeta^R(\mfg_N)$. First note that
\eqn{
(\varphi_{R,J} \ot \varphi_{R,J})(\Delta(T(u)) = {} & \msum_{i,j,k} E_{ij} \ot \varphi_{R,J}(t_{ik}(u)) \ot \varphi_{R,J}(t_{kj}(u)) \\ = {} & (\rho\ot id\ot id)\big(\mcR_{12}(-u)\,\mcR_{13}(-u) \big) =(id\ot\Delta)((\rho\ot id)\,\mcR(-u))=\Delta(\varphi_{R,J}(T(u))),  }
and hence $\varphi_{R,J} \ot \varphi_{R,J} \circ \Delta=\Delta\circ \varphi_{R,J}$. In addition, it follows from a standard argument that $(id\ot\eps)\,\mcR(u)=1$ (see \cite[Theorem VIII.2.4]{Ka}), and hence that 
$\eps(\varphi_{R,J}(T(u)))=I$. Since $\eps(T(u))=I$, we can conclude that $\varphi_{R,J}$ is indeed a homomorphism of bialgebras. Finally, since 
\eqn{
\varphi_{R,J} (S(T(u))) = \varphi_{R,J}(T^{-1}(u)) = (\rho\ot id)(\mcR^{-1}(-u)) = {} & (\rho\ot id)(id\ot S)(\mcR(-u))=S(\varphi_{R,J}(T(u))),  
}
it is in fact a Hopf algebra isomorphism. (This also follows from the fact that a bialgebra isomorphism is automatically a Hopf algebra isomorphism.)
\end{proof}

We will end this subsection by showing that there exists an automorphism $\varkappa$ of $Y_\zeta(\mfg_N)$ such that the composition of $\varphi_{R,J}:Y_\zeta^R(\mfg_N)\to Y_\zeta(\mfg_N)$
with $\varkappa$ yields an isomorphism $\Phi_{R,J}$ which equals the identity map when restricted to $\mfU\mfg_N$. 
In fact, we may take $\varkappa$ to be the extension of the Chevalley involution of $\mathfrak{Ug}_N$ (which is also denoted $\varkappa$ -- see \eqref{varkappa0})  to $Y_\zeta(\mfg_N)$ furnished by Lemma \ref{L:ext-aut}:
\begin{proposition}\label{P:vark}
 Let $\varkappa$ denote the extension of the Chevalley involution to $Y_\zeta(\mfg_N)$ provided by Lemma \ref{L:ext-aut}. Then the isomorphism 
 \begin{equation*}
  \Phi_{R,J}=\varkappa\circ \varphi_{R,J}:Y_\zeta^R(\mfg_N)\to Y_\zeta(\mfg)
 \end{equation*}
is the identity map when restricted to $\mfU\mfg_N$. 
\end{proposition}
\begin{proof}
 Comparing \eqref{varkappa0} with the relations \eqref{x->I-1} through \eqref{x->I-4} we deduce that $\varkappa(F_{ij})=-F_{ji}$ for all ${-}n\leq i,j\leq n$. This implies that 
  $\Phi_{R,J}$ is given on $t_{ij}^{(1)}$ and $t_{ij}^{(2)}$ by the formulas 
\begin{equation*}
 t_{ij}^{(1)}\mapsto F_{ij}\quad \text{  and }\quad t_{ij}^{(2)}\mapsto \zeta^{-1}J(F_{ij})+\tfrac{1}{2}\msum_{k} F_{ik}F_{kj}\quad \text{ for all }\;{-}n\leq i,j\leq n, 
\end{equation*}
from which the proposition follows. 
\end{proof}


\subsection{An approach independent of the universal \texorpdfstring{$R$}{R}-matrix}\label{indapp}

Our proof that $\varphi_{R,J}$, and thus $\Phi_{R,J}$, is an isomorphism relies on Drinfeld's remarkable result that there exists $\mcR(u)\in 1+u^{-1}(Y_\zeta(\mfg_N)\ot Y_\zeta(\mfg_N))[[u^{-1}]]$ satisfying all of the properties listed in Theorem \ref{uniR}.  The purpose of this section is to give a more direct and conceptually more rudimentary proof that $\Phi_{J,R}=\Phi_{R,J}^{-1}$ is a Hopf algebra isomorphism which does not rely on $\mcR(u)$. In order to do this, we will first make use of the isomorphism of Theorem \ref{T:Ycr(g)-} together with the presentation of $Y_\zeta^{\mr{cr}}(\mfg_N)$ given in Theorem \ref{smallerset} to obtain a realization of $Y_\zeta(\mfg_N)$ which better suits our goal. 

\begin{proposition}\label{L:J->R}
Let $\mfg_N$ be a simple Lie algebra of orthogonal or symplectic type not isomorphic to $\mfsl_2$. Then $Y_\zeta(\mfg_N)$ is generated by $\{F_{ij},J(F_{ij})\}_{-n\le i,j\le n}$ subject only to the defining relations 
\begin{gather}
 [F_{ij},F_{kl}]=\delta_{jk}F_{il}-\delta_{il}F_{kj}+\theta_{ij}\delta_{j,-l}F_{k,-i}-\theta_{ij}\delta_{i,-k}F_{-j,l},\quad F_{ij}+\theta_{ij}F_{-j,-i}=0, \label{Jred:1}\\
 [F_{ij},J(F_{kl})]=\delta_{jk}J(F_{il})-\delta_{il}J(F_{kj})+\theta_{ij}\delta_{j,-l}J(F_{k,-i})-\theta_{ij}\delta_{i,-k}J(F_{-j,l})=[J(F_{ij}),F_{kl}], \label{Jred:2}\\
  J(F_{ij})+\theta_{ij}J(F_{-j,-i})=0,\label{Jred:2.5}
 \end{gather}
 for all ${-}n\leq i,j\leq n$, together with 
 \begin{equation}
 4\zeta^{-2}[J(F_{ii}),J(F_{jj})]=\msum_{a,b}\big[ F_{ia}F_{ai},F_{jb}F_{bj}\big]+2\msum_{a}\big(F_{ja}F_{a,-i}F_{-i,j}-F_{j,-i}F_{-i,a}F_{aj}\big) \quad \forall \quad 1\leq i \neq j\leq n.\label{Jred:3}
\end{equation}
\end{proposition}
\begin{proof}
 It is immediate from the relations \eqref{J0} and \eqref{J1} that \eqref{Jred:1}, \eqref{Jred:2} and \eqref{Jred:2.5} are satisfied in $Y_\zeta(\mfg_N)$ for all ${-}n\leq i,j\leq n$. On the other hand, it is not difficult to see that \eqref{Jred:1}--\eqref{Jred:2.5} imply \eqref{J0} and \eqref{J1}. 
One obtains the definition of $J(X)$ for arbitrary $X\in \mfg_N$ in the natural way: first write $X=\sum_{i+j(\geq) 0}a_{ij}F_{ij}$ and then set $J(X)=\sum_{i+j(\geq) 0}a_{ij}J(F_{ij})$.

In Step 1.2 of the proof of Theorem \ref{T:Ycr(g)-}, we argued that, given \eqref{J0} and \eqref{J1}, the relation obtained by applying $\Phi_{\mr{cr},J}$ to \eqref{Le1} is equivalent to 
\begin{equation}
 [J(h_i),J(h_j)]=\zeta^2[v_j,v_i] \quad \text{ for all }\; i,j\in \mcI. \label{JF.1}
\end{equation}
Therefore, \eqref{Jred:1}, \eqref{Jred:2}, \eqref{Jred:2.5} and \eqref{JF.1} provide a set of defining relations for $Y_\zeta(\mfg_N)$. To complete the proof, we need to rewrite \eqref{JF.1} in terms of the basis $\{F_{ii}\}_{i=1}^n$ of the Cartan subalgebra $\mfh_N$ of $\mfg_N$, and then show that the resulting relations can be expressed as in \eqref{Jred:3}.

\medskip 

\noindent \textit{Claim:} The set of relations \eqref{JF.1} are equivalent to
\begin{equation}
16\zeta^{-2}[J(F_{ii}),J(F_{jj})]=\msum_{k< l,\;r<s}\big[F_{sr}[F_{jj},F_{rs}],F_{lk}[F_{ii},F_{kl}] \big] \quad \text{ for all }\; 1\leq i,j\leq n. \label{JF:cl}
\end{equation}
\begin{proof}[Proof of claim]
The relations of the claim are linear in both $F_{jj}$ and $F_{ii}$. Hence, we can rewrite them equivalently in terms of the basis $\{h_k\}_{k\in \mcI}$ (see \eqref{x->I-1} - \eqref{x->I-4}) as 
\begin{equation}
 16\zeta^{-2}[J(h_i),J(h_j)]=\msum_{k< l,\;r<s}\big[F_{sr}[h_j,F_{rs}],F_{lk}[h_i,F_{kl}] \big]. \label{JF:cl2}
\end{equation}
Note that $\sum_{r<s,\;r+s<0}F_{sr}[h_j,F_{rs}]=\sum_{r<s,\;r+s<0}F_{-r,-s}[h_j,F_{-s,-r}]=\sum_{r<s,\;r+s>0}F_{sr}[h_j,F_{rs}]$, and therefore 
\begin{equation*}
 \msum_{r<s}F_{sr}[h_j,F_{rs}]=\msum_{1\le s\le n} F_{s,-s}[h_j,F_{-s,s}]+2\msum_{r<s,\;r+s>0}F_{sr}[h_j,F_{rs}].
\end{equation*}
The same relation holds with $(r,s)$ replaced by $(k,l)$ and $h_j$ by $h_i$. For each ${-}n\leq i_1<i_2\leq n$ with $i_1+i_2(\geq) 0$, let $\alpha_{i_1,i_2}$ 
denote the positive root $\mr{sign}(i_1)\,\eps_{|i_1|}-\mr{sign}(i_2)\,\eps_{|i_2|}$. Then we have $x_{\alpha_{i_1,i_2}}^+=(\sqrt{2})^{-\delta_{i_1,-i_2}}F_{i_1,i_2}$ and 
$x_{\alpha_{i_1,i_2}}^-=(\sqrt{2})^{-\delta_{i_1,-i_2}}F_{i_2,i_1}$. This implies that \eqref{JF:cl2}, and thus \eqref{JF:cl}, is equivalent to
\begin{align*}
  4\zeta^{-2}[J(h_i),J(h_j)]= {} &\left[\tfrac{1}{2}\msum_{1\le s\le n} F_{s,-s}[h_j,F_{-s,s}]+\msum_{r<s,\;r+s>0}F_{sr}[h_j,F_{rs}],\; \tfrac{1}{2}\msum_{1\le l\le n} F_{l,-l}[h_i,F_{-l,l}]+\msum_{k<l,\;k+l>0}F_{lk}[h_i,F_{kl}] \right]\\
                            = {} &\msum_{\alpha,\beta\in \Delta_+}\Big[x_{\alpha}^-[h_j,x_\alpha^+],x_{\beta}^-[h_i,x_\beta^+]\Big]=4[v_j,v_i], 
\end{align*}
where the last equality follows from the fact that $v_k=\tfrac{\ka}{2}\,h_k+\tfrac{1}{2}\sum_{\al\in \Delta_+}(\al,\al_k)\,x_\al^-x_\al^+-\tfrac{1}{2}h_k^2$ for all $k\in \mcI$ (see the proof of Corollary \ref{C:J(X)-b}). This completes the proof of the claim. \let\qed\relax
\end{proof}
Our next step is to argue that the right-hand side of \eqref{Jred:3} is equal to the right-hand side of \eqref{JF:cl} up to a factor of~$4$. First note that 
\begin{equation*}
 \msum_{r<s}F_{sr}[F_{kk},F_{rs}]=\msum_{r<s}(\delta_{kr}-\delta_{ks}+\delta_{k,-s}-\delta_{k,-r})F_{sr}F_{rs}=\msum_{s=k+1}^nF_{sk}F_{ks}-\msum_{r=-n}^{k-1}F_{kr}F_{rk} + \msum_{r=-n}^{-k-1}F_{-k,r}F_{r,-k} - \msum_{s=-k+1}^nF_{s,-k}F_{-k,s}.
\end{equation*}
When $k=j$, we expand this as 
\begin{align*}
  \msum_{r<s}F_{sr}[F_{jj},F_{rs}]= {} &-2\msum_{r=-n}^{n}F_{jr}F_{rj}+4\msum_{r=j+1}^{n}F_{jr}F_{rj}+2\,\Big(\msum_{r=j+1}^nF_{rr}-(n-j)F_{jj}\Big)+2F_{jj}^2.
\end{align*}
and when $k=i$, we instead expand it as 
\begin{equation*}
  \msum_{r<s}F_{sr}[F_{ii},F_{rs}]=2\msum_{r=-n}^{n}F_{ir}F_{ri}-4\msum_{r=-n}^{i-1}F_{ir}F_{ri}+2\,\Big(\msum_{r=i+1}^nF_{rr}-(n-i)F_{ii}\Big)-2F_{ii}^2.
\end{equation*}
Thus, using that elements of the form $F_{kr}F_{rk}$ belong to the centralizer of $\mf{Uh}_N$ in $\mf{Ug}_N$, we deduce from \eqref{JF:cl} and the above relations that 
\begin{align}
\label{Jred:3.-1}
4\zeta^{-2}[J(F_{ii}),J(F_{jj})] =\msum_{b,a=-n}^n[F_{ia}F_{ai},F_{jb}F_{bj}]+2\msum_{a=-n}^{j-1}\msum_{r=-n}^{i-1}[F_{ja}F_{aj},F_{ir}F_{ri}]+2\msum_{b=i+1}^n\msum_{r=j+1}^n[F_{jr}F_{rj},F_{ib}F_{bi}].
\end{align}
We now show that $\sum_{r>j,b>i}[F_{jr}F_{rj},F_{ib}F_{bi}]=0$. By (anti)symmetry, it is enough to prove this in the case where $i>j$. For each pair $(r,b)$ such that
$s>j$ and $b>i$, we have 
\begin{align*}
 [F_{jr}F_{rj},F_{ib}F_{bi}]= {} &-\delta_{rb}F_{jr}F_{ij}F_{bi}-\delta_{ri}F_{ji}F_{ib}F_{bj}+\delta_{br}F_{ir}F_{ji}F_{rj} 
                               +\delta_{ir}F_{jb}F_{bi}F_{ij}.
\end{align*}
Taking the sum over all $r>j$ and $b>i$, we obtain 
\begin{align}\label{Jred:3.0}
\begin{split}
 \msum_{r>j,b>i}[F_{jr}F_{rj},F_{ib}F_{bi}]= {} &\msum_{b>i}\left(F_{jb}[F_{bi},F_{ij}]+[F_{ib},F_{ji}]F_{bj}\right)=\msum_{b>i}\left(F_{jb}F_{bj}-F_{jb}F_{bj}\right)=0. 
\end{split}                               
\end{align}
Next, we expand $\sum_{a=-n}^{j-1}\sum_{r=-n}^{i-1}[F_{ja}F_{aj},F_{ir}F_{ri}]$. As before we assume $i>j$. This gives
\begin{align*}
 [F_{ja}F_{aj},F_{ir}F_{ri}]= {} &F_{ja}(-\delta_{ar}F_{ij}+\theta_{aj}\delta_{j,-r}F_{i,-a}-\theta_{aj}\delta_{a,-i}F_{-j,r})F_{ri}
                              + F_{ja}F_{ir}(\delta_{jr}F_{ai}-\theta_{aj}\delta_{a,-r}F_{-j,i})\\
                            &+F_{ir}(\delta_{ar}F_{ji}+\theta_{ja}\delta_{a,-i}F_{r,-j}-\theta_{ja}\delta_{j,-r}F_{-a,i})F_{aj}
                            +(-\delta_{jr}F_{ia}+\theta_{ja}\delta_{a,-r}F_{i,-j})F_{ri}F_{aj}.
\end{align*}
Taking the sum over ${-}n\leq a\leq j-1$ and ${-}n\leq r\leq i-1$, we find that 
\begin{align}\label{Jred:3.1}
\begin{split}
 \msum_{a=-n}^{j-1}\msum_{r=-n}^{i-1}[F_{ja}F_{aj},F_{ir}F_{ri}]=&\msum_{a<j}\theta_{aj}F_{ja}F_{i,-a}F_{-j,i}+\msum_{a<i}\theta_{j,-i}F_{ia}F_{a,-j}F_{-i,j}-\msum_{a=-j+1}^{i-1}\theta_{-a,j}F_{j,-a}F_{ia}F_{-j,i} \\
 &+\msum_{a=-j+1}^{i-1}\theta_{j,-a}F_{i,-j}F_{ai}F_{-a,j}-\msum_{a<i}\theta_{-i,j}F_{j,-i}F_{-j,a}F_{a,i}-\msum_{a<j}\theta_{ja}F_{i,-j}F_{-a,i}F_{aj}. 
\end{split}
 \end{align}
The first term on the first line on the right-hand side of \eqref{Jred:3.1} can be rewritten as follows:
\begin{equation}\label{Jred:3.1a}
 \msum_{a<j}\theta_{aj}F_{ja}F_{i,-a}F_{-j,i}=\msum_aF_{ja}F_{a,-i}F_{-i,j}-\msum_{a=j}^nF_{ja}F_{a,-i}F_{-i,j}.
\end{equation}
For the second and third terms, we have 
\begin{equation}\label{Jred:3.1b}
\begin{gathered}
 \msum_{a<i}\theta_{j,-i}F_{ia}F_{a,-j}F_{-i,j}=\msum_{a=-i+1}^nF_{ja}F_{a,-i}F_{-i,j} +\msum_{a=-i+1}^n[F_{a,-i},F_{ja}]F_{-i,j}, \\ 
 -\msum_{a=-j+1}^{i-1}\theta_{-a,j}F_{j,-a}F_{ia}F_{-j,i}=-\msum_{a=-i+1}^{j-1}F_{j,a}F_{a,-i}F_{-i,j}.
 \end{gathered}
\end{equation}
Adding \eqref{Jred:3.1a} and \eqref{Jred:3.1b}, we obtain that the first line on the right-hand side of \eqref{Jred:3.1} is equal to
\begin{equation*}
 \msum_aF_{ja}F_{a,-i}F_{-i,j}+\msum_{a=-i+1}^n[F_{a,-i},F_{ja}]F_{-i,j}. 
\end{equation*}
Similar computations show that the second line of \eqref{Jred:3.1} is equal to $-\sum_a F_{j,-i}F_{-i,a}F_{a,j}-\sum_{a=-i+1}^nF_{j,-i}[F_{aj},F_{-ia}]$.
Hence, 
\begin{equation*}
 \msum_{a=-n}^{j-1}\msum_{r=-n}^{i-1}[F_{ja}F_{aj},F_{ir}F_{ri}]=\msum_a\left(F_{ja}F_{a,-i}F_{-i,j}-F_{j,-i}F_{-i,a}F_{aj} \right)+\msum_{a=-i+1}^n\left([F_{a,-i},F_{ja}]F_{-i,j}- F_{j,-i}[F_{aj},F_{-ia}]\right).
\end{equation*}
By \eqref{Jred:3.-1} and \eqref{Jred:3.0}, all that remains is to show is that the second summation on the right-hand side vanishes, which is immediate by the bracket relations of $\mfg_N$. 
\end{proof}

\begin{proposition}\label{T:J->R}
 Let $\Phi_{J,R}$ denote the assignment $\{J(F_{ij}),F_{ij}\,:\,{-}n\leq i,j\leq n\}\to Y_\zeta^R(\mfg_N)$ defined by 
 \begin{equation}
  F_{ij}\mapsto t_{ij}^{(1)},\quad J(F_{ij})\mapsto \zeta\left(t_{ij}^{(2)}-\tfrac{1}{2}\msum_k t_{ik}^{(1)}t_{kj}^{(1)}\right). \label{J->R}
 \end{equation}
Then $\Phi_{J,R}$ extends uniquely to an isomorphism $\Phi_{J,R}:Y_\zeta(\mfg_N)\to Y_\zeta^R(\mfg_N)$ of Hopf algebras. 
\end{proposition}
\begin{proof} 
\noindent \textit{Step 1}: $\Phi_{J,R}$ is a homomorphism of algebras. 
\medskip 

To prove this, we will check that $\Phi_{J,R}$ preserves the relations of Proposition \ref{L:J->R}. That is, 
we must show that $\Phi_{J,R}(J(F_{ij}))$ and $\Phi_{J,R}(F_{ij})$ satisfy the relations \eqref{Jred:1}, \eqref{Jred:2}, \eqref{Jred:2.5} and \eqref{Jred:3}. 
Since $F_{ij}\mapsto t_{ij}^{(1)}$ defines an embedding, \eqref{Jred:1} is preserved by $\Phi_{J,R}$. When $\mfg_N=\mfsp_{2n}$, it was established in Lemma 5.29 of \cite{AMR} that the subspace of $Y_\zeta(\mfg_N)$ spanned by the elements $\Phi_{J,R}(J(F_{ij}))$ is isomorphic to the adjoint representation of $\mfg_N$, which shows that, in this case, $\Phi_{J,R}$ preserves \eqref{Jred:2} and \eqref{Jred:2.5}. Aside from the appearance of the complex parameter $\zeta$, the proof in the general setting is the same, but we provide some details nonetheless. Taking the $u^{-1}$ coefficient of both sides of relation \eqref{RTT:BCD}, we obtain
 \begin{equation*}
  \zeta\,[ t_{ij}^{(1)},t_{kl}(v)]=\delta_{jk}\zeta t_{il}(v)-\delta_{il}\zeta t_{kj}(v)+\theta_{ij}\delta_{j,-l}\zeta t_{k,-i}(v)-\theta_{ij}\zeta \delta_{i,-k}t_{-j,l}(v).
 \end{equation*}
 Comparing the $v^{-2}$ terms gives the identity 
 \begin{equation}
  [t_{ij}^{(1)},t_{kl}^{(2)}]=\delta_{jk} t_{il}^{(2)}-\delta_{il} t_{kj}^{(2)}+\theta_{ij}\delta_{j,-l} t_{k,-i}^{(2)}-\theta_{ij} \delta_{i,-k}t_{-j,l}^{(2)}. \label{J->R.1}
 \end{equation}
We also have 
\begin{equation*}
[t_{ij}^{(1)},t_{ka}^{(1)}t_{al}^{(1)}]=(\delta_{jk} t_{ia}^{(1)}-\delta_{ia} t_{kj}^{(1)}+\theta_{ij}\delta_{j,-a}t_{k,-i}^{(1)}-\theta_{ij} \delta_{i,-k}t_{-j,a}^{(1)})\,t_{al}^{(1)}+t_{ka}^{(1)}\,(\delta_{ja} t_{il}-\delta_{il}t_{aj}+\theta_{ij}\delta_{j,-l} t_{a,-i}-\theta_{ij}\delta_{i,-a}t_{-j,l} ).
\end{equation*}
Taking the sum as $a$ varies over ${-}n\leq a\leq n$, we obtain 
\begin{equation*}
\bigg[t_{ij}^{(1)},\msum_at_{ka}^{(1)}t_{al}^{(1)}\bigg]=
\delta_{jk} \msum_a t_{ia}^{(1)}t_{al}^{(1)}-\delta_{il}\msum_a t_{ka}^{(1)}t_{aj}^{(1)}+\theta_{ij}\delta_{j,-l}\msum_a t_{ka}^{(1)}t_{a,-i}^{(1)}-\theta_{ij} \delta_{i,-k}\msum_a t_{-j,a}^{(1)}t_{al}^{(1)}.
\end{equation*}
Combining this with \eqref{J->R.1}, we see that $\Phi_{J,R}$ preserves the relation \eqref{Jred:2}. 

We now turn to \eqref{Jred:2.5}. Taking the $u^{-2}$ coefficient of both sides of the second equality in \eqref{TT=TT=I} and using that $t_{kl}^{(1)}+\theta_{kl}t_{-l,-k}^{(1)}=0$ for all $k$ and $l$, we find that
 \begin{equation}
  t_{ij}^{(2)}=-\theta_{ij}t_{-j,-i}^{(2)}-\ka\, t_{ij}^{(1)}+\msum_at_{ia}^{(1)}t_{aj}^{(1)}. \label{t^2:tr}
 \end{equation}
Thus 
\begin{align*}
 \zeta^{-1}\left(\Phi_{J,R}(J(F_{ij}))+\theta_{ij}\Phi_{J,R}(J(F_{-j,-i})\right)=\tfrac{1}{2}\msum_k\left( t_{ik}^{(1)}t_{kj}^{(1)}-\theta_{ij}t_{-j,-k}^{(1)}t_{-k,-i}^{(1)}\right)-\ka\,t_{ij}^{(1)} =\tfrac{1}{2}\msum_k \Big[t_{ik}^{(1)},t_{kj}^{(1)}\Big]-\ka\, t_{ij}^{(1)}=0.
\end{align*}
 We now turn to proving that $\Phi_{J,R}$ respects the relation \eqref{Jred:3}. We need to see that 
 \begin{equation}\label{ter-chk}
  4\left[t_{ii}^{(2)}-\tfrac{1}{2}\msum_k t_{ik}^{(1)}t_{ki}^{(1)},t_{jj}^{(2)}-\tfrac{1}{2}\msum_k t_{jk}^{(1)}t_{kj}^{(1)}\right]=\msum_{a,b}\left[ t_{ia}^{(1)}t_{ai}^{(1)},t_{jb}^{(1)}t_{bj}^{(1)}\right]+2\msum_a\left(t_{ja}^{(1)}t_{a,-i}^{(1)}t_{-i,j}^{(1)}-t_{j,-i}^{(1)}t_{-i,a}^{(1)}t_{aj}^{(1)}\right). 
 \end{equation}
The left-hand side is equal to 
\begin{equation}\label{J->R.2}
 4\Big[t_{ii}^{(2)},t_{jj}^{(2)}\Big]+2\msum_k \Big[t_{jk}^{(1)}t_{kj}^{(1)},t_{ii}^{(2)}\big]-2\msum_k \Big[t_{ik}^{(1)}t_{ki}^{(1)},t_{jj}^{(2)}\Big]+\msum_{a,b}\Big[ t_{ia}^{(1)}t_{ai}^{(1)},t_{jb}^{(1)}t_{bj}^{(1)}\Big].
\end{equation}
Consider the two middle terms. By \eqref{J->R.1}, for $1\le c \neq d \le n$, we have 
\begin{gather*}
 \Big[t_{dk}^{(1)}t_{kd}^{(1)},t_{cc}^{(2)}\Big]=t_{dk}^{(1)}\Big[t_{kd}^{(1)},t_{cc}^{(2)}\Big]+\Big[t_{dk}^{(1)},t_{cc}^{(2)}\Big]\,t_{kd}^{(1)}=t_{dk}^{(1)}\Big(-\delta_{kc}t_{cd}^{(2)}-\theta_{kd}\delta_{k,-c}t_{-d,c}^{(2)}\Big )+\Big(\delta_{kc}t_{dc}^{(2)}+\theta_{dk}\delta_{k,-c}t_{c,-d}^{(2)}\Big)\,t_{kd}^{(1)}.
\end{gather*}
Taking the sum over all $k$ yields 
\begin{gather*}
 \msum_k \Big[t_{dk}^{(1)}t_{kd}^{(1)},t_{cc}^{(2)}\Big]=-t_{dc}^{(1)}t_{cd}^{(2)}-\theta_{-c,d}t_{d,-c}^{(1)}t_{-d,c}^{(2)} +t_{dc}^{(2)}t_{cd}^{(1)}+\theta_{d,-c}t_{c,-d}^{(2)}t_{-c,d}^{(1)}.
\end{gather*}
Applying again \eqref{J->R.1} we find that $t_{cd}^{(1)}t_{dc}^{(2)}=t_{dc}^{(2)}t_{cd}^{(1)}+t_{cc}^{(2)}-t_{dd}^{(2)}$. Therefore, setting $(c,d)=(i,j)$ and $(c,d)=(j,i)$, we obtain
\begin{equation}
 \msum_k \Big(\Big[t_{jk}^{(1)}t_{kj}^{(1)},t_{ii}^{(2)}\Big]-\Big[t_{ik}^{(1)}t_{ki}^{(1)},t_{jj}^{(2)}\Big]\Big)=-2t_{ji}^{(1)}t_{ij}^{(2)}+2t_{ji}^{(2)}t_{ij}^{(1)}-\theta_{-i,j}t_{j,-i}^{(1)}t_{-j,i}^{(2)}+\theta_{j,-i}t_{i,-j}^{(2)}t_{-i,j}^{(1)}+\theta_{-j,i}t_{i,-j}^{(1)}t_{-i,j}^{(2)}-\theta_{i,-j}t_{j,-i}^{(2)}t_{-j,i}^{(1)}. \label{J->R.3}
\end{equation}
The first term in \eqref{J->R.2} can be computed directly using the defining relation \eqref{RTT:BCD} and is equal to $\big[t_{ii}^{(2)},t_{jj}^{(2)}\big]=t_{ji}^{(1)}t_{ij}^{(2)}-t_{ji}^{(2)}t_{ij}^{(1)}$. Combining this with \eqref{J->R.3}, we can rewrite the first three terms of \eqref{J->R.2} as follows:
\begin{equation}\label{J->R.4}
 4\Big[t_{ii}^{(2)},t_{jj}^{(2)}\Big]+2\msum_k \Big(\Big[t_{jk}^{(1)}t_{kj}^{(1)},t_{ii}^{(2)}\Big]-\big[t_{ik}^{(1)}t_{ki}^{(1)},t_{jj}^{(2)}\Big]\Big)
 =2\left(-\theta_{-i,j}t_{j,-i}^{(1)}t_{-j,i}^{(2)}+\theta_{j,-i}t_{i,-j}^{(2)}t_{-i,j}^{(1)}+\theta_{-j,i}t_{i,-j}^{(1)}t_{-i,j}^{(2)}-\theta_{i,-j}t_{j,-i}^{(2)}t_{-j,i}^{(1)}\right). 
\end{equation}
Consider the second term on the right-hand side of the above equality. Employing \eqref{t^2:tr}, we find that it can be written~as 
\begin{equation*}
 \theta_{j,-i}\,t_{i,-j}^{(2)}t_{-i,j}^{(1)}=-t_{j,-i}^{(2)}t_{-i,j}^{(1)}-\theta_{j,-i}\ka\, t_{i,-j}^{(1)}t_{-i,j}^{(1)}+\theta_{j,-i}\msum_at_{ia}^{(1)}t_{a,-j}^{(1)}t_{-i,j}^{(1)}
\end{equation*}
Similarly, the third term can be expressed as 
\begin{equation*}
 \theta_{-j,i}\,t_{i,-j}^{(1)}t_{-i,j}^{(2)}=-t_{j,-i}^{(1)}t_{-i,j}^{(2)}=\theta_{-i,j}t_{j,-i}^{(1)}t_{-j,i}^{(2)}+\ka\,t_{j,-i}^{(1)}t_{-i,j}^{(1)}-t_{j,-i}^{(1)}\msum_at_{-i,a}^{(1)}t_{aj}^{(1)}. 
\end{equation*}
Substituting these new expressions back into \eqref{J->R.4} and using the identity $\sum_a[t_{-i,a}^{(1)},t_{aj}^{(1)}]=2\ka\, t_{-i,j}^{(1)}$, we arrive at the relation 
\begin{align*}
4\Big[t_{ii}^{(2)},t_{jj}^{(2)}\Big]+2\msum_k \Big(\Big[t_{jk}^{(1)}t_{kj}^{(1)},t_{ii}^{(2)}\Big]-\big[t_{ik}^{(1)}t_{ki}^{(1)},t_{jj}^{(2)}\Big]\Big)
= {} &2\,\bigg(2\ka t_{j,-i}^{(1)}t_{-i,j}^{(1)}-\msum_a \Big(t_{ia}^{(1)}t_{a,-j}^{(1)}t_{-j,i}^{(1)}+t_{j,-i}^{(1)}t_{-i,a}^{(1)}t_{aj}^{(1)} \Big)\bigg)\\
= {} &2\msum_a \Big(-t_{j,-i}^{(1)}t_{a,j}^{(1)}t_{-i,a}^{(1)}+t_{a,-i}^{(1)}t_{j,a}^{(1)}t_{-i,j}^{(1)} \Big).
\end{align*}
A similar calculation shows that
\begin{equation*}
 2\msum_{a} \Big(-t_{j,-i}^{(1)}t_{a,j}^{(1)}t_{-i,a}^{(1)}+t_{a,-i}^{(1)}t_{j,a}^{(1)}t_{-i,j}^{(1)}\Big) = 2\msum_{a}\Big(t_{j,a}^{(1)}t_{a,-i}^{(1)}t_{-i,j}^{(1)}-t_{j,-i}^{(1)}t_{-i,a}^{(1)}t_{a,j}^{(1)}\Big).
\end{equation*}
From this, together with \eqref{J->R.2}, we can conclude that \eqref{ter-chk} is satisfied for all $1\leq i \neq j\leq n$. 
This proves that $\Phi_{J,R}$ is a homomorphism of algebras.

\medskip 

\noindent \textit{Step 2:} $\Phi_{J,R}$ is a morphism of coalgebras.

\medskip 
For the remainder of the proof, we will write $\eps_R,\Delta_R$ and $S_R$ for the counit, coproduct, and antipode, respectively, of the Hopf algebra $Y_\zeta^R(\mfg_N)$. 

Let us start by showing $\Phi_{J,R}$ is a morphism of coalgebras. It is immediate that $\eps_R\circ \Phi_{J,R}=\eps$. To verify that 
$\Delta_R \circ \Phi_{J,R}=(\Phi_{J,R}\ot \Phi_{J,R})\circ \Delta$, we actually just need to check that both sides are equal when applied to $F_{kl}$ and 
$J(F_{ii})$ for all ${-}n\leq k,l\leq n$ and $1\leq i\leq n$, which are generators of $Y_\zeta(\mfg_N)$ and we already know $\Phi_{J,R}$ is a algebra morphism. Since $\iota:\mf{Ug}_N\into Y_\zeta^R(\mfg_N)$, $F_{ij}\mapsto t_{ij}^{(1)}$ is an embedding of Hopf algebras, we just need to verify that
\begin{equation}
 \Delta_R(\Phi_{J,R}(J(F_{ii})))=(\Phi_{J,R}\ot \Phi_{J,R})(\Delta(J(F_{ii}))) \quad \text{ for all }\; 1\leq i\leq n. \label{J->R.co1}
\end{equation}
We have 
\begin{equation}
 (\Phi_{J,R}\ot \Phi_{J,R})(\Delta(J(F_{ii})))=(\Phi_{J,R}\ot \Phi_{J,R})(\square(J(F_{ii}))+\tfrac{\zeta}{2}[F_{ii} \ot 1,\Omega])=
 \square\left(\zeta t_{ii}^{(2)}-\tfrac{\zeta}{2}\msum_{k} t_{ik}^{(1)}t_{ki}^{(1)}\right)+\tfrac{\zeta}{2}[t_{ii}^{(1)} \ot 1,\Omega_R],\label{J->R.co2}
 \end{equation}
where $\Omega_R=(\iota\ot \iota)(\Omega)$ and we recall that, given any vector space $W$, $\square$ is the linear map given by $\square(w)=w\ot 1+1\ot w$ for all $w\in W$. 

Consider now the left-hand side of the equality \eqref{J->R.co1}.  Since $\Delta_R(t_{kl}(u))=\sum_at_{ka}(u)\ot t_{al}(u)$, we have $\Delta_R(t_{kl}^{(2)})=\square(t_{kl}^{(2)})+\sum_a t_{ka}^{(1)}\ot t_{al}^{(1)}$. Therefore, 
\begin{align*}
  \Delta_R(\Phi_{J,R}(J(F_{ii})))= \Delta_R\left( \zeta t_{ii}^{(2)}-\tfrac{\zeta}{2}\msum_{k} t_{ik}^{(1)}t_{ki}^{(1)}\right)= {} &\square(\zeta t_{ii}^{(2)})+\zeta\msum_{a} t_{ia}^{(1)}\ot t_{ai}^{(1)}-\tfrac{\zeta}{2}\msum_{a} \square(t_{ia}^{(1)})\square(t_{ai}^{(1)})\\
  = {} &\square(\zeta t_{ii}^{(2)}-\zeta \msum_{a} t_{ia}^{(1)}t_{ai}^{(1)})+\tfrac{\zeta}{2}\msum_{a} (t_{ia}^{(1)}\ot t_{ai}^{(1)}-t_{ai}^{(1)}\ot t_{ia}^{(1)}).
\end{align*}
In order for this to be equal to \eqref{J->R.co2}, we need to establish that $[t_{ii}^{(1)} \ot 1,\Omega_R]=\sum_{a} (t_{ia}^{(1)}\ot t_{ai}^{(1)}-t_{ai}^{(1)}\ot t_{ia}^{(1)})$, which can be computed directly. This completes the proof that \eqref{J->R.co1} is satisfied, and thus that $\Phi_{J,R}$ is a morphism of coalgebras. 

\medskip 

\noindent \textit{Step 3:} $\Phi_{J,R}$ is an isomorphism of Hopf algebras. 

\medskip 
All that is left to show is that $\Phi_{J,R}$ is a bijection because an isomorphism of bialgebras must be an isomorphism of Hopf algebras. (This is a consequence of the uniqueness of the antipode.) This follows from the same type of argument as used in the proofs of Theorems \ref{T:Ycr(g)-} and 
\ref{T:YR-iso}: $\Phi_{J,R}$ is a filtered morphism and, after identifying both $\gr Y_\zeta(\mfg_N)$ and $\gr Y_\zeta^R(\mfg_N)$ with $\mf{Ug}_N[s]$, $\gr \Phi_{J,R}$ is just the identity map.
\end{proof}


\section{An application to the representation theory of the Yangian of \texorpdfstring{$\mfg_N$}{}}\label{sec:rep}


In this section we set $\zeta=1$, and write $Y^R(\mfg_N)$, $Y(\mfg_N)$ and $Y^{\mr{cr}}(\mfg_N)$ for $Y^R_{\zeta=1}(\mfg_N)$, $Y_{\zeta=1}(\mfg_N)$ and $Y_{\zeta=1}^{\mr{cr}}(\mfg_N)$, respectively. Let $d_i=1$ for $1\le i \le n-1 $, and set $d_0=1$ if $\mfg_N=\mfso_{2n}$, $d_0=2$ if $\mfg_N=\mfsp_{2n}$ and $d_0=\tfrac{1}{2}$ if $\mfg_N=\mfso_{2n+1}$, {\it i.e.}~as in Corollary~\ref{C:J(X)-b}. 
 
Let $\Phi_{\mr{cr},R}: Y^{\mr{cr}}(\mfg_N)\to Y^R(\mfg_N)$ be the Hopf algebra isomorphism obtained by composing $\Phi_{J,R}$ with the isomorphism 
$\Phi_{\mr{cr},J}: Y^{\mr{cr}}(\mfg_N)\to Y(\mfg_N)$ of Theorem \ref{T:Ycr(g)-}. The goal of this section is to prove that $\Phi_{\mr{cr},R}$ is compatible with the two definitions for highest weight module which have appeared independently for $Y^{\mr{cr}}(\mfg_N)$ and $Y^R(\mfg_N)$ in \cite{Dr3} and  \cite{AMR}, respectively, and also to show that $\Phi_{\mr{cr},R}$ induces an explicit equivalence between the two classification theorems for finite-dimensional irreducible modules proven in those two papers. Both of these results constitute Theorem \ref{T:cr-R}.

Let us begin by recalling the  classification of finite-dimensional irreducible modules obtained in \cite{Dr3} (for $Y^{\mr{cr}}(\mfg_N)$) and \cite{AMR} (for $Y^R(\mfg_N)$). 

A $Y^{\mr{cr}}(\mfg_N)$-module $V$ is a highest weight module if it is generated by a vector $\xi$ such that $x_{ir}^+\xi=0$ and $h_{ir}\xi=\lambda_i^r\xi$ for all $i\in \mcI$ and $r\geq 0$, where $(\lambda_i^r)_{i\in \mcI,r\geq 0}$ is a family of complex numbers which is called the highest weight of $V$. Such a vector $\xi$ is called a highest weight vector. By Theorem 2 of \cite{Dr3},  every finite-dimensional irreducible module is a highest weight module, and moreover  the isomorphism classes of these modules are in one-to-one correspondence with tuples of monic polynomials $(Q_i(u))_{i\in \mcI}$, called \textit{Drinfeld polynomials}. The relation between the highest weight $(\lambda_i^r)_{i\in \mcI,r\geq 0}$ of a finite-dimensional irreducible module $V$ and the corresponding tuple $(Q_i(u))_{i\in \mcI}$ is provided by
\begin{equation}
 1+\msum_{r=0}^\infty \lambda_i^r u^{-r-1}=\frac{Q_i(u+d_i)}{Q_i(u)} \quad \text{ for all }\; i\in \mcI. \label{Dr-class}
\end{equation}
Conversely, a $Y^R(\mfg_N)$-module $V$ is a highest weight module if it is generated by a vector $\xi$ such that $t_{kl}(u)\,\xi=0$ for all ${-}n\leq k<l\leq n$ and 
$t_{kk}(u)\,\xi=\lambda_k(u)\, \xi$ for all ${-}n\leq k\leq n$, where $\lambda_k(u)\in 1+u^{-1}\C[[u^{-1}]]$ for each $k$. The vector $\xi$ is called the highest weight vector, 
and the tuple $\lambda(u)=(\lambda_k(u))_{{-}n\leq k\leq n}$ is called the highest weight. Theorem 5.1 of \cite{AMR} implies that every finite-dimensional irreducible $Y^R(\mfg_N)$-module is of highest weight type, and by Theorem~5.16 and Corollary 5.19 of \cite{AMR} the isomorphism classes of such modules are parameterized by tuples of monic polynomials 
$(P_k(u))_{1\leq k\leq n}$, which are also called Drinfeld polynomials. The relationship between the highest weight $\lambda(u)$ and the associated tuple $(P_k(u))_{1\leq k\leq n}$ is given by 
\begin{gather*}
 \frac{\lambda_{k-1}(u)}{\lambda_k(u)}=\frac{P_k(u+1)}{P_k(u)} \; \text{ for all }\; 2\leq k\leq n \qu\text{ and}\qu \frac{\lambda_{k_0}(u)}{\lambda_{k_1}(u)}=\frac{P_1(u+d_0)}{P_1(u)},
\end{gather*}
where $(k_0,k_1)=(0,1)$ if $\mfg_N=\mfso_{2n+1}$, $(k_0,k_1)=(-1,1)$ if $\mfg_N=\mfsp_{2n}$, and $(k_0,k_1)=(-1,2)$ if $\mfg_N=\mfso_{2n}$.

In Definition \ref{Funrep} we gave the definition of the fundamental $Y^{\mr{cr}}(\mfg_N)$-module $V(i;a)$, where $i\in \mcI$ and $a\in \C$. 
We also have the notion of fundamental representation for $Y^R(\mfg_N)$: 
\begin{definition}
Let $a\in \C$ and $1\leq k\leq n$. We denote by $V^R(k;a)$ the unique, up to isomorphism, finite-dimensional irreducible $Y^R(\mfg_N)$ -module with Drinfeld polynomials 
$P_1(u),\ldots,P_n(u)$ given by $P_j(u)=1$ if $j\neq i$ and $P_i(u)=u-a$. 
\end{definition}
If $\xi\in V^R(k;a)$ denotes the highest weight vector, then the $\mfg_N$-module generated by $\xi$ is isomorphic to the fundamental module $V(\omega_{k-1})$.

Given a $Y^R(\mfg_N)$-module $V$ with corresponding morphism $\varrho: Y^R(\mfg_N)\to \End V$, let $\Phi_{\mr{cr},R}^*(V)$ denote the $Y^{\mr{cr}}(\mfg_N)$-module realized in the vector space $V$ with module structure given by $\varrho\circ \Phi_{\mr{cr},R}$. For each $b\in \C$, we denote by $\tau_b^R$ the automorphism of $Y^R(\mfg_N)$ 
given by $T(u)\mapsto T(u-b)$. We also denote by $\tau_b$ the automorphism of $Y^{\mr{cr}}(\mfg_N)$ inherited from the automorphism of $Y(\mfg_N)$ of the same name via the isomorphism $\Phi_{\mr{cr},J}:Y^{\mr{cr}}(\mfg_N)\iso Y(\mfg_N)$: see Proposition~12.1.5 of \cite{ChPr2} for its definition on the generators $h_{ir}$ and $x_{ir}^\pm$. 
The observations that $\tau_{0,b}(\mcR(-u))=\mcR(-u+b)$ (by \eqref{RR=I}) and $\varkappa\circ \tau_b=\tau_b\circ \varkappa$ (see Proposition \eqref{P:vark}), together with Proposition 2.14 of \cite{ChPr1} and the definition of $\tau_b^R$ imply the following lemma.

\begin{lemma}\label{tau}
 For each $b\in \C$ we have $\tau_b^R\circ \Phi_{\mr{cr},R}=\Phi_{\mr{cr},R}\circ \tau_b$. In particular, if $V$ is a finite-dimensional irreducible  $Y^R(\mfg_N)$-module 
 with Drinfeld tuple $(P_i(u))_{i=1}^n$, and $\Phi_{\mr{cr},R}^*(V)$ is attached to the tuple $(Q_i(u))_{i\in \mcI}$, then $V^{\tau_b}$ corresponds to the tuple 
 $(P_i(u-b))_{i=1}^n$ and $\Phi_{\mr{cr},R}^*(V^{\tau_b})$ corresponds to $(Q_i(u-b))_{i\in \mcI}$. 
\end{lemma}

In order to prove Theorem \ref{T:cr-R} stated below, we will need to know more explicit information about certain fundamental representations of $Y^R(\mfg_N)$. 
The majority of the following lemma is a reformulation of Lemma 5.18 of \cite{AMR}: 

\begin{lemma}\label{spin}
Let $i=0$ if $\mfg_N=\mfso_{2n+1}$, and $i=0$ or $1$ if $\mfg_N=\mfso_{2n}$, and let $f(u)\in1+u^{-1}\C[[u^{-1}]]$ be the unique series such that 
\begin{equation}
 f(u)f(u+\ka)=\frac{4u(u+\ka)}{4u(u+\ka)-2\ka-1}. \label{f(u)}
\end{equation}
Then the fundamental representation of $V(\omega_i)$ of $\mfso_N$ can be extended to the Yangian
$Y^R(\mfg_N)$ by assigning $t_{kl}(u)\mapsto f(u)(\delta_{kl}+F_{kl}u^{-1})$ for all ${-}n\leq k,l\leq n$, and the resulting module is isomorphic to $V^R(i+1;\frac{1}{2})$. 

In addition, $\Phi_{J,R}(J(F_{kl}))$ operates on this module as $-\mfrac{\ka}{2}F_{kl}$ for all ${-}n\leq k,l\leq n$.
\end{lemma}

\begin{proof}
In Lemma 5.18 of \cite{AMR} it was shown that $V(\omega_i)$ extends to a representation $V$ of the extended Yangian $X(\mfso_N)$  (defined in the same way as $Y(\mfso_N)$ only without \eqref{TT=TT=I}) by the rule $\wt t_{kl}(u)\mapsto \delta_{kl}+F_{kl}u^{-1}$, and moreover that the resulting module has the Drinfeld tuple $(u-1/2,1,\ldots,1)$ if $i=0$ and $(1,u-1/2,1,\ldots,1)$ if $i=1$. Here the notation $\wt t_{kl}(u)$ is used to denote the generating series of the extended Yangian $X(\mfso_N)$ (these were denoted by $t_{kl}(u)$ in \cite{AMR}). Therefore, to complete the proof of the first part of the lemma we just need to show that the restriction of this $X(\mfso_N)$-module to the subalgebra $Y(\mfso_N)$ is indeed given by the formula  $t_{kl}(u)\mapsto f(u)(\delta_{kl}+F_{kl}u^{-1})$ for all $k$ and $l$. Let $F=\sum_{k,l}E_{kl}\ot F_{kl}\in \End(\C^N)\ot \End(V(\omega_i))$. Then, since $F^t=-F$, we have the equality of operators
 \begin{equation*}
  \wt T(u)\, \wt T^t(u+\ka)=z(u)\,I= I+\frac{F^t}{u+\ka} +\frac{F}{u} +\frac{FF^t}{u\,(u+\ka)}= 1+\frac{\ka F-F^2}{u\,(u+\ka)}, 
 \end{equation*}
where $z(u)$ is the distinguished central series of $X(\mfso_N)$ defined in (2.26) of \cite{AMR}. By (5.44) of \cite{AMR}, 
\begin{equation}
 F^2=\left(\mfrac{\ka}{2}+\mfrac{1}{4} \right)I+\ka F \label{spin.2},
\end{equation}
and hence $z(u)$ operates as multiplication by  $1-\left(\tfrac{\ka}{2}+\tfrac{1}{4} \right)u^{-1}(u+\ka)^{-1}$ in $V$. Now $T(u)$ and $\wt T(u)$ are related by 
$T(u)=y(u)^{-1}\wt T(u)$, where $y(u)$ is the central series with constant term uniquely determined by $y(u)\,y(u+\ka)=z(u)$. It follows by definition of $f(u)$ given in \eqref{f(u)} that  $y(u)^{-1}$ must operate as multiplication by $f(u)$ in $V$, and hence that $t_{kl}(u)$ acts as $f(u)\,(\delta_{kl}+F_{kl}u^{-1})$ for all ${-}n\leq k,l\leq n$. 

Now let us show that $\Phi_{J,R}(J(F_{kl}))$ operates as $-\frac{\ka}{2}F_{kl}$ in $V$. From \eqref{f(u)} we can deduce that $f(u)=1+\left(\tfrac{\ka}{4}+\tfrac{1}{8}\right)u^{-2}+ O(u^{-3})$. Therefore the coefficient of $u^{-2}$ in $f(u)(\delta_{kl}+F_{kl}u^{-1})$ is $\left(\tfrac{\ka}{4}+\tfrac{1}{8}\right)\delta_{kl}$, {\it i.e.}~$t_{kl}^{(2)}$ operates as multiplication by $\left(\tfrac{\ka}{4}+\tfrac{1}{8}\right)\delta_{kl}$ in $V$. Combining this with \eqref{spin.2} and \eqref{J->R}, we obtain the equality of operators
\begin{equation*}
 \Phi_{J,R}(J(F_{kl}))= t_{kl}^{(2)}-\mfrac{1}{2}\msum_{a} t_{ka}^{(1)}t_{al}^{(1)}=\left(\mfrac{\ka}{4}+\mfrac{1}{8}\right)\delta_{kl}-\left(\mfrac{\ka}{4}+\mfrac{1}{8}\right)\delta_{kl}-\mfrac{\ka}{2}F_{kl}=-\mfrac{\ka}{2}F_{kl} \; \text{ for all }\; {-}n\leq k,l\leq n. \qedhere
\end{equation*}
\end{proof}

Let $g(u)$ be the unique series in $1+u^{-1}\C[[u^{-1}]]$ such that $g(u)\,g(u+\ka)=\frac{u^2}{u^2-1}$. It was proven in part of the proof of Theorem 3.6 of \cite{AMR} (see (3.20) therein with $c=\ka$) that the natural representation $\C^N=V(\omega_{n-1})$ of $\mfg_N$ extends to a representation of $Y^R(\mfg_N)$ by the assignment 
\begin{equation}
 t_{ij}(u)\mapsto g(u)\left(\delta_{ij}+E_{ij}(u-\ka)^{-1}-\theta_{ij}E_{-j,-i}u^{-1}\right) \; \text{ for all }\; {-}n\leq i,j\leq n. \label{R-nat}
\end{equation}
Our choice of specializing the parameter $c$ to $\ka$ is motivated by the fact that with this choice, one can show that $\Phi_{J,R}(J(F_{ij}))$ operates as $0$ 
in $\C^N$, so  that $\Phi_{\mr{cr},R}^*(\C^N)$ coincides with the extension of the natural representation of $\mfg_N$ to $Y^{\mr{cr}}(\mfg_N)$ constructed in Proposition \ref{P:Y(g)-rep}.

For each $a\in \C$, set $(\C^N)_a={\tau_a^R}^*(\C^N)$. Then, for any $1\leq m\leq n$, we can consider the $Y^R(\mfg_N)$-module $\C^N\ot (\C^N)_{-1}\ot \cdots \ot (\C^N)_{-m+1}$. As in \cite{AMR}, we let $\xi_m$ be the element of this module defined by 
\begin{equation*}
 \xi_m=\msum_{\si\in \mf{S}_m}\mr{sign}(\si)\, e_{-n-1+\si(1)}\ot \cdots \ot e_{-n-1+\sigma(m)}.
\end{equation*}
We let $\C^{N,m}$ denote the submodule of $\C^N\ot (\C^N)_{-1}\ot \cdots \ot (\C^N)_{-m+1}$
generated by $\xi_m$.

The last lemma we will need is also a reformulation of a result from \cite{AMR}, where it played a crucial role in proving the main classification theorem \cite[Theorem 5.16]{AMR}.
\begin{lemma}\label{L:CN,m}
The $Y^R(\mfg_N)$-module $\C^{N,m}$ is a finite-dimensional highest weight module with the highest weight vector $\xi_m$. Its highest weight $\lambda(u)$ is determined by the relations
\begin{equation}\label{R-CN,m}
 g_m(u)^{-1}\lambda_i(u)=\begin{cases}
               \left(\mfrac{u-\ka+m}{u-\ka+m-1} \right) \; &\text{ if }\; {-}n\leq i\leq -n+m-1,\\
               1 \; &\text{ if }\; {-}n+m\leq i\leq n-m,\\
               \left(\mfrac{u-1}{u} \right)\; &\text{ if }\; n-m+1\leq i\leq n,
              \end{cases}
\end{equation}
where $g_m(u)=g(u)\,g(u+1)\cdots g(u+m-1)$. 
\end{lemma}

\begin{proof}
 After taking into consideration the series $g(u)$ in the formula \eqref{R-nat}, this lemma is just a restatement of (5.39) and (5.40) of \cite{AMR} with $u$ replaced by $u-\ka$. 
\end{proof}

\begin{theorem}\label{T:cr-R}
 Suppose that $V$ is a highest weight $Y^R(\mfg_N)$-module with the highest weight vector $\xi$. Then  $\Phi_{\mr{cr},R}^*(V)$ is a highest weight $Y^\mr{cr}(\mfg_N)$-module with highest weight vector $\xi$. Moreover, if $V$ is a finite-dimensional irreducible module corresponding to the Drinfeld tuple $(P_1(u),\ldots,P_n(u))$, then $\Phi_{\mr{cr},R}^*(V)$ is the unique (up to isomorphism) $Y^{\mr{cr}}(\mfg_N)$-module corresponding to the Drinfeld tuple $(Q_0(u),\ldots,Q_{n-1}(u))$, where 
\begin{align}
  Q_i(u)= {} &P_{i+1}\left(u+\mfrac{n+\ka-i}{2}\right) \quad \text{ for all }\; 1\leq i\leq n-1, \label{Q1}\\
  Q_0(u)= {} &\begin{cases}
          P_1(u+\ka)  & \text{ if } \; \mfg_N=\mfso_{2n} \text{ or } \mfg_N=\mfsp_{2n},\\
          P_1\left(u+\ka+\tfrac{n}{2}\right)\qu & \text{ if } \; \mfg_N=\mfso_{2n+1}.
         \end{cases} \label{Q2}
 \end{align}
\end{theorem}

\begin{remark}
It should also be possible to use the isomorphism obtained very recently in \cite{JLM} via the Gauss decomposition to prove this theorem.
\end{remark}

\begin{proof}
\textit{Part I}: The assignment $V\mapsto \Phi_{\mr{cr},R}^*(V)$ sends highest weight vectors (resp.~modules) to highest weight vectors (resp.~modules).

Let $V$ be a highest weight $Y^R(\mfg_N)$-module with the highest weight vector $\xi$. 
 Let us first show that $\Phi_{\mr{cr},R}(h_{ir})\xi\in \C \xi$ for all $i\in \mcI$ and $r\geq 0$. After identifying $t_{kl}^{(1)}$ with $F_{kl}$ for each 
 ${-}n\leq k,l\leq n$, $\Phi_{\mr{cr},R}$ is just the identity when restricted to $\mf{Ug}_N$, and thus  $\Phi_{\mr{cr},R}(h_{i0})\xi\in \C \xi$ for all $i\in \mcI$. To prove the general case, we employ the standard result that, if $\lambda^{(1)}=(\lambda_{i}^{(1)})_{i=1}^n$ denotes the tuple of $F_{ii}$-weights of $\xi$ (for $i\geq 1$), then the weight space $V_{\lambda^{(1)}}$ is one-dimensional. This follows from the Poincar\'{e}-Birkhoff-Witt property proven in \cite[Corollary 3.7]{AMR}, together with the commutation relations $[F_{ii},t_{kl}(v)]$ which can be found in (5.3) of \textit{ibid}. Since each member of $\{\Phi_{\mr{cr},R}(h_{ir})\}_{i\in \mcI,r\geq 0}$ commutes with $F_{kk}\in \mfh_N$, we have $F_{kk}(\Phi_{\mr{cr},R}(h_{ir}) \xi)=\lambda_k^{(1)}(\Phi_{\mr{cr},R}(h_{ir})\xi)$ for all $1\leq k\leq n$, and hence $\Phi_{\mr{cr},R}(h_{ir})\xi\in V_{\lambda^{(1)}}=\C \xi$ for 
 each $i\in \mcI$ and $r\geq 0$. 
 
 It now follows easily that $\Phi_{\mr{cr},R}(x_{ir}^+)\xi=0$ for all $i\in \mcI$ and $r\geq 0$ by induction on $r$. The $r=0$ case is immediate since $\Phi_{\mr{cr},R}(x_{i0}^+)$ is a scalar multiple of an element of the form $t_{kl}^{(1)}=F_{kl}$ with $k<l$. The general case follows from the fact that $x_{i,r+1}^+=(\alpha_i,\alpha_i)^{-1}[\tilde h_{i1},x_{ir}^+]$, the induction assumption, and the result that $\Phi_{\mr{cr},R}(\tilde h_{i1})\xi\in \C \xi$ proven in the previous paragraph. Since we already know $\xi$ generates $\Phi_{\mr{cr},R}^*(V)$, we can conclude that $\Phi_{\mr{cr},R}^*(V)$ is indeed a highest weight module with the highest weight vector $\xi$. 
 
\noindent\textit{Part II}: Proving the correspondences \eqref{Q1} and \eqref{Q2} of Drinfeld polynomials. 
 
 Due to Part I, the fact that $\Phi_{\mr{cr},R}$ is a Hopf-algebra isomorphism, and the multiplicative property of Drinfeld polynomials (see \cite[Lemma 5.17]{AMR} for $Y^R(\mfg_N)$ and \cite[Proposition 12.1.12]{ChPr2} for $Y^{\mr{cr}}(\mfg_N)$), it suffices to show that the correspondence \eqref{Q1} and \eqref{Q2} hold for each fundamental representation $V^R(i;a)$ of $Y^R(\mfg_N)$. In fact, by Lemma \ref{tau}, we only need to prove that \eqref{Q1} and \eqref{Q2} hold when $V=V^R(k;a_k)$ for $1\leq k\leq n$ and 
 $(a_k)_{k=1}^n$ a set of fixed complex numbers. 
 
 Suppose first that $2\leq i\leq n-1$ and set $m=n-i$. Then by Lemma \ref{L:CN,m}, and in particular the formulas \eqref{R-CN,m}, the irreducible quotient $V$ of the highest weight module $\C^{N,m}$ is isomorphic to the fundamental representation $V^R(i+1;1)$. By Part I of the theorem, we already know that $\Phi_{\mr{cr},R}^*(V)$ must be isomorphic to the fundamental module 
 $V(i;a)$ for some $a\in \C$. Indeed, $\Phi_{\mr{cr},R}^*(V)$ is a finite-dimensional irreducible module with the highest weight vector $\bar\xi_m$ equal to the image of $\xi_m$ in $V$, and since $\Phi_{\mr{cr},R}$ is the identity when restricted to $\mf{Ug}_N$, the $\mfg_N$ weight of $\bar\xi_m$ is equal to $\omega_{i}$. On the other hand, it is also equal to the element $\lambda\in \mfh_N^*$ defined by $\lambda(h_{j0})=d_j \deg Q_j(u)$, where $(Q_j(u))_{j\in \mcI}$ denotes the Drinfeld tuple of $\Phi_{\mr{cr},R}^*(V)$
 (see Remark 2 of \cite[Section 12.1.C]{ChPr2}). This implies that $\deg Q_j(u)=\delta_{ij}$, and hence that $\Phi_{\mr{cr},R}^*(V)\cong V(i;a)$ for some $a\in \C$. 
 Thus, to verify \eqref{Q1} and \eqref{Q2} for $V=V^R(i+1;1)$ (with $2\leq i\leq n-1$), it suffices to show that 
 \begin{equation}
 a=1-\mfrac{n+\ka-i}{2}. \label{a}
 \end{equation}
The same argument as used to prove the relation \eqref{J(X)-b} of Corollary \ref{C:J(X)-b} shows that, if $b$ denotes the $\bar \xi_m$ weight of the operator $J(h_{i0})$, 
then $a$ is determined by 
\begin{equation}
 a=d_{i}^{-1}b-\tfrac{1}{2}(\ka-d_{i}). \label{a<->b}
\end{equation}
 Let us compute $b$. Since $g(u)\,g(u+\ka)=u^2(u^2-1)^{-1}$, we have $g(u)=1+\tfrac{1}{2}u^{-2}+O(u^{-3})$, and thus $g_m(u)=1+\tfrac{m}{2}u^{-2}+O(u^{-3})$. Therefore, from \eqref{R-CN,m} we deduce that 
\begin{equation}
 t_{jj}^{(2)}\,\bar\xi_m=\mfrac{m}{2}\,\bar\xi_m \; \text{ for all }\; 0\leq j\leq n. \label{Jact.1}
\end{equation}
Conversely, since $t_{kk}^{(1)}\,\bar\xi_m=\omega_{i}(F_{kk})\,\bar\xi_m=-\bar \xi_m$ if $k\geq i+1$ and $\omega_{i}(F_{kk})\,\bar\xi_m=0$ otherwise, 
\begin{equation}\label{Jact.2}
 \msum_k t_{jk}^{(1)}t_{kj}^{(1)}\,\bar\xi_m=\msum_{k=j}^n t_{jk}^{(1)}t_{kj}^{(1)}\,\bar\xi_m=(n-j+1)\,t_{jj}^{(1)}\,\bar\xi_m-\msum_{k=j}^n t_{kk}^{(1)}\,\bar \xi_m+(t_{jj}^{(1)})^2\,\bar \xi_m=\begin{cases}
        m\bar\xi_m & \text{ if } 0\leq j\leq i ,\\
        \bar \xi_m & \text{ if } i+1\leq j\leq n.                                                                                                                                                                                           \end{cases}
\end{equation}
This allows us to conclude that $\Phi_{J,R}(J(F_{jj}))\,\bar\xi_m=0$ if $0\leq j\leq i$, and $\Phi_{J,R}(J(F_{jj}))=\left(\frac{m-1}{2}\right)\bar\xi_m$ if $i+1\leq j\leq n$. 
Since $h_{i,0}=F_{ii}-F_{i+1,i+1}$ for $2\leq i\leq n-1$, we obtain 
\begin{equation*}
 b\bar \xi_m=\Phi_{J,R}(J(h_{i,0}))\,\bar\xi_m=\left(\tfrac{1-m}{2}\right)\bar\xi_m \qu\implies\qu a=b-\tfrac{1}{2}(\ka-1)=1-\mfrac{n+k-i}{2},
\end{equation*}
where we have used that $d_{i}=1$ for $2\leq i\leq n-1$. This completes the verification of \eqref{a}, and hence the proof of the relations \eqref{Q1} and \eqref{Q2} in the case where 
$V$ is the fundamental module $V^R(i+1;1)$ with $2\leq i\leq n-1$. 

When $i=1$ and $m=n-i=n-1$, the irreducible quotient $V$ of $\C^{N,m}$ is no longer isomorphic to a fundamental module when $\mfg_N=\mfso_{2n}$, but it still is isomorphic to 
$V^R(2;1)$ when $\mfg_N=\mfso_{2n+1}$ or $\mfg_N=\mfsp_{2n}$, and the exact same argument as used above can be repeated. When $i=0$, $m=n$, 
the irreducible quotient $V$ of $\C^{N,m}$ is only a fundamental module when $\mfg_N=\mfsp_{2n}$, where we have  $V=V^R(1;2)$. The same reasoning applies in this case, except now we must show instead that 
$a=2-\ka$, 
and we must also take into account the fact that $d_0=2$ in the expression \eqref{a<->b} and that $h_{00}=-2F_{11}$. The relations \eqref{Jact.1} and \eqref{Jact.2} still hold and together with \eqref{a<->b} yield that 
\begin{equation*}
 b\bar\xi_m=\Phi_{J,R}(J(h_{00}))\bar\xi_m=-2\left(\mfrac{m-1}{2}\right)\bar\xi_m=(1-n)\,\bar\xi_m \qu\implies\qu a=\mfrac{1-n}{2}-\mfrac{\ka-2}{2}=2-\ka.
\end{equation*}
Thus, to complete the proof of the theorem is suffices to show that the relations \eqref{Q1} and \eqref{Q2} hold for $\mfg_N=\mfso_N$ when $V=V^R(1;1/2)$, and for  $\mfg_N=\mfso_{2n}$ when $V=V^R(2;1/2)$. Realizations of these modules are provided by Lemma~\ref{spin}. Consider first the $Y^R(\mfso_{2n})$-module $V=V^R(2;1/2)$. By Lemma \ref{spin}, this module is isomorphic to $V(\omega_1)$ as a $\mfso_{2n}$-module, and it is extended to all of $Y^R(\mfg_N)$ by allowing $\Phi_{J,R}(J(F_{ij}))$ to operate as
$-\tfrac{\ka}{2}F_{ij}$ for all $i,j$. By Corollary \ref{C:J(X)-b}, this means precisely that $\Phi_{\mr{cr},R}^*(V)$ is isomorphic to $V(1;a)$ where $a=-\ka+\tfrac{1}{2}$, and hence 
the formulas \eqref{Q1} and \eqref{Q2} hold for the $Y^R(\mfso_{2n})$-module $V^R(2;1/2)$. 

Lastly, we have to show \eqref{Q1} and \eqref{Q2} are respected when $V$ is the $Y^R(\mfso_N)$-module $V^R(1;1/2)$. Just as above, Lemma \ref{spin} and  Corollary \ref{C:J(X)-b}  
yield that $\Phi_{\mr{cr},R}^*(V)\cong V(0;a)$, where $a$ is given by \eqref{a<->b} with $b=-\tfrac{\ka}{2}$. Hence, the relation \eqref{Q1} is respected.
In the $\mfso_{2n}$ case, we again obtain $a=-\ka+\tfrac{1}{2}$, which agrees with \eqref{Q2}. For $\mfg_N=\mfso_{2n+1}$, $d_0=\tfrac{1}{2}$ and we obtain 
$a=-\frac{3}{2}\ka+\frac{1}{4}=-\ka-\tfrac{n}{2}+\tfrac{1}{2}$, which is also compatible with \eqref{Q2}. 
\end{proof}


\appendix


\section{Proof of Theorem \ref{T:Ycr(g)-} when \texorpdfstring{$\mfg=\mfsl_2$}{}}\label{app}

In this appendix, we complete the proof of Theorem \ref{T:Ycr(g)-} in the case where $\mfg=\mfsl_2$. The extra difficulty when $\mfg= \mfsl_2$ is to check that $\Phi_{\mr{cr},J}$ preserves the relation $[\tilde h_{i1},[x_{i1}^+,x_{i1}^-]]=0$ of $Y_\zeta^\cur(\mfg)$. 

We normalize our symmetric invariant non-degenerate bilinear form $(\cdot, \cdot )$ on $\mfsl_2$ so that it is given by $(X,Y)=\mr{Tr}(XY)$ for all $X,Y\in \mfsl_2$. With this choice, the single positive root $\alpha$ has length $2$ and $\{x_{i0}^+,x_{i0}^-,h\}$ coincides with the standard basis $\{e,f,h\}$ of $\mfsl_2$. We will primarily employ the latter notation throughout the remainder of this proof, but when dealing with equations involving $e$ and $f$ simultaneously, we will sometimes replace them with $x^+$ and $x^-$, respectively. An orthonormal basis with respect to $(\cdot ,\cdot)$ is $\mcB=\{E,F,H\}$, where 
\begin{equation*}
E=\tfrac{1}{\sqrt{2}}(e+f)\quad F=\tfrac{1}{\sqrt{-2}}(e-f),\quad H=\tfrac{1}{\sqrt{2}}h.
\end{equation*}
In particular, we have 
\[
[H,E]=e-f=\sqrt{-2}F,\qu 
[H,F]=-\tfrac{\sqrt{-1}}{2}[h,e-f]=\tfrac{1}{i}(e+f)=-\sqrt{-2}E,\qu
[E,F]=\sqrt{-1} h=\sqrt{-2}H.
\]

\medskip

\noindent \textit{Step 1}: Reducing the defining relations.

\medskip

We begin by showing that the relation  
\begin{equation*}
 [[J(e),J(f)],J(h)]=\zeta^2\left(fJ(e)-J(f)\,e\right) h
\end{equation*}
is satisfied in $Y_\zeta(\mfg)$.  This relation is well known throughout the literature (see for instance (4.64) in \cite{BeCr} and Definition 2.22 in \cite{Mo2}), but to our knowledge the relevant computations have never appeared explicitly, so we have decided to include them.

Consider equation \eqref{J3} with $X_1=E=X_3$ and  $X_2=F=X_4$. Since $[J(E),J(F)]=\sqrt{-1}[J(e),J(f)]$ and $[E,J(F)]=\sqrt{-1}J(h)$, the left hand side is equal to $-2\left[[J(e),J(f)],J(h)\right]$. Thus it suffices to show that the right-hand side of \eqref{J3} 
with $X_1=E=X_3$ and  $X_2=F=X_4$ is equal to $-2\zeta^2\left(fJ(e)-J(f)\,e\right)h$. By definition, it is equal to 
\begin{equation}
 2\zeta^2 \msum_{X_\lambda,X_\mu,X_\nu\in \mcB}\left([E,X_\lambda],\big[[F,X_\mu],[[E,F],X_\nu] \big]\right)\{X_\lambda,X_\mu,J(X_\nu)\}. \label{RHS}
\end{equation}
As $[E,F]=\sqrt{-1}h$ and $(\cdot,\cdot)$ is invariant and symmetric, we have 
\begin{equation*}
 \left([E,X_\lambda],\big[[F,X_\mu],[[E,F],X_\nu] \big]\right)=\sqrt{-1}\left([h,X_\nu],\big[[E,X_\lambda],[F,X_\mu] \big] \right). 
\end{equation*}
Thus, we can rewrite \eqref{RHS} as 
\begin{equation}\label{RHS.1}
 \msum_{X_\lambda,X_\mu\in \mcB}\left([h,E],\big[[E,X_\lambda],[F,X_\mu] \big] \right)\{X_\lambda,X_\mu,J(E)\}
  +\msum_{X_\lambda,X_\mu\in \mcB}\left([h,F],\big[[E,X_\lambda],[F,X_\mu] \big] \right)\{X_\lambda,X_\mu,J(F)\}
\end{equation}
multiplied by  $2\sqrt{-1}\zeta^2$. Consider the different nonzero possibilities for $\big[[E,X_\lambda],[F,X_\mu] \big]$. They are 
\begin{equation*}
\big[[E,H],[F,H]\big],\quad \big[[E,H],[F,E]\big]\qu\text{and}\qu \big[[E,F],[F,H]\big].
\end{equation*}
Both $([H,E],\cdot)$ and $([H,F],\cdot)$ applied to the first of these is zero, while $([H,E],\cdot)$ applied to $[[E,H],[F,E]]$ is zero, and 
$([H,F],\cdot)$ applied to $[[E,F],[F,H]]$ is zero. Therefore, \eqref{RHS.1} reduces to 
\begin{equation*}
 \left([h,E],\left[[E,F],[F,H] \right] \right)\{F,H,J(E)\}+\left([h,F],\left[[E,H],[F,E] \right] \right)\{H,E,J(F)\}.
\end{equation*}
By invariance, $\left([h,F],\left[[E,H],[F,E] \right] \right)=-\left([h,E],\left[[E,F],[F,H] \right] \right)$, and moreover 
\begin{equation*}
 \left([h,E],\left[[E,F],[F,H] \right] \right)=-\sqrt{2}\left(e-f,[h,e+f] \right)=-2\sqrt{2}\left(e-f,e-f\right)= 4\sqrt{2}.
\end{equation*}
Hence, \eqref{RHS.1} multiplied by $2\sqrt{-1}\zeta^2$, which is the right-hand side of \eqref{J3} with $X_1=E=X_3$ and $X_2=F=X_4$, is equal to
$8i\sqrt{2}\zeta^2\left(\{F,H,J(E)\}-\{H,E,J(F)\}\right)$. Let us rewrite this in terms of $e,f$ and $h$: 
\begin{equation*}
 8i\sqrt{2}\zeta^2\left(\{F,H,J(E)\}-\{H,E,J(F)\}\right)=4\zeta^2\left(\{e-f,h,J(e+f)\}-\{h,e+f,J(e-f)\}\right)=8\zeta^2\left(\{e,h,J(f)\}-\{h,f,J(e)\}\right).
\end{equation*}
Expanding $\{h,f,J(e)\}-\{e,h,J(f)\}$ and using the defining relations of $\mfsl_2$ to rewrite it only in terms of $fJ(e)\,h$ and $J(f)\,eh$, we find that it is equal to $\frac{1}{4}\big(J(f)\,eh-fJ(e)\,h\big)$, which gives the desired result.

\medskip

\noindent \textit{Step 2}: Checking the relation $\big[\tilde{h}_1,[x_1^-,x_1^+]\big]=0$ is preserved.

\medskip

By definition of $\Phi_{\mr{cr},J}$, this amounts to checking that 
\begin{equation*}
 [J(h)-\zeta v,[J(f)-\zeta w^-,J(e)-\zeta w^+]]=0 ,
\end{equation*}
where \begin{equation*}
 v=\tfrac{1}{2}(\{e,f\}-h^2),\quad w^+=-\tfrac{1}{4}\{e,h\},\quad w^-=-\tfrac{1}{4}\{f,h\}.
\end{equation*}
Expanding the left-hand side we obtain
\begin{gather}\label{2:big}
\begin{split}
 \left[J(h),[J(f),J(e)]\right]-\zeta\left[J(h),[J(f),w^+]+[w^-,J(e)]\right]+\zeta^2\left[J(h),[w^-,w^+]\right] \\
 -\zeta\left[v,[J(f),J(e)]\right]+\zeta^2\left[v,[J(f),w^+]+[w^-,J(e)]\right]-\zeta^3 \left[v,[w^-,w^+]\right]
 \end{split}
\end{gather}

\medskip

\noindent \textit{Step 2.1}: $-\zeta^3 \left[v,[w^-,w^+]\right]=0$.

\medskip

We have 
\begin{equation}
 [w^+,w^-]=\tfrac{1}{16}([eh,fh]+[eh,hf]+[he,fh]+[he,hf])=\tfrac{1}{16}(4h^3-2\{e,f\}h-2e\{f,h\}-2\{f,h\}e-2h\{e,f\}). \label{w+w-}
\end{equation}
Now, since $[v,h]=0$ and $[\{e,f\},h]=0$, we have that 
$[v,\{e,f\}h+h\{e,f\}]=0$. Hence, 
\begin{equation*}
 [v,[w^+,w^-]]=-\tfrac{1}{8}\,[v,2e\{f,h\}+2\{f,h\}e]=-\tfrac{1}{16}\,[\{e,f\}-h^2,e\{f,h\}+\{f,h\}e].
\end{equation*}
Using the commutator relations of $\mfsl_2$, we can write $e\{f,h\}+\{f,h\}e=4efh-2h-2h^2$.
Note that $[h,efh]=0$. Thus, 
\begin{equation*}
 [v,[w^+,w^-]]=-\tfrac{1}{4}[\{e,f\},efh]=\tfrac{1}{4}([ef,efh]+[fe,efh])=\tfrac{1}{4}([fe,ef]h)=0.
\end{equation*}

\medskip

\noindent \textit{Step 2.2}: The terms in \eqref{2:big} involving $\zeta^2$.

\medskip

After rewriting $\left[J(h),[J(f),J(e)]\right]=\zeta^2(fJ(e)-J(f)e)h$, the sum of the terms involving $\zeta^2$ in \eqref{2:big}
is 
\begin{equation}
 \zeta^2\left((fJ(e)-J(f)e)h+\left[J(h),[w^-,w^+]\right]+\left[v,[J(f),w^+]+[w^-,J(e)]\right]\right) \label{h2}
\end{equation}
Consider first $\left[J(h),[w^-,w^+]\right]$. Recall that \eqref{w+w-} and 
$\{e,f\}=2v+h^2$. Hence, since $J(h)$ commutes with $h$ and $v$, we have 
\begin{align*}
 \left[J(h),[w^-,w^+]\right] = {} & \tfrac{1}{8}\left([J(h),2vh+2hv]+[J(h),\{e,\{f,h\}\}]\right) \\
 = {} & \tfrac{1}{8}\left(4[J(h),v]h+[J(h),\{e,\{f,h\}\}]\right).
\end{align*}
Observe that $e\{f,h\}+\{f,h\}e=4efh-2h-2h^2$ and $[J(h),\{e,\{f,h\}\}]=8(J(e)fh-eJ(f)h)$. 
Therefore, we have
\begin{equation}
 \left[J(h),[w^-,w^+]\right]=\tfrac{1}{2}\left([J(h),v]h+2(J(e)fh-eJ(f)h)\right). \label{Jw-w+}
\end{equation}
Let us now turn to the last term in \eqref{h2}. The identity
\begin{equation*}
 [J(x^\pm),w^{\mp}]=-\tfrac{1}{4}[J(x^\pm),x^\mp h +h x^\mp]=\mp\tfrac{1}{2}(J(h)h-x^\mp J(x^\pm)-J(x^\pm)x^\mp)
\end{equation*}
implies that we can expand $[J(f),w^+]+[w^-,J(e)]$ in the form
\begin{align*}
 [J(f),w^+]+[w^-,J(e)]=\tfrac{1}{2}\left(J(h)h-eJ(f)-J(f)e)+(J(h)h-fJ(e)-J(e)f)\right)=J(h)h-\tfrac{1}{2}(\{e,J(f)\}+\{f,J(e)\}).
\end{align*}
Hence, 
\begin{equation}
 [v,[J(f),w^+]+[w^-,J(e)]]=[v,J(h)]h-\tfrac{1}{2}([v,\{e,J(f)\}+\{f,J(e)\}]). \label{deg:h^2}
\end{equation}
We would like to simplify the last term. Since $[h,x^\pm J(x^\mp)+J(x^\mp)x^\pm]=0$, we can replace $v$ by $\tl v$ ($= \nu + \frac12 h^2 = \frac12\{e,f\}$). Next, the sequence of equalities 
\begin{align*}
[\tilde v,\{x^\pm,J(x^\mp)\}]= {} &[\tilde v, x^\pm]J(x^\mp)+x^\pm[\tilde v,J(x^\mp)]+[\tilde v,J(x^\mp)]x^\pm+J(x^\mp)[\tilde v, x^\pm]\\
                             = {} &\tfrac{1}{2}\left(\mp x^\pm hJ(x^\mp)\mp hx^\pm J(x^\mp)\pm x^\pm x^\mp J(h)\pm x^\pm J(h)x^\mp \pm x^\mp J(h)x^\pm \pm J(h)x^\mp x^\pm \mp J(x^\mp)x^\pm h\mp J(x^\mp)hx^\pm \right).
\end{align*}
allow us to express $2[\tilde v,\{e,J(f)\}+\{f,J(e)\}]$ as
\begin{gather*}
 -e hJ(f)- he J(f)+ efJ(h)+ e J(h)f + f J(h)e + J(h)f e - J(f)e h- J(f)he  \\
 + f hJ(e)+ hf J(e)- f e J(h)- f J(h)e - e J(h)f - J(h)e f + J(e)f h+ J(e)hf
\end{gather*}
The terms involving $J(h)$ all cancel, and we are left with 
\begin{equation*}
 2[\tilde v,\{e,J(f)\}+\{f,J(e)\}]=-e hJ(f)- he J(f)- J(f)e h- J(f)he + f hJ(e)+ hf J(e) + J(e)f h+ J(e)hf.
\end{equation*}
We would like to express the right hand side in terms of $fJ(e)h$ and $J(f)eh$ only. Using the defining relations of $\mfsl_2$, we find that 
\begin{equation*}
 2\left[\tl v,\{e,J(f)\}+\{f,J(e)\}\right]=4\,(fJ(e)\,h-J(f)\,eh).
\end{equation*}
Multiplying by $\frac{1}{4}$ and substituting the result back into \eqref{deg:h^2}, we obtain 
\begin{equation*}
 [v,[J(f),w^+]+[w^-,J(e)]]=[v,J(h)]\,h-(fJ(e)\,h-J(f)\,eh).
\end{equation*}
Substituting this and  \eqref{Jw-w+} back into \eqref{h2} gives 
\begin{align*}
 \zeta^2 & \left((fJ(e)-J(f)e)h+\left[J(h),[w^-,w^+]\right]+\left[v,[J(f),w^+]+[w^-,J(e)]\right]\right)\\
 = {} &\zeta^2\left((fJ(e)-J(f)e)h+\tfrac{1}{2}\left([J(h),v]h+2(J(e)fh-eJ(f)h)\right)+[v,J(h)]h-(fJ(e)h-J(f)eh) \right)\\
 = {} &\tfrac{\zeta^2}{2}\left([v,J(h)]h+2(J(e)fh-eJ(f)h)\right).
\end{align*}
Since $[v,J(h)]\,h=\frac{1}{2}[ef+fe,J(h)]\,h=\left(eJ(f)-J(e)f+J(f)e-fJ(e)\right)h=2\left(eJ(f)-J(e)f\right)h$, the above expression is zero, hence  \eqref{h2} is also zero. 

\medskip

\noindent \textit{Step 2.2}: The terms involving $\zeta$.

\medskip

In order to see that \eqref{2:big} vanishes, it remains to see that 
\begin{equation}
 \left[J(h),[J(f),w^+]+[w^-,J(e)]\right]+\left[v,[J(f),J(e)]\right]=0. \label{deg:h}
\end{equation}
First observe that $[h^2,[J(f),J(e)]]=0$, so we may replace $v$ by $\tilde v$ in \eqref{deg:h}. We then have 
\begin{equation*}
 \left[\tilde v,[J(f),J(e)]\right]=-[J(f),[J(e),\tilde v]]-[J(e),[\tilde v,J(f)]].
\end{equation*}
Since $2w^+=[\tilde v, e]$ and $-2w^-=[\tilde v,f]$, we can rewrite the left-hand side of \eqref{deg:h} as 
\begin{equation}
 -[J(f),[J(e),\tilde v]]-[J(e),[\tilde v,J(f)]]+\tfrac{1}{2}[J(h),[J(f),[\tilde v, e]]]-\tfrac{1}{2}[J(h),[[\tilde v,f],J(e)]]. \label{deg:h.1}
\end{equation}
Consider the last two terms. Since $[J(h),[J(x^\pm),[\tilde v,x^\mp]]]=[J(h),[[J(x^\pm),\tilde v],x^\mp]]\pm [J(h),[\tilde v,J(h)]]$, we have 
\begin{align*}
 [J(h),[J(f),[\tilde v, e]]]&-[J(h),[[\tilde v,f],J(e)]]\\
   = {} &[J(h),[[J(f),\tilde v],e]]-[J(h),[[\tilde v,J(e)],f]]\\
   = {} &[[J(h),[J(f),\tilde v]],e]+2[[J(f),\tilde v],J(e)]-[[J(h),[\tilde v,J(e)]],f]+2[[\tilde v, J(e)],J(f)].                                              
\end{align*}
Substituting these new expression back into \eqref{deg:h.1} we obtain that the left-hand side of \eqref{deg:h} is equal to 
\begin{equation}
 \tfrac{1}{2}\left([[J(h),[J(f),\tilde v]],e]-[[J(h),[\tilde v,J(e)]],f] \right). \label{deg:h.2}
\end{equation}
Let us show that this vanishes. Since $\tilde v=\frac{1}{2}\{e,f\}$, we have 
\begin{equation*}
 [J(x^\pm),\tilde v]=\tfrac{1}{2}[J(x^\pm),x^+f+fx^+]=\pm\tfrac{1}{2}(x^\pm J(h)+J(h)x^\pm).
\end{equation*}
Thus, $[J(h),[J(f),\tilde v]]=J(f)J(h)+J(h)J(f)$ and, similarly, $[J(h),[\tilde v,J(e)]]=-J(e)J(h)-J(h)J(e)$. Substituting these back into \eqref{deg:h.2}, we see that it vanishes, and hence \eqref{deg:h} holds. This completes the proof that $\Phi_{\mr{cr},J}$ preserves the relation $[h_1,[x_1^-,x_1^+]]=0$. \qed


\end{document}